\documentclass{article}
\usepackage{geometry}
\usepackage{tabularx}
\usepackage{titlesec}
\usepackage{ulem}
\usepackage{amsmath, amsthm} 
\usepackage{amssymb} 
\usepackage{amsfonts} 
\usepackage{mathrsfs} 
\usepackage{hyperref} 
\usepackage{setspace}

\newtheorem{theorem}{Theorem}[section]
\newtheorem{mydef}[theorem]{Definition}
\newtheorem{remark}[theorem]{Remark}
\newtheorem{lemma}[theorem]{Lemma}
\newtheorem{prop}[theorem]{Proposition}

\numberwithin{equation}{section}

\makeatletter
\def\blfootnote{\xdef\@thefnmark{}\@footnotetext}
\makeatother

\title{Global well-posedness of the periodic Majda-Biello systems in the resonant case}
\author{Ke Wang and Xin Yang}
\date{}
\begin{document}
	\maketitle
\begin{abstract}
We study the Cauchy problem for the following Majda-Biello system in the resonant case $\alpha \in (0,4]\setminus\{1\}$ on the periodic domain $\mathbb{T}$.
$$
\left\{
\begin{array}{rcl}
	u_{t} + u_{xxx} & = & - v v_x,\\
	v_{t} + \alpha v_{xxx} & = & - (uv)_{x},\\
	(u,v)|_{t=0} & = & (u_0,v_0),
\end{array}
\right. \quad x \in \mathbb{T}, \, t \in \mathbb{R}.
$$
Here, the initial data $u_0$ and $v_0$ live in the Sobolev space $H^{s}(\mathbb{T})$ and $u_0$ has zero mean over $\mathbb{T}$. We establish the optimal global well-posedness for the above system when $s\geq \min\{1, s_\alpha+\}$, where $s_\alpha$ is a constant depending only on $\alpha$, and $s_\alpha = \frac12$ for almost every $\alpha \in (0,4]\setminus\{1\}$. These indices match the sharp ones in the local theory obtained by Oh (Int. Math. Res. Not., 18:3516-3556, 2009). 
The main novelty of our proof is to introduce a third modified energy in the I-method by performing a partial correction to the second modified energy around the resonant region. This partial correction avoids the singularity in the full correction, and therefore makes it possible to pointwise bound the quintilinear multipliers arising in the time derivative of the third modified energy. Another feature in our proof is the combination of the Diophantine approximation theory with the sparsity of nearly resonant frequency pairs on the periodic lattice to obtain the key quadrilinear and quintilinear estimates, leading to an almost conservation law sufficient for the global iteration.
\end{abstract}

\blfootnote{\hspace{-0.25in} 2020 Mathematics Subject Classification. 35Q53; 35G55; 35L56; 35D30.}

\blfootnote{\hspace{-0.25in} Key words and phrases. KdV-KdV systems;   Majda-Biello systems; Global well-posedness; I-Method; Periodic domains}

\begin{center}
	\tableofcontents
\end{center}

\section{Introduction}\label{Sec, intro}

In this paper, we consider the periodic initial value problem for the
Majda--Biello system
\begin{equation}\label{GeneralMajdaBiello}
\left\{
\begin{aligned}
 u_t+u_{xxx}+vv_x&=0,\\
 v_t+\alpha v_{xxx}+(uv)_x&=0,\\
 (u,v)|_{t=0}&=(u_0,v_0)\in H_0^s(\mathbb T)\times H^s(\mathbb T),
\end{aligned}
\right.
\qquad (x,t)\in\mathbb T\times\mathbb R,
\end{equation}
where $\alpha\in\mathbb R\setminus\{0\}$ is the relative dispersion coefficient, 
$\mathbb T:=\mathbb R/(2\pi\mathbb Z)$, 
$H^s(\mathbb{T})$ refers to the standard Sobolev space with $s\in\mathbb{R}$, and
\[
H_0^s(\mathbb T)
:=
\left\{f\in H^s(\mathbb T):\int_0^{2\pi}f(x)\,dx=0\right\}.
\]
Majda and Biello \cite{MajdaBiello} introduced this system as a reduced
asymptotic model for the resonant interaction of equatorial baroclinic
and barotropic Rossby waves in the atmosphere and ocean.  It is also a
basic example of coupled KdV--KdV systems whose two components propagate
according to different dispersion relations.

Since the Majda--Biello system is a coupled KdV system, the
well-posedness theory for the scalar KdV equation provides a natural
point of reference.  Bourgain \cite{Bourgain93} introduced Fourier restriction spaces and
proved global well-posedness in $L^2$ on both $\mathbb R$ and
$\mathbb T$.  Kenig, Ponce, and Vega \cite{KPV96SahrpIndex} subsequently improved the local well-posedness index to $s>-3/4$ on $\mathbb R$ and for
$s\geq-1/2$ on $\mathbb T$.  Colliander, Keel, Staffilani, Takaoka, and
Tao \cite{Iterm03} obtained the corresponding global results using the
$I$-method.  The endpoint cases were obtained by Guo
\cite{GuoZiHua09} and Kishimoto \cite{Kishimoto09}.  By using the complete integrability property of the KdV equation, Kappeler and Topalov
\cite{KappelerTopalov06} justified the global well-posedness in
$H^{-1}(\mathbb T)$, and Killip and Visan
\cite{KillipRowanVicsan19} established the corresponding result on
$\mathbb R$.  Molinet \cite{Molinet2011,Molinet2012} proved the
associated sharp ill-posedness results.  

We next recall the well-posedness results for the Majda--Biello system.
When
\[
\alpha\in(-\infty,0)\cup(4,\infty)\cup\{1\},
\]
there is no additional resonance of the type arising in the coupled
system, and the local theory is analogous to that of the scalar KdV
equation; see \cite{Oh2009}.  The main difficulty occurs in the resonant
range $\alpha\in(0,4]\setminus\{1\}$.  On the real line $\mathbb{R}$, Oh
\cite{Oh2009} proved sharp local and global well-posedness in
$L^2(\mathbb R)\times L^2(\mathbb R)$ for
$\alpha\in(0,4)\setminus\{1\}$.  At $\alpha=4$, Yang
and Zhang \cite{YangZhang} obtained the sharp local threshold $s=3/4$, and Yang \cite{YangRealLine} later proved global well-posedness at this threshold.
On the torus $\mathbb{T}$, the spatial frequencies lie on a lattice, and the Diophantine properties of the roots of the resonance functions play an essential role. Consider the resonance function $R$ defined as follows:
\[
R(x,y,z):=x^3+\alpha y^3+\alpha z^3,
\qquad \forall\, x+y+z=0.
\]
For $\alpha\in(0,4]$, putting $y=-x-z$ into $R(x,y,z)$ yields
\[
R(x,y,z)=-3\alpha x(z+c_1x)(z+c_2x),
\]
where
\begin{equation}\label{c1c2}
	c_1=\frac12-\frac16\sqrt{\frac{12}{\alpha}-3},
	\qquad
	c_2=\frac12+\frac16\sqrt{\frac{12}{\alpha}-3}.
\end{equation}
Hence, the resonant sets are given by
$z = - c_1 x$ and $z = - c_2 x$.
In the periodic setting, lower bounds for the resonance function $R$ depend on the Diophantine approximation properties of $c_1$ and $c_2$. We next recall the Diophantine condition used in the local theory.
\begin{mydef}[Arnol'd {\cite[p.~8]{Arnold1988}}; Oh {\cite[p.~3521]{Oh2009}}]\label{Def,mbti}
A real number $\rho$ is said to be of type $\nu$ if there exists a
constant $K=K(\rho,\nu)>0$ such that
\[
\left|\rho-\frac{p}{q}\right|
\geq \frac{K}{|q|^{2+\nu}}
\]
for all $p,q\in\mathbb Z$ with $q\ne0$.  The minimal type index of
$\rho$ is defined by
\[
\nu(\rho)
:=
\inf\{\nu\in\mathbb R:\rho\text{ is of type }\nu\},
\]
where the infimum is understood to be $\infty$ if the set is empty.
\end{mydef}

\begin{remark}\label{PropOfTypeindex}
It follows directly from the definition that $\nu(\rho)=\infty$ for
$\rho\in\mathbb Q$.  In addition, Dirichlet's approximation theorem gives $\nu(\rho)\geq0$ for every $\rho\in\mathbb R$; see
\cite[p.~112]{Arnold1988}.  Moreover, $\nu(\rho)=0$ for almost every $\rho\in\mathbb R$ with respect to the Lebesgue measure; see \cite[p.~116, Lemma~3]{Arnold1988}.
Further properties can be found in
\cite[Propositions~3.4--3.5]{YZ22DCDS}.
\end{remark}

For $\alpha\in(0,4]\setminus\{1\}$, define
\begin{equation}\label{s0}
\rho_\alpha:=\sqrt{12/\alpha-3},
\qquad
s_\alpha
:=
\frac{1+\nu(\rho_\alpha)}{2}.
\end{equation}
According to Remark \ref{PropOfTypeindex}, $s_\alpha \geq \frac12$ and $s_\alpha = \frac12$ for a.e. $\alpha\in (0,4]\setminus\{1\}$. 
Combining bilinear estimates with Diophantine approximation, Oh
\cite{Oh2009} proved that 
the periodic problem of the Majda-Biello system (\ref{GeneralMajdaBiello}) is locally well-posed for
\[\left\{\begin{array}{lll}
	s \geq 1 & \text{if} & s_\alpha \geq 1, \\
	s > s_\alpha &\text{if} & s_\alpha < 1.
\end{array}\right.\]
The same paper also shows that this threshold is sharp if the solution map is required to be $C^3$.  In particular, for almost every $\alpha\in(0,4]\setminus\{1\}$, $s_\alpha=1/2$ and then (\ref{GeneralMajdaBiello}) is locally well-posed when $s>\frac12$.

At the global level, the following quantities
\begin{equation}\label{mass_con}
\int_{\mathbb T}u\,dx,\qquad
\int_{\mathbb T}v\,dx,\qquad
\int_{\mathbb T}(u^2+v^2)\,dx,
\end{equation}
and the Hamiltonian
\begin{equation}\label{Hamiltonian}
H(u,v):=\int_{\mathbb T}\bigl(u_x^2+\alpha v_x^2-uv^2\bigr)\,dx
\end{equation}
are conserved. Together with the local theory, these quantities give global well-posedness whenever $s\geq 1$.  Thus, the remaining nontrivial case is when $s_\alpha<1$ and $s\in (s_\alpha, 1)$.  For $\alpha=1$, Oh
\cite{OhEJDE} adapted the $I$-method and obtained global
well-posedness on $\mathbb T$ for $s\geq-1/2$.  For
$\alpha\in(0,4]\setminus\{1\}$ with $s_\alpha<1$, Oh
\cite{OhEJDE} used the second modified energy to prove global
well-posedness when
\[
s>
\max\left\{
\frac{6s_\alpha-2s_\alpha^2}{5-s_\alpha},
\frac{2s_\alpha+9}{14}
\right\}.
\]
In particular, this gives $s>5/7$ for almost every $\alpha$, whereas the
local threshold is $s>1/2$.  Oh \cite{OhEJDE} also observed that
the standard full third correction produces an unbounded quintilinear multiplier, which prevents a direct application of the standard KdV modified-energy scheme. The purpose of this paper is to close the gap to establish sharp global well-posedness whose critical indices match the local ones. Our main result is stated below.
\begin{theorem}\label{Theorem}
For $\alpha \in (0,4]\setminus\{1\}$, the Cauchy problem \eqref{GeneralMajdaBiello} for the periodic Majda--Biello system is globally well-posed in $H_0^{s}(\mathbb{T})\times H^{s}(\mathbb{T})$ for
\[\left\{\begin{array}{lll}
	s \geq 1 & \text{if} & s_\alpha \geq 1, \\
	s > s_\alpha &\text{if} & s_\alpha < 1.
\end{array}\right.\]
In particular, (\ref{GeneralMajdaBiello}) is globally well-posed when $s>\frac12$ for almost every $\alpha \in (0,4]\setminus\{1\}$.
\end{theorem}

In Theorem \ref{Theorem}, the case $s_\alpha\geq 1$ is trivial and the sharp global well-posedness when $s\geq 1$ had already been known by combining the local theory \cite{Oh2009} and the conservation of quantities in (\ref{mass_con}) and (\ref{Hamiltonian}). Therefore, we only need to focus on the case when $s_\alpha < 1$ and $s\in(s_\alpha,1)$. Meanwhile, we also know that $s_\alpha\geq \frac12$ based on Remark \ref{PropOfTypeindex}. So for convenience, we introduce the notation
\begin{equation}\label{alpha_set}
\mathcal{A} := \{\alpha\in (0,4]\setminus\{1\}: 1/2 \leq s_{\alpha} < 1 \}
\end{equation}
and will only be concerned with the case $\alpha\in\mathcal A$.

We give an overview of the main ideas used in the proof of Theorem \ref{Theorem}. The general framework is the combination of a modified $I$-method with refined multilinear estimates, so that  local solutions to the periodic Majda--Biello system \eqref{GeneralMajdaBiello} can be extended to arbitrarily long times when $\alpha\in\mathcal{A}$ and $s>s_\alpha$.
We first scale the original system (\ref{GeneralMajdaBiello}) via a parameter $\lambda\geq 1$ so that the existence time of the local solution $(u^\lambda, v^\lambda)$ to the scaled system (\ref{MajdaBielloScaling}) is at least $1$. We then consider standard modified energies as computed in \cite{OhEJDE}. The purpose of these modified energies is to use the oscillation generated by the linear part of the system (\ref{MajdaBielloScaling}) to cancel the nonlinear contribution to the energy increment. As pointed out in \cite{OhEJDE}, the second modified energy $E^{(2)}$ alone is not sufficient to obtain the sharp global well-posedness indices, while the standard third modified energy $\widetilde E^{(3)}$ contains a singularity and one of the quintilinear multipliers in $\frac{d}{dt} \widetilde E^{(3)}(t)$ is unbounded. 
To overcome this difficulty, we examine the increment of the second modified energy $E^{(2)}$ more carefully, and a refined estimate for $\frac{d}{dt} E^{(2)}(t)$ shows that the growth rate of $E^{(2)}$ has sufficient decay in most frequency regions for establishing the sharp global well-posedness, while the insufficiency only occurs in a particular portion $\Omega_*$ of the resonant region:
\begin{equation*}
	\Omega_{*} = \left\{ (n_1, n_2, n_3, n_4) \in \Gamma_4^\lambda : |n_1| \gtrsim N,\, 0 < |n_2| \le 1,\, |n_4 + c_1 n_1| \le \frac{\kappa_\alpha}{\lambda} \text{ or } |n_4 + c_2 n_1| \le \frac{\kappa_\alpha}{\lambda} \right\},
\end{equation*}
where $\Gamma_4^\lambda$ refers to the hyperplane $\sum_{i=1}^{4} n_i = 0$ in the frequency space, $N$ is the parameter in the standard $I$-operator, and $\kappa_\alpha$ is a constant depending only on $\alpha$ and its precise value will be given later; see \eqref{eq:kappa-alpha-choice}. 
We point out that $\Omega_*$ is an asymmetric subregion of the original resonant set and corresponds to a high--low interaction: $n_1,n_3,n_4$ are all at high frequency, whereas $n_2$ lies below frequency $1$.
In our decomposition, all bad regions in the analysis for the quadrilinear multipliers in $\frac{d}{dt} E^{(2)}(t)$ are contained in $\Omega_*$. Based on this observation, we introduce a refined third modified energy $E^{(3)}$ which partially corrects $E^{(2)}$ only on the region $\Omega_*$ to cancel the bad quadrilinear multipliers on $\Omega_*$. Moreover, all quintilinear multipliers in $\frac{d}{dt}E^{(3)}(t)$ caused by the correction term can be bounded pointwise thanks to the restriction of the partial correction to the region  $\Omega_*$. This pointwise boundedness of quintilinear multipliers enables us to control the growth rate of $E^{(3)}(t)$ effectively to establish the desired sharp global well-posedness. 
We mention that partial-correction techniques in the $I$-method have also appeared for other dispersive equations; see, for instance, Colliander--Keel--Staffilani--Takaoka--Tao \cite{CKSTTResonant2008}, Li--Wu--Xu \cite{LWX2011}, Miao--Shao--Wu--Xu \cite{MiaoShaoWu2010}, and Miao--Wu--Xu \cite{MiaoWuXu2011}.

Another technique innovation of this paper is to take advantage of the sparsity of frequencies in resonant sets to provide additional decay for multilinear estimates. For example, on the nearly resonant set
\[
\Delta_4:=
\left\{
(n_1,n_2,n_3,n_4)\in\Gamma_4^\lambda:
\min_{j=1,2}|n_4+c_jn_1|
\leq\frac{\kappa_\alpha}{\lambda}
\right\},
\]
which is denoted by $\Delta$ in \eqref{res_set},
we establish the following quadrilinear estimate in
Lemma~\ref{lem:resonant-quadrilinear}(a):
\[
\int_{\sum_{i=1}^4\tau_i=0}\int_{\Delta_4}
\frac{
\langle H_\alpha\rangle^{1/2}
\prod_{i=1}^4|\widetilde f_i(n_i,\tau_i)|
}{
\langle\tau_1-n_1^3\rangle^{1/2}
\langle\tau_2-n_2^3\rangle^{1/2}
\langle\tau_3-\alpha n_3^3\rangle^{1/2}
\langle\tau_4-\alpha n_4^3\rangle^{1/2}
}
\lesssim
\frac1\lambda
\prod_{i=1}^4\|f_i\|_{L^2(\mathbb T_\lambda\times\mathbb R)},
\]
for $\alpha\in\mathcal A$ and $\lambda\gg1$.
The quartic resonance function $H_\alpha$ is defined by
\[
H_\alpha(n_1,n_2,n_3,n_4)
:=n_1^3+n_2^3+\alpha n_3^3+\alpha n_4^3,
\qquad n_1+n_2+n_3+n_4=0,
\]
as in \eqref{eq:quartic-resonance-H}.
The constants $c_1,c_2$ and $\kappa_\alpha$ are defined in
\eqref{c1c2} and \eqref{eq:kappa-alpha-choice}, respectively.
We point out that, for each fixed $n_1$, each resonance case
in $\Delta_4$ admits at most finite possible choice of $n_4$.
Thanks to this sparsity, we obtain the decay factor
$\lambda^{-1}$, together with the resonance weight
$\langle H_\alpha\rangle^{1/2}$.
The same finiteness property is also used in the proof of
the sharp bilinear estimate for the local theory in \cite[proof of Proposition~3.7]{Oh2009}.
More precisely, consider the nearly resonant set
\[
\Delta_3:=
\left\{
(n_1,n_2,n_3)\in\Gamma_3^\lambda:
\min_{j=1,2}|n_2+c_jn_1|<\frac1\lambda
\right\}.
\]
In fact, the finite frequency summation argument in that proof gives,
in the present notation,
\[
\int_{\sum_{i=1}^3\tau_i=0}\int_{\Delta_3}
\frac{
\langle R\rangle^{1/2}
\prod_{i=1}^3|\widetilde f_i(n_i,\tau_i)|
}{
\langle\tau_1-n_1^3\rangle^{1/2}
\langle\tau_2-\alpha n_2^3\rangle^{1/2}
\langle\tau_3-\alpha n_3^3\rangle^{1/2}
}
\lesssim
\frac1{\lambda^{1/2}}
\prod_{i=1}^3\|f_i\|_{L^2(\mathbb T_\lambda\times\mathbb R)},
\]
where the cubic resonance function $R(n_1,n_2,n_3)
:=n_1^3+\alpha n_2^3+\alpha n_3^3$ is as defined above Definition \ref{Def,mbti}.
For each fixed $n_1$, there are only finitely many admissible
values of $n_2$, uniformly in $\lambda$.
Thus, our quadrilinear estimate is in the same spirit
as the sharp bilinear estimate. This connection suggests
that the sparsity of nearly resonant interactions can also
be used to control the energy increment without raising
the regularity threshold of the local theory.
In our argument, we make use of this sparsity in both
the quadrilinear and quintilinear estimates; see
Sections~\ref{Sec, multi-est} and~\ref{Sec, quint-est}.
Combined with the partial correction in \eqref{E3_def},
these multilinear estimates yield the almost conservation
law in Proposition~\ref{Prop, growth_E3}, which allows us
to establish the sharp global well-posedness stated in
Theorem~\ref{Theorem}.

The rest of the paper is organized as follows.  Section \ref{Sec, pre} introduces the scaled problem, the Fourier restriction spaces, and auxiliary estimates needed in subsequent sections.  In Section \ref{Sec, mod_energy}, we construct the modified energies and the partial correction. Section \ref{Sec, gwp} proves the main theorem based on the almost conserved third modified energy and the bounds for the correction terms on each time slice. Finally, Section \ref{Sec, multi-est} and Section \ref{Sec, quint-est} establish the quadrilinear and quintilinear estimates respectively, which are used to justify the almost conserved third modified energy.

\section{Preliminaries}\label{Sec, pre}
\subsection{Notations}\label{Subsec, notation}
It is well-known that the well-posedness problems of KdV type equations in $H^{s}$ are subcritical if $s>-\frac32$. In such cases, we can perform scaling such that the scaled initial data are small in $H^{s}$, which can be helpful in establishing the well-posedness.
Let $\lambda\geq 1$ and scale the equation \eqref{GeneralMajdaBiello} by $(u,v) \longmapsto (u^{\lambda}, v^{\lambda})$, where
\begin{equation}\label{Scaling_form}
	(u^{\lambda}, v^{\lambda})(x,t)= \frac{1}{\lambda^2} (u, v)\left(\lambda^{-1} x, \lambda^{-3} t \right), \qquad x\in \mathbb{T}_{\lambda} := [0,2\pi\lambda), \, t\in\mathbb{R}.
\end{equation}
Then $u^{\lambda}$ and $v^{\lambda}$ satisfy the following equations:
\begin{equation}\label{MajdaBielloScaling}
	\left\{\begin{aligned}
		&u_t^{\lambda}+u_{xxx}^{\lambda}+v^{\lambda}v_x^{\lambda}= 0 , \\
		&v_t^{\lambda}+ \alpha v_{xxx}^{\lambda}+(u^{\lambda}v^{\lambda})_x=0 , \\
		&(u^{\lambda},v^{\lambda})\mid_{t=0} = (u^{\lambda}_{0}, v^{\lambda}_{0}) \in H_0^s(\mathbb{T}_{\lambda})\times H^s(\mathbb{T}_{\lambda}),
	\end{aligned}\right.
	\quad\quad x\in \mathbb{T}_{\lambda}, \, t\in \mathbb{R},
\end{equation}
where 
\[ 
(u^{\lambda}_{0}, v^{\lambda}_{0})(x) := \lambda^{-2} (u_0, v_0)\left(\lambda^{-1} x \right), \qquad x\in \mathbb{T}_{\lambda}.
\]
For functions on $\mathbb{T}_{\lambda}$, the corresponding frequency space is
$\mathbb{Z}_{\lambda} :=\big\{k \,\big|\, k=n/\lambda\,\,\, \text{for some $n\in \mathbb{Z}$}\big\}$. 
For $1 \leq p < \infty$, we say $f \in \ell^p(\mathbb{Z}_\lambda)$ if 
\begin{equation*}
	\|f\|_{\ell^p(\mathbb{Z}_\lambda)} = \left(\int_{\mathbb{Z}_\lambda} |f(n)|^p ~ {d} n\right)^{1/p} := \left(\frac{1}{2\pi\lambda}\sum_{n \in \mathbb{Z}_\lambda}|f(n)|^p\right)^{1/p} < \infty.
\end{equation*}
For $p=\infty$, we say $f\in \ell^\infty(\mathbb{Z}_\lambda)$ if 
$\|f\|_{\ell^\infty(\mathbb{Z}_\lambda)} := \sup\limits_{n\in\mathbb{Z}_\lambda} |f(n)| < \infty$.

The solutions of the KdV equation are usually studied in Fourier restriction spaces, which were originally introduced in \cite{Bourgain93,KPV96SahrpIndex}. On the real line, one typically works in $X^{s,b}$ with $b>\frac12$. In the periodic setting, however, the relevant estimates require $b=\frac12$, and $X^{s,\frac12}$ does not embed into $C(\mathbb{R}_t;H^s_x)$. We therefore work in $Y^s$, an augmentation of $X^{s,\frac12}$ that embeds continuously into $C(\mathbb{R}_t;H^s_x)$; see, e.g., \cite{Iterm03}.

\begin{mydef}
Let $\lambda\in\mathbb{R}$ with $\lambda\geq 1$. For any function $w$ in the Schwartz space $\mathscr{S}(\mathbb{T}_{\lambda}\times\mathbb{R})$, we define $\widetilde{w}$ to be the space-time Fourier transform of $w$:
\[
	\widetilde{w}(n,\tau) = \int_{ \mathbb{R}}\int_{\mathbb{T}_\lambda} e^{-i(xn + t\tau)} w(x,t) \,dx \,dt, \quad\forall\, n\in \mathbb{Z}_{\lambda}, \, \tau\in\mathbb{R}.
\]
Meanwhile, we define $\widehat{w}(n,t)$ to be the Fourier transform of $w(x,t)$ in the spatial variable:
\[
	\widehat{w}(n,t) = \int_{\mathbb{T}_\lambda} e^{-i xn} w(x,t) \, dx, \quad\forall\, n\in\mathbb{Z}_{\lambda}, \, t\in\mathbb{R}.
\]
\end{mydef}

\begin{mydef}
Let $\alpha, \,s,\,b,\,\lambda\in\mathbb{R}$ with $\alpha\neq 0$ and $\lambda\geq 1$.
\begin{itemize}
	\item[(1)] The Fourier restriction space, $X^{s,b}_{\alpha,\lambda}$, is defined to be the completion of the Schwartz space $\mathscr{S}(\mathbb{T}_{\lambda}\times\mathbb{R})$ with respect to the norm
	\begin{equation}\label{FR norm}
	\|w\|_{X^{s,b}_{\alpha,\lambda}} := \left\| \langle n \rangle^s \langle \tau - \alpha n^3\rangle^{b} \widetilde{w}(n,\tau) \right\|_{\ell^2_n(\mathbb{Z}_\lambda )L^2_\tau(\mathbb{R})},\end{equation}
	where  $\widetilde{w}$ refers to the space-time Fourier transform of $w$.
	
	\item[(2)] The spaces $Y^s_{\alpha,\lambda}$ and $Z^s_{\alpha,\lambda}$ are defined similarly by the following norms, respectively.
	\begin{align}
		\|w\|_{Y^s_{\alpha,\lambda}} &:= \|w\|_{X^{s,\frac{1}{2}}_{\alpha,\lambda}} + \|\langle n\rangle^s \widetilde{w}(n,\tau)\|_{\ell^2_n(\mathbb{Z}_\lambda)L^1_\tau(\mathbb{R})}, \label{work_space} \\
		\|w\|_{Z^s_{\alpha,\lambda}} &:= \|w\|_{X^{s,-\frac{1}{2}}_{\alpha,\lambda}} + \|\langle n\rangle^s\langle\tau-\alpha n^3\rangle^{-1}\widetilde{w}(n,\tau)\|_{\ell^2_n(\mathbb{Z}_\lambda)L^1_\tau(\mathbb{R})}.  \label{auxillaryspace}
	\end{align}
\end{itemize}
\end{mydef}
The $Z^s_{\alpha,\lambda}$ spaces are used to control the
$Y^s_{\alpha,\lambda}$ norm of the Duhamel terms; see, for example,
\eqref{Duhamel_Estimate}. For notation simplicity, we drop the subscript $\alpha$ in the above spaces when $\alpha=1$. For example, $X^{s,b}_{\lambda}$, $Y^{s}_{\lambda}$ and $Z^{s}_{\lambda}$ refer to $X^{s,b}_{1, \lambda}$, $Y^{s}_{1,\lambda}$ and $Z^{s}_{1,\lambda}$ respectively.
On the other hand, for any time interval $J$, we define
$X^{s,b}_{\alpha,\lambda}(\mathbb{T}_{\lambda} \times J)$ to be the restriction of
$X^{s,b}_{\alpha,\lambda}$ to the time interval $J$ with the norm
\begin{equation}
	\|w\|_{X^{s,b}_{\alpha,\lambda}(\mathbb{T}_{\lambda} \times J)} = \inf\left\{ \|W\|_{X^{s,b}_{\alpha,\lambda}} : W|_{t\in J } = w|_{t\in J} \right\}.
\end{equation}
When there is no confusion, we simply denote $X^{s,b}_{\alpha,\lambda}(\mathbb{T}_{\lambda} \times J)$ as $X^{s,b}_{\alpha,\lambda} $.

For $\frac12\leq s < 1$, we fix a smooth even function $\varphi_s\in C^\infty(\mathbb{R})$ that is nonincreasing in $|\xi|$ and satisfies $0\leq\varphi_s\leq1$ and
\begin{equation}\label{varphi}
\varphi_{s}(\xi) = 
\left\{\begin{array}{ccl}
	1 & \text{if} &  |\xi| \leq 1, \\
	|\xi|^{s-1}  & \text{if} & |\xi| \geq 2.
\end{array}\right.
\end{equation}
For any $N\geq 1$, we define the function $m_{N,s}$ to be 
\[
	m_{N,s}(\xi) := \varphi_{s}(N^{-1} \xi) = 
	\left\{\begin{array}{ccl}
		1 & \text{if} &  |\xi| \leq N, \\
		N^{1-s} |\xi|^{s-1}  & \text{if} & |\xi| \geq 2N.
	\end{array}\right.
\]
Then $m_{N,s} $ is a smooth, nonincreasing and even function satisfying $0 < m_{N,s} \leq 1$. More properties of $m_{N,s}$ can be found in \cite{Iterm03} or \cite[Section 3.2]{YangRealLine}.

We now use $m_{N,s}$ with large $N$ to define the $I$-operators which were introduced in \cite{Iterm03}. 
\begin{mydef}
Let $\frac{1}{2} \le s < 1$ and $N \geq 100$. The Fourier multiplier operator $I_{N,s}$ is defined by
\begin{equation}
	\widehat{I_{N,s} f}(n) = m_{N,s}(n) \widehat{f}(n),
\end{equation}
where $f \in L^{1}(\mathbb{T}_{\lambda})$ with $\lambda\geq 1$,  and $\widehat{f}$ refers to the spatial Fourier transform of $f$.
\end{mydef}

It is readily seen that the operator $I_{N,s}$ is a linear bounded map from $H^s(\mathbb{T}_\lambda)$ to $H^1(\mathbb{T}_\lambda)$. Moreover, $I_{N,s}$ can also be regarded as a linear bounded map from $X^{s,b}_{\alpha,\lambda}$ to $X^{1,b}_{\alpha,\lambda}$ for any $\alpha$ and $b$.
Throughout the paper, $N\gg 1$ denotes a sufficiently large parameter,
whose lower bound may depend on the fixed parameters $\alpha$ and $s$,
and on auxiliary positive exponents. Its value will be chosen in the
global iteration argument in Section \ref{Subsec, gwp}.

For notation simplicity, we write $I=I_{N,s}$ and $m=m_{N,s}$ in what follows.
On the other hand, we fix $\eta \in C^\infty_c(\mathbb{R})$ to be a smooth cutoff function supported on $[-2,2]$ with $\eta \equiv 1$ on  $[-1,1]$. 
Furthermore, we write $A\lesssim B$ to mean $A\le C B$, where the implicit constant $C$ is independent of $N$ and $\lambda$ unless otherwise specified.

\subsection{Basic estimates}
\label{Subsec, basic}
In this section, we collect some standard and useful lemmas. 
First, we recall the Airy group $S_{\alpha,\lambda}$ on $\mathbb{T}_{\lambda}$, which is the solution operator for
\[
	w_t+\alpha w_{xxx}=0,\qquad w(0)=w_0,
\]
defined by
\[
	\widehat{ S_{\alpha, \lambda}(t) w_0}(n) := e^{i\alpha n^3 t} \widehat {w_0}(n),
	\qquad n\in \mathbb{Z}_\lambda.
\]
Secondly, throughout this paper, we fix $\eta$ to be any cut-off function in $C^{\infty}(\mathbb{R})$ such that $\eta\equiv 1$ on $[-1,1]$ and $\text{supp } \eta \subseteq [-2,2]$.

\begin{lemma}\label{lemma1}
	For any $s,\alpha,\lambda\in\mathbb{R}$ with $\alpha\neq 0$ and $\lambda\geq 1$, we have
	\begin{equation}\label{First_Linear_Estimate}
		\|\eta(t)S_{\alpha,\lambda}(t)f\|_{Y^{s}_{\alpha,\lambda}} \le C_1\|f\|_{H^s(\mathbb{T}_{\lambda})},
	\end{equation}
	and 
	\begin{equation}\label{Duhamel_Estimate}
		\left\|\eta(t) \int_0^t S_{\alpha,\lambda}(t-t') g(t') \,dt'\right\|_{Y^s_{\alpha,\lambda}} \le C_2 \|g\|_{Z^s_{\alpha,\lambda}},
	\end{equation}
	where the constants $C_1$ and $C_2$ depend only on $\eta$.
\end{lemma}
\begin{proof}
	The proof can be found in Lemma~7.1 and Lemma~7.2 of \cite{Iterm03}.
\end{proof}
\begin{lemma}\label{lemma2} Let $\alpha\in\mathcal{A}$ and $\lambda\geq 1$. Then for any $s\in (s_{\alpha},1)$ and $\delta > 0$, the following bilinear estimates hold:
\begin{align}
	\|I[\partial_x(v^\lambda v^\lambda)]\|_{Z^1_{\lambda}(\mathbb{T}_\lambda \times \mathbb{R})}
	&\le C_3 \lambda^{s_{\alpha}+\delta}
	\|Iv^\lambda\|_{Y^1_{\alpha,\lambda}(\mathbb{T}_\lambda \times \mathbb{R})}^2,  \label{bl-Iop1}\\
	\|I[\partial_x(u^\lambda v^\lambda)]\|_{Z^1_{\alpha,\lambda}(\mathbb{T}_\lambda \times \mathbb{R})}
	&\le  C_3\lambda^{s_{\alpha}+\delta}
	\|Iu^\lambda\|_{Y^1_{\lambda}(\mathbb{T}_\lambda \times \mathbb{R})}
	\|Iv^\lambda\|_{Y^1_{\alpha,\lambda}(\mathbb{T}_\lambda \times \mathbb{R})},  \label{bl-Iop2}
\end{align}
where the function $u^\lambda$ in \eqref{bl-Iop2} is required to have zero mean, and $C_3$ depends on $\alpha$, $s$ and $\delta$, but not on $N$.
Moreover, the above estimates are sharp in the sense that $\delta$ cannot be chosen as $0$.
\end{lemma}
\begin{proof}
 Propositions~3.7 and Proposition~3.8 of \cite{Oh2009} give the following bilinear estimates when $s_{\alpha}<1$:
\begin{equation*}
	\|\partial_x(v^\lambda v^\lambda)\|_{Z^s_{\lambda}(\mathbb{T}_\lambda \times \mathbb{R})}
	\lesssim \lambda^{s_{\alpha}+\delta}
	\|v^\lambda\|_{Y^s_{\alpha,\lambda}(\mathbb{T}_\lambda \times \mathbb{R})}^2,
\end{equation*}
and
\begin{equation*}
		\|\partial_x(u^\lambda v^\lambda)\|_{Z^s_{\alpha,\lambda}(\mathbb{T}_\lambda \times \mathbb{R})}
		\lesssim\lambda^{s_{\alpha}+\delta}
		\|u^\lambda\|_{Y^s_{\lambda}(\mathbb{T}_\lambda \times \mathbb{R})}
		\|v^\lambda\|_{Y^s_{\alpha,\lambda}(\mathbb{T}_\lambda \times \mathbb{R})}.
\end{equation*}
Based on the proof of these bilinear estimates, for the corresponding estimates (\ref{bl-Iop1}) and (\ref{bl-Iop2}) with the $I$-operator, the lemma follows directly from the proof in \cite{Oh2009} as long as the following multiplier bound is justified:
\begin{equation}\label{target_est1}
	\langle n_3\rangle^{1-s} m(n_3) \lesssim \langle n_1\rangle^{1-s} m(n_1) \langle n_2\rangle^{1-s} m(n_2),
\end{equation}
where $n_1+n_2+n_3=0$. Denote
$A(n):=\langle n\rangle^{1-s}m(n)$.
Then it reduces to proving 
\begin{equation}\label{A_ineq1}
	A(n_3) \lesssim A(n_1) A(n_2). 
\end{equation}
Without loss of generality, we may assume that $|n_1| \geq |n_2|$. Then it follows from $n_3 = -(n_1+n_2)$ that $|n_3| \leq 2|n_1|$, which implies that $A(n_3) \lesssim A(n_1)$. Meanwhile, we can check $A(n_2) \gtrsim 1$. Thus, (\ref{A_ineq1}) is verified.
\end{proof}
\begin{lemma}\label{lemma3}
	For any $s,\alpha,\lambda\in\mathbb{R}$ with $\alpha\neq 0$ and $\lambda\geq 1$, there exists a universal constant $C_4$ such that 
	\begin{equation}\label{L6L2_X0_13}
		\|f\|_{L^6_t L^2_x(\mathbb{T}_\lambda \times \mathbb{R})} \le C_4\|f\|_{X^{0,\frac{1}{3}}_{\alpha,\lambda}}, 
	\end{equation}
	and
	\begin{equation}\label{L2L8_X0_38}
		\|f\|_{L^8_t L^2_x(\mathbb{T}_\lambda \times \mathbb{R})}
		\le C_4\|f\|_{X^{0,\frac{3}{8}}_{\alpha,\lambda}}.
	\end{equation}
\end{lemma}
\begin{proof}
	Since
	\[
	\|f\|_{L_t^6 L_x^2}
	= \left( \int_{\mathbb{R}} \left( \int_{\mathbb{T}_\lambda} |f(x,t)|^2 ~ {d} x \right)^3 ~ {d} t \right)^{1/6}
	= \left( \int_{\mathbb{R}} \left( \int_{\mathbb{Z}_\lambda} |\widehat{f}(n,t)|^2 ~ {d} n \right)^3 ~ {d} t \right)^{1/6},
	\]
	by Minkowski inequality and Sobolev embedding $H_t^{\frac13}(\mathbb R)\hookrightarrow L_t^6(\mathbb R)$, we have
	\begin{align*}
		\|f\|_{L_t^6 L_x^2}
		&\leq \left( \int_{\mathbb{Z}_\lambda} \left( \int_{\mathbb{R}} |\widehat{f}(n,t)|^6 ~ {d} t \right)^{1/3} ~ {d} n \right)^{1/2}
		= \left( \int_{\mathbb{Z}_\lambda} \left\| e^{-i\alpha t n^3} \widehat{f}(n,t) \right\|_{L_t^6(\mathbb{R})}^2 ~ {d} n  \right)^{1/2} \\
		&\lesssim \left( \int_{\mathbb{Z}_\lambda} \left\| e^{-i\alpha t n^3} \widehat{f}(n,t) \right\|_{H_t^{1/3}(\mathbb{R})}^2 ~ {d} n \right)^{1/2}
		= \|f\|_{X^{0, \frac{1}{3}}_{\alpha,\lambda}}.
	\end{align*}
	The estimate \eqref{L2L8_X0_38} follows similarly but using the Sobolev embedding
	$H_t^{\frac38}(\mathbb R)\hookrightarrow L_t^8(\mathbb R)$.
\end{proof}
\begin{lemma}\label{lemma4}
	For any $s,\alpha,\lambda\in\mathbb{R}$ with $\alpha\neq 0$ and $\lambda\geq 1$, there exists a constant $C_5$, which only depends on $\alpha$, such that 
	\begin{equation}\label{L4L4_X013}
		\|f\|_{L^4_{x,t}(\mathbb{T}_\lambda \times \mathbb{R})} \le C_5  \|f\|_{X^{0,\frac{1}{3}}_{\alpha,\lambda}}.
	\end{equation} 
\end{lemma}
\begin{proof}
	See \cite[Proposition~7.15]{Bourgain93} or \cite[Lemma~2.4.1]{Oh2007}.
\end{proof}

\begin{lemma}\label{lemma6} 
Let $\alpha\in (0,4]$, $\lambda\geq 1$. Assume for some time $t$, $u^\lambda(\cdot, t)$ has zero mean over $\mathbb{T}_{\lambda}$ and
\begin{equation}\label{bdd_Iscale}
	\|(Iu^\lambda, Iv^\lambda)(t)\|_{L^2(\mathbb{T}_\lambda) \times L^2(\mathbb{T}_{\lambda})} 
	\le 1.
\end{equation}
Then 
\begin{align}
	\|(Iu^\lambda(t),Iv^\lambda(t))\|^2_{H^1(\mathbb{T}_\lambda) \times H^1(\mathbb{T}_\lambda)} &\le C_6\Big(  \|(Iu^\lambda, Iv^\lambda)(t)\|_{L^2(\mathbb{T}_\lambda) \times L^2(\mathbb{T}_\lambda)}^2
	+ H(Iu^\lambda,Iv^\lambda)(t)\Big), \label{constant1}\\
	|H(Iu^\lambda, Iv^\lambda)(t)| &\le C_7 \|(Iu^\lambda, Iv^\lambda)(t)\|_{H^1(\mathbb{T}_\lambda) \times H^1(\mathbb{T}_\lambda)}^2.  \label{constant2}
\end{align}
where $H$ refers to the Hamiltonian in (\ref{Hamiltonian}), $C_6 = \max\big\{2, \frac1\alpha\big\}$ and $C_7 = \max\big\{\frac32, \alpha\big\}$.
\end{lemma}
\begin{proof}
The proof essentially follows from Lemma 5.3 in \cite{Oh2007}, but we provide more details here.
Fix $t$ and write $f=Iu^\lambda(t)$ and $g=Iv^\lambda(t)$. By Fourier inversion and the Cauchy-Schwarz inequality,
\begin{equation}\label{infty_bdd}
	\|f\|_{L^\infty}
	\le \|\widehat f\|_{\ell^1(\mathbb{Z}_\lambda)}
	\le \| \langle n\rangle^{-1} \|_{\ell^2(\mathbb{Z}_\lambda)}
	\big\|\langle n\rangle\widehat f(n)\big\|_{\ell^2(\mathbb{Z}_\lambda)}
	\le \|f\|_{H^1(\mathbb{T}_{\lambda})},
\end{equation}
where we used
\[
	\| \langle n\rangle^{-1} \|_{\ell^2(\mathbb{Z}_\lambda)}^2 
	= \int_{\mathbb{Z}_\lambda}\langle n\rangle^{-2}~ {d} n
	= \frac{1}{2\pi\lambda}\sum_{n \in \mathbb{Z}_\lambda} \langle n \rangle^{-2}
	\le \frac{1}{2\pi\lambda}
	+\frac{1}{\pi\lambda}\int_0^\infty \frac{1}{(1+x/\lambda)^2}~ {d} x
	\le 1 .
\]
Due to (\ref{infty_bdd}), we know
\[
	\left|\int_{\mathbb{T}_\lambda}fg^2~ {d} x\right|
	\le \|f\|_{L^\infty}\|g\|_{L^2}^2
	\le \|f\|_{H^1}\|g\|_{L^2}^2.
\]
Using the assumption (\ref{bdd_Iscale}), we obtain
\begin{equation}\label{cubic_est}
\left|\int_{\mathbb{T}_\lambda}fg^2~ {d} x\right|
\le \frac12\|f\|_{H^1}^2+\frac12\|g\|_{L^2}^4
	\le \frac12\|f\|_{H^1}^2+\frac12\|g\|_{L^2}^2.
\end{equation}
Since
\begin{equation}\label{H_esp}
	H(f,g) = \|f_x\|_{L^2}^2+\alpha\|g_x\|_{L^2}^2 - \int_{\mathbb{T}_\lambda}fg^2~ {d} x,
\end{equation}
the combination of (\ref{H_esp}) with (\ref{cubic_est}) yields
\[
	|H(f,g)|
	\le \|f_x\|_{L^2}^2+\alpha\|g_x\|_{L^2}^2
	+\frac12\|f\|_{H^1}^2+\frac12\|g\|_{L^2}^2 \le \max\left\{\frac32,\alpha\right\}\|(f,g)\|_{H^1\times H^1}^2,
\]
which justifies \eqref{constant2}.

Meanwhile, we combine (\ref{H_esp}) with (\ref{cubic_est}) again to find that
\[
	H(f,g) \geq \|f_x\|_{L^2}^2 + \alpha\|g_x\|_{L^2}^2 - \frac12\|f\|_{H^1}^2 - \frac12\|g\|_{L^2}^2.
\]
This implies that 
\[\begin{split}
	\|(f,g)\|_{L^2\times L^2}^2+H(f,g)  
	\geq \frac12\|f\|_{H^1}^2 + \alpha\|g_x\|_{L^2}^2 + \frac12\|g\|_{L^2}^2 
	\ge \min\left\{\frac12,\alpha\right\}\|(f,g)\|_{H^1\times H^1}^2,
\end{split}\]
which proves \eqref{constant1}. 
\end{proof}

\section{Modified energies}\label{Sec, mod_energy}

\subsection{First and second modified energies}

In this section, we will first follow \cite[Section 5.1]{OhEJDE} to define modified energies $E^{(1)}$, $E^{(2)}$ and $\widetilde E^{(3)}$, and then introduce a new third modified energy $E^{(3)}$ in (\ref{E3_def}) which will be shown to be almost conserved in Proposition \ref{Prop, growth_E3}.

We consider the scaled system (\ref{MajdaBielloScaling}) with $\lambda\geq 1$. The spatial domain for both $u^\lambda$ and $v^\lambda$ is \(\mathbb{T}_\lambda\), and hence the Fourier transforms \(\widehat{u^\lambda}, \widehat{v^\lambda}\) are defined on \(\mathbb{Z}_{\lambda}\).
Let $k$ be any positive integer and define $\Gamma_k^{\lambda} \subset (\mathbb{Z}_{\lambda})^k$ by
\[
\Gamma_k^{\lambda} := \{(n_1, \dots, n_k) \in (\mathbb{Z}_{\lambda})^k : n_1 + \dots + n_k = 0\},
\]
endowed with the push-forward of the normalized product counting measure $~ {d} n_1 \cdots ~ {d} n_{k-1}$ on $(\mathbb{Z}_{\lambda})^{k-1}$ under the map
\[
(n_1, \dots, n_{k-1}) \mapsto \Big( n_1, \dots, n_{k-1}, - \sum_{j=1}^{k-1}  n_j \Big).
\]
In this paper, integrals over $\Gamma_k^\lambda$ are always equipped with this measure. For a collection of indices $\{i_1, i_2, \cdots, i_\ell\}$, we write
\[
	n_{i_1\cdots i_\ell}:= n_{i_1}+\cdots+n_{i_\ell}.
\]
Given a multiplier \(M_k\) and Schwartz functions $f_1, \dots, f_k$, we define the $k$-linear functional $\Lambda_{k}$ by
\[
\Lambda_k(M_k; f_1, \dots, f_k) := \int_{\Gamma^{\lambda}_k} M_k(n_1, \dots, n_k) \widehat{f_1}(n_1) \cdots \widehat{f_k}(n_k)=\frac{1}{(2\pi\lambda)^{k-1}}
\sum_{\Gamma_k^\lambda}
M_k(n_1,\ldots,n_k)
\prod_{j=1}^k \widehat f_j(n_j).
\]

Based on the above preparations, we define the first modified energy by
\[
	E^{(1)}(t) := H \big(Iu^\lambda(t), Iv^\lambda(t) \big) = \int_{\mathbb{T}_{\lambda}} (\partial_x Iu^\lambda)^2 +\alpha (\partial_x Iv^\lambda)^2 -(Iu^\lambda)(Iv^\lambda)^2 ~ {d} x.
\]
According to the Plancherel identity, we find 
\begin{equation}\label{E1_def}\begin{split}
	E^{(1)}(t) &= -\Lambda_2(n_1n_2m(n_1)m(n_2);u^\lambda,u^\lambda) - \alpha \Lambda_2(n_1n_2m(n_1)m(n_2);v^\lambda,v^\lambda) \\
	&
	\quad\,\ - \Lambda_3(m(n_1)m(n_2)m(n_3);u^\lambda,v^\lambda,v^\lambda).
\end{split}\end{equation}
Furthermore, by direct computation, we obtain the following identity (also see \cite{OhEJDE}):
\begin{equation}\label{deri_E1}
\begin{split}
	\frac{dE^{(1)}(t)}{dt} &= \Lambda_3(M_3;u^\lambda,v^\lambda,v^\lambda) + 2i \Lambda_4 \big( n_{23} m(n_1) m(n_{23}) m(n_4);u^\lambda,u^\lambda,v^\lambda,v^\lambda \big) \\
	&\quad + \frac{i}{2} \Lambda_4 \big(n_{14} m(n_{14}) m(n_2) m(n_3);v^\lambda,v^\lambda,v^\lambda,v^\lambda \big),
\end{split}
\end{equation}
where
\begin{equation}\label{M3_def}\begin{split}
	M_3 &= \underbrace{i\left[ n_1^3 m^2 (n_1) + \alpha n_2^3 m^2 (n_2) + \alpha n_3^3 m^2 (n_3)\right]}_{M_{31}}
	- 
	i\left(n_1^3 + \alpha n_2^3 + \alpha n_3^3\right) m(n_1)m(n_2)m(n_3).
\end{split}\end{equation}

Next, we would like to eliminate the trilinear functional $\Lambda_3$ in (\ref{deri_E1}), so we define the second modified energy $E^{(2)}$ by
\begin{equation}\label{E2_def}
	E^{(2)}(t) = E^{(1)}(t) + \Lambda_3(\sigma_3;u^\lambda,v^\lambda,v^\lambda),
\end{equation}
where
\begin{equation}\label{sigma3}
	\sigma_3(n_1,n_2,n_3) 
	:= \frac{i M_3}{{ n_1^3 + \alpha n_2^3 + \alpha n_3^3}}.
\end{equation}
As a result, 
\begin{equation}\label{deri_E2}
	\frac{d}{dt}E^{(2)}(t) = \Lambda_4(M_{41};u^\lambda,u^\lambda,v^\lambda,v^\lambda) + \Lambda_4(M_{42};v^\lambda,v^\lambda,v^\lambda,v^\lambda),
\end{equation}
where
\begin{equation}\label{Multiplier41}
	M_{41} = 2i \, n_{23} \frac{n_1^3 m^2 (n_1) + \alpha n_{23}^3 m^2 (n_{23}) + \alpha n_4^3 m^2 (n_4)}{ n_1^3 + \alpha n_{23}^3 + \alpha n_4^3} ,
\end{equation}
\begin{equation}\label{Multiplier42}
	M_{42} = \frac{i}{2}  \, n_{14} \frac{n_{14}^3 m^2(n_{14}) + \alpha n_{2}^3 m^2 (n_{2}) + \alpha n_3^3 m^2 (n_3)}{ n_{14}^3 + \alpha n_2^3 + \alpha n_3^3} .
\end{equation}
Meanwhile, plugging (\ref{E1_def}), (\ref{sigma3}) and (\ref{M3_def}) into (\ref{E2_def}) yields 
\begin{equation}\label{E2_comp}
	\begin{aligned}
		E^{(2)}(t) &= -\Lambda_2(n_1n_2m(n_1)m(n_2);u^\lambda,u^\lambda) - \alpha \Lambda_2(n_1n_2m(n_1)m(n_2);v^\lambda,v^\lambda) \\
		&\quad\, + \Lambda_3\left(\frac{i M_{31}}{n_1^3 + \alpha n_2^3 + \alpha n_3^3};u^\lambda,v^\lambda,v^\lambda\right),
	\end{aligned}
\end{equation}
where $M_{31}$ is as defined in (\ref{M3_def}).

\begin{remark}
	We point out that the multipliers $\sigma_3$, $M_{41}$ and $M_{42}$ are nonsingular. We explain this claim for $\sigma_3$ in the following. 
	Since $u^\lambda$ has zero mean, the frequency region, where $n_1=0$, makes no contribution to the trilinear functional $\Lambda_3$ in $E^{(2)}$. We may therefore restrict attention to $n_1\neq0$ for the multiplier $\sigma_3$, which guarantees that its denominator 
	$n_1^3+\alpha n_2^3+\alpha n_3^3$ does not vanish when $\alpha\in\mathcal{A}$. In fact, by direct computation, 
	\begin{equation}\label{decomp-0918}
		n_1^3+\alpha n_2^3+\alpha n_3^3 = -3\alpha n_1(n_3+c_1n_1)(n_3+c_2n_1) = -3\alpha n_1^3\left(c_1+\frac{n_3}{n_1}\right)\left(c_2+\frac{n_3}{n_1}\right),
	\end{equation}
	where $c_1$ and $c_2$ are the constants given in (\ref{c1c2}) which can be rewritten as below:
	\[
		c_1=\frac{1}{2}-\frac{1}{6}\rho_\alpha, \qquad
		c_2=\frac{1}{2}+\frac{1}{6}\rho_\alpha, \qquad \rho_\alpha = \sqrt{12/\alpha-3}.
	\]
	Based on the definition of $\nu(\cdot)$ in Definition \ref{Def,mbti}, 
	we have
	\[\nu(\rho_\alpha) = \nu(\sqrt{12/\alpha-3}) = \nu(c_1)=\nu(c_2).\] 
	As a result, for any $\alpha\in\mathcal{A}$, $s_{\alpha}<1$ and $\nu(\rho_\alpha)<1$, which implies that
	$c_1,c_2\in\mathbb{R}\setminus\mathbb{Q}$. 
	Since
	$\frac{n_3}{n_1} \in\mathbb{Q}$, neither $c_1 + \frac{n_3}{n_1}$ nor $c_2 + \frac{n_3}{n_1}$
	can vanish. Hence, it follows from (\ref{decomp-0918}) that
	$n_1^3+\alpha n_2^3+\alpha n_3^3\neq0$ for every
	$(n_1,n_2,n_3)\in\Gamma_3^\lambda$ with $n_1\neq0$. By an analogous argument, both multipliers $M_{41}$ and $M_{42}$ are nonsingular.
\end{remark}

\subsection{Third modified energy}

The second modified energy $E^{(2)}$ may still grow too fast to reach the critical index in the local well-posedness result for the Majda-Biello system. If one only uses $E^{(2)}$, then our computation obtains the global well-posedness
result for
\begin{equation}\label{index_Second_Iter}
	s> \frac{-4s_{\alpha}^2+10s_{\alpha}+1}{11-4s_{\alpha}},
\end{equation}
see (\ref{eq:ss0}) in Remark~\ref{remark:O4}. In particular, for almost every $\alpha$, $s_{\alpha}=\frac12$ and this gives
$s>\frac59$. Although this result improves \cite[Theorem 1.8]{OhEJDE}, it is still away from the sharp local well-posedness
threshold $s>\frac12$. 

A natural idea is to perform one more energy correction. However, as pointed out in \cite[Section 5.1]{OhEJDE}, a full third correction contains a singularity when $\alpha\neq 1$ for the Majda-Biello system. More precisely, if one formally defines
\begin{equation*}
	\widetilde E^{(3)}(t)
	:=
	E^{(2)}(t)
	+\Lambda_4 \big(i M_{41}/H_\alpha ; u^\lambda,u^\lambda,v^\lambda,v^\lambda\big) 
	+\Lambda_4\big(i M_{42}/ H_\alpha' ; v^\lambda,v^\lambda,v^\lambda,v^\lambda\big),
\end{equation*}
where
\begin{equation}\label{eq:quartic-resonance-H}
	H_\alpha(n_1,n_2,n_3,n_4) := n_1^3+n_2^3+\alpha n_3^3+\alpha n_4^3 \quad\text{on}\quad \Gamma_{4}^{\lambda},
\end{equation}
and
\begin{equation}\label{eq:quartic-resonance-Hprime}
	H_\alpha' (n_1,n_2,n_3,n_4) := \alpha(n_1^3+n_2^3+n_3^3+n_4^3) \quad\text{on}\quad \Gamma_{4}^{\lambda},
\end{equation}
then the computation in \cite[Section 5.1]{OhEJDE} gives
\begin{equation}\label{deri_E3_tilde}
	\frac{d}{dt} \widetilde E^{(3)}(t)
	=
	\Lambda_5( M_{51};u^\lambda,u^\lambda,u^\lambda,v^\lambda,v^\lambda) 
	+\Lambda_5( M_{52};u^\lambda,v^\lambda,v^\lambda,v^\lambda,v^\lambda) 
	+\Lambda_5( M_{53};u^\lambda,v^\lambda,v^\lambda,v^\lambda,v^\lambda),
\end{equation}
where
\begin{equation*}
	M_{51}
	=
	4i\,n_{235}n_{35}
	\frac{
		n_1^3m^2(n_1)
		+\alpha n_{235}^3m^2(n_{235})
		+\alpha n_4^3m^2(n_4)
	}{
		n_1^3+\alpha n_{235}^3+\alpha n_4^3
	}\times
	\frac{1}{
		n_1^3+n_2^3+\alpha n_{35}^3+\alpha n_4^3
	}.
\end{equation*}
When $n_1=-n_2$ and $n_{35}=-n_4$, the denominator $n_1^3+n_2^3+\alpha n_{35}^3+\alpha n_4^3$ vanishes while the numerator for $M_{51}$ does not. Hence, $M_{51}$ becomes singular at
certain frequencies. Moreover, there is no cancellation between $\Lambda_5( M_{51};u^\lambda,u^\lambda,u^\lambda,v^\lambda,v^\lambda)$ with the other two quintilinear functionals in (\ref{deri_E3_tilde}).

The novelty of this paper is to introduce a partial correction, rather than a full third correction to avoid the singularity. This idea comes from the work~\cite{MiaoShaoWu2010} and related papers. Recalling the growth rate of $E^{(2)}$ in (\ref{deri_E2}):
\[
	\frac{d}{dt}E^{(2)}(t) = \Lambda_4(M_{41};u^\lambda,u^\lambda,v^\lambda,v^\lambda) + \Lambda_4(M_{42};v^\lambda,v^\lambda,v^\lambda,v^\lambda),
\]
it turns out that the estimate associated with $M_{42}$ and the majority of the estimate associated with $M_{41}$ have already had sufficient decay; see Lemma~\ref{M42Estimate} and Lemma~\ref{M41Estimate} respectively. The only challenging region in the estimate associated with $M_{41}$ is $\Omega_{*}$ defined as 
\begin{equation}\label{eq:bad-region-def}
	\Omega_{*} = \left\{ (n_1, n_2, n_3, n_4) \in \Gamma_4^\lambda : |n_1| \gtrsim N,\, 0 < |n_2| \le 1,\, |n_4 + c_1 n_1| \le \frac{\kappa_\alpha}{\lambda} \text{ or } |n_4 + c_2 n_1| \le \frac{\kappa_\alpha}{\lambda} \right\},
\end{equation}
where $c_1$ and $c_2$ are as defined in (\ref{c1c2}), $\kappa_\alpha$ is fixed as follows and only depends on $\alpha$:
\begin{equation}\label{eq:kappa-alpha-choice}
	\kappa_\alpha
	:=
	\min\left\{
	\frac18,\, \frac{\rho_\alpha}{3},\,
	\frac{1}{16C_*^2 \rho_\alpha}
	\right\},
\end{equation}
where
\begin{equation}\label{eq:C-alpha}
	\begin{aligned}
		C_*:=\max\left\{\frac{2}{|c_1|},\frac{2}{|c_2|},
		|c_1|+2,|c_2|+2\right\}.
	\end{aligned}
\end{equation}
In particular,
$\kappa_\alpha\leq1/8$, so every interval of length
$2\kappa_\alpha/\lambda$ contains at most one point of
$\mathbb Z_\lambda$.

Note that $\Omega_*$ is part of the resonant region which corresponds to a high-low interaction: $n_1$, $n_3$ and $n_4$ are all at high frequency, whereas $n_2$ lies below frequency $1$.
Thus, the only term that needs to be modified is this region for $M_{41}$. Consequently, we define the partially corrected energy by
\begin{equation}\label{E3_def}
	E^{(3)}(t)
	:=
	E^{(2)}(t) +\Lambda_4 \big(\sigma_4^P;u^\lambda,u^\lambda,v^\lambda,v^\lambda\big),
\end{equation}
where the correction multiplier $\sigma_4^P$ is given by $\frac{i M_{41}}{H_\alpha}  \chi_{\Omega_{*}}$, that is, 
\begin{equation}\label{eq:sigma4P}
	\sigma_4^P(n_1,n_2,n_3,n_4)
	:= \frac{
		iM_{41}(n_1,n_2,n_3,n_4)
		\chi_{\Omega_{*}}(n_1,n_2,n_3,n_4)
	}{
		H_\alpha(n_1,n_2,n_3,n_4)
	}, \quad\forall\, (n_1,n_2,n_3,n_4)\in \Gamma_4^{\lambda},
\end{equation}
where $\chi_{\Omega_*}$ denotes the characteristic function of  $\Omega_*$ and $H_\alpha$ is as defined in (\ref{eq:quartic-resonance-H}).

Plugging (\ref{deri_E2}) and (\ref{eq:sigma4P}) into (\ref{E3_def}), we find that 
\begin{equation}\label{deri_E3}
	\begin{aligned}
		\frac{d}{dt}E^{(3)}(t)
		&=
		\Lambda_4\left(
		M_{41}(1-\chi_{\Omega_*});
		u^\lambda,u^\lambda,v^\lambda,v^\lambda
		\right) +
		\Lambda_4\left(
		M_{42};
		v^\lambda,v^\lambda,v^\lambda,v^\lambda
		\right) \\
		&\quad+
		\Lambda_5\left(M^P_{51};v^\lambda,u^\lambda,v^\lambda,v^\lambda,v^\lambda\right) +
		\Lambda_5\left(M^P_{52};u^\lambda,v^\lambda,v^\lambda,v^\lambda,v^\lambda\right) \\
		&\quad+
		\Lambda_5\left(M^P_{53};u^\lambda,u^\lambda,u^\lambda,v^\lambda,v^\lambda\right) +
		\Lambda_5\left(M^P_{54};u^\lambda,u^\lambda,v^\lambda,u^\lambda,v^\lambda\right),
	\end{aligned}
\end{equation}
where 
\begin{equation}\label{M5jMultipliers}
	\begin{aligned}
		&M^P_{51}
		:=
		-\frac{i}{2}\sigma_4^P(n_{15},n_2,n_3,n_4)n_{15}, \qquad
		M^P_{52}
		:=
		-\frac{i}{2}\sigma_4^P(n_1,n_{25},n_3,n_4)n_{25}, \\
		&M^P_{53}
		:=
		-i\sigma_4^P(n_1,n_2,n_{35},n_4)n_{35}, \qquad
		M^P_{54}
		:=
		-i\sigma_4^P(n_1,n_2,n_3,n_{45})n_{45}.
	\end{aligned}
\end{equation}
We point out that all the multipliers in (\ref{M5jMultipliers}) are due to the partial correction term $\Lambda_{4}(\sigma^{P}_{4};u^{\lambda}, u^\lambda, v^\lambda, v^\lambda)$ in (\ref{E3_def}). The significance of this partial correction term is to generate the factor $1-\chi_{\Omega_*}$ for $M_{41}$ in (\ref{deri_E3}). Later in Section \ref{Subsec, nonsin_third_me}, we will justify the pointwise boundedness of the multiplier $\sigma^{P}_{4}$ which guarantees the nonsingularity of the multipliers in (\ref{M5jMultipliers}). 

As a summary, we list the three modified energies that are needed in this paper.
\begin{equation}\label{energies}\begin{split}
	E^{(1)}(t) &= H(I u^{\lambda}(t), I v^{\lambda}(t)), \\
	E^{(2)}(t) &= E^{(1)}(t) + \Lambda_3(\sigma_3;u^\lambda,v^\lambda,v^\lambda), \\
	E^{(3)}(t) &= E^{(2)}(t)  +\Lambda_4(\sigma_4^P;u^\lambda,u^\lambda,v^\lambda,v^\lambda)\\  &= E^{(1)}(t)+ \Lambda_3(\sigma_3;u^\lambda,v^\lambda,v^\lambda)+\Lambda_4(\sigma_4^P;u^\lambda,u^\lambda,v^\lambda,v^\lambda).
\end{split}\end{equation}

\subsection{Nonsingularity of the third modified energy}
\label{Subsec, nonsin_third_me}

In this section, we will establish pointwise boundedness of $\sigma^{P}_{4}$ on $\Omega_*$ by providing a lower bound for its denominator $H_\alpha(n_1, n_2, n_3, n_4)$ on $\Omega_*$.

\begin{lemma}\label{lemma:sigma4}
	Let $\alpha\in\mathcal{A}$, $s\in (s_\alpha, 1)$ and $\lambda\geq 1$. Let the constants $C_*$ and $\kappa_\alpha$ be given as in \eqref{eq:C-alpha} and \eqref{eq:kappa-alpha-choice}, respectively. 
	Let $\sigma_4^P$ be defined as in \eqref{eq:sigma4P}. Then on $\Omega_*$, the	frequencies satisfy
	\begin{equation}\label{eq:Temp4}
		C_*^{-1}|n_1|
		\leq |n_3|, |n_4|, |n_{23}|
		\leq C_*|n_1|,
	\end{equation}
	and the denominator $H_\alpha$ of $\sigma_4^P $ has the following lower bound:
	\begin{equation}\label{eq:Temp3}
		|H_\alpha(n_1, n_2, n_3, n_4)|
		\geq \frac{\alpha|n_3|^2}{4\lambda}.
	\end{equation}
	Moreover, for any $\epsilon > 0$, one has
	\begin{equation}\label{eq:sigma4_est}
		|\sigma_4^P(n_1,n_2,n_3,n_4)|
		\leq C	\lambda^{2+2s_\alpha+\epsilon}
		|n_1|^{2s_\alpha+\epsilon}m^2(n_1)\,
		\chi_{\Omega_*}(n_1,n_2,n_3,n_4),
	\end{equation}
where the constant $C$ only depends on $\alpha$ and $\epsilon$.
\end{lemma}
\begin{proof}
	We first establish \eqref{eq:Temp4}. According to the definition of $\Omega_*$ in (\ref{eq:bad-region-def}), either $|n_4+c_1 n_1|\leq\frac{\kappa_\alpha}{\lambda}$ or 
	$|n_4+c_2 n_1|\leq\frac{\kappa_\alpha}{\lambda}$. Due to symmetry, we only consider the case when
	\[
		|n_4+c_1n_1|\leq\frac{\kappa_\alpha}{\lambda}.
	\]
	Then there exists $\theta\in\mathbb R$, which may depend on both $n_1$ and $n_4$, such that
	\begin{equation}\label{n4n1}
		n_4=-c_1n_1+\theta,
		\qquad
		|\theta|\leq\frac{\kappa_\alpha}{\lambda}.
	\end{equation}
	Using $\sum_{i=1}^4n_i=0$ and $c_1+c_2=1$, we also obtain
	\begin{equation}\label{n23n1}
		n_{23}=-c_2n_1-\theta,
		\qquad i.e., \qquad
		n_3=-c_2n_1-n_2-\theta.
	\end{equation}
	Since $|n_1|\gtrsim N\gg 1$ in $\Omega_*$, for all sufficiently large $N$, it follows from (\ref{n4n1}) that 
	\[
		\frac{|c_1|}{2}|n_1|
		\leq |n_4|
		\leq (|c_1|+1)|n_1|.
	\]
	Meanwhile, since $|n_2|\leq1$, it then follows from (\ref{n23n1}) that
	\begin{equation*}
		\frac{|c_2|}{2}|n_1|
		\leq |n_{23}|,|n_3|
		\leq (|c_2|+1)|n_1|.
	\end{equation*}
	By the definition of
	$C_*$ in \eqref{eq:C-alpha}, these bounds give \eqref{eq:Temp4} on all
	of $\Omega_*$. 

	We next justify \eqref{eq:Temp3} and \eqref{eq:sigma4_est}. By the definition of $\sigma_4^P$ in \eqref{eq:sigma4P},
	\begin{equation}\label{sigma_4P_sec}
		\sigma_4^P(n_1,n_2,n_3,n_4)
		=-2  \, n_{23} \frac{n_1^3 m^2 (n_1) + \alpha n_{23}^3 m^2 (n_{23}) + \alpha n_4^3 m^2 (n_4)}{ n_1^3 + \alpha n_{23}^3 + \alpha n_4^3}
		\times \frac{\chi_{\Omega_*}(n_1,n_2,n_3,n_4)}{H_\alpha(n_1, n_2, n_3, n_4)}.
	\end{equation}
	On $\Omega_*$, the numerator satisfies
	\begin{equation}\label{eq:Temp1}
	|n_{23}|\cdot|n_1^3 m^2 (n_1) + \alpha n_{23}^3 m^2 (n_{23}) + \alpha n_4^3 m^2 (n_4)|\lesssim |n_1|^4\,m^2(n_1).
	\end{equation}
	We turn to provide a lower bound for the first denominator $|n_1^3 + \alpha n_{23}^3 + \alpha n_4^3|$ in (\ref{sigma_4P_sec}). Since $|n_4 + c_1 n_1| \leq \kappa_\alpha/\lambda$, we have 
	\begin{equation}\label{Dio_dual}
		|n_4 + c_2 n_1| \geq |(c_2 - c_1)n_1| - |n_4 + c_1 n_1| 
		\geq \frac13 \rho_\alpha |n_1| - \frac{\kappa_\alpha}{\lambda} \gtrsim |n_1|.
	\end{equation}
	Meanwhile, noticing that
	\begin{equation}\label{D_decom1}
		|n_4 + c_1 n_1| = |n_1| \Big| c_1 + \frac{n_4}{n_1} \Big| = |n_1| \Big| c_1 - \frac{- \lambda n_4}{\lambda n_1}  \Big|,
	\end{equation}
	where both $\lambda n_4$ and $\lambda n_1$ are integers, it follows from Definition \ref{Def,mbti} that 
	\begin{equation}\label{Dio_c1}
		\Big| c_1 - \frac{- \lambda n_4}{\lambda n_1}  \Big| \gtrsim \frac{K}{|\lambda n_1|^{2+\nu(c_1)+\epsilon}},
	\end{equation}
	where $K$ is some positive constant only depending on $c_1$ and $\epsilon$. Since $\nu(c_1) = \nu(\rho_\alpha)$, the combination of (\ref{D_decom1}) and (\ref{Dio_c1}) yields 
	\begin{equation}\label{Dterm_est}
		|n_4 + c_1 n_1| \gtrsim \frac{K |n_1|}{|\lambda n_1|^{2+\nu(\rho_\alpha)+\epsilon}}.
	\end{equation}
	Similar to the decomposition (\ref{decomp-0918}), we have 
	\begin{equation}\label{decom_n23}
		|n_1^3+\alpha n_{23}^3+\alpha n_4^3| = |-3\alpha n_1(n_4+c_1n_1)(n_4+c_2n_1)|.
	\end{equation}
	Plugging (\ref{Dio_dual}) and (\ref{Dterm_est}) into (\ref{decom_n23}) gives
	\begin{equation}\label{eq:Temp2}
		|n_1^3+\alpha n_{23}^3+\alpha n_4^3| \gtrsim \lambda^{-2-\nu(\rho_\alpha)-\epsilon} |n_1|^{1-\nu(\rho_\alpha)-\epsilon}
		= \lambda^{-1-2s_\alpha-\epsilon}|n_1|^{2-2s_\alpha-\epsilon},
	\end{equation}
	where we used the relation $s_\alpha = [1+\nu(\rho_\alpha)]/2$ in (\ref{s0}). Combining (\ref{sigma_4P_sec}), (\ref{eq:Temp1}), and (\ref{eq:Temp2}), we obtain 
	\begin{equation}\label{sigma4P_bdd1}
	\begin{split}
		\big| \sigma^{P}_{4}(n_1,n_2,n_3,n_4) \big| &\lesssim \frac{|n_1|^4 m^2(n_1)}{\lambda^{-1-2s_\alpha-\epsilon}|n_1|^{2-2s_\alpha-\epsilon}} \times \frac{\chi_{\Omega_*}(n_1,n_2,n_3,n_4)}{|H_\alpha(n_1, n_2, n_3, n_4)|} \\
		&= \frac{\lambda^{1+2s_\alpha+\epsilon} |n_1|^{2+2s_\alpha+\epsilon} m^2(n_1)}{|H_\alpha(n_1, n_2, n_3, n_4)|} \chi_{\Omega_*}(n_1,n_2,n_3,n_4).
	\end{split}\end{equation}

	We now derive a lower bound for the denominator $|H_\alpha(n_1, n_2, n_3, n_4)|$ in (\ref{sigma4P_bdd1}) which can be decomposed as
	\begin{equation}\label{decomp_fm1}\begin{split}
		H_\alpha(n_1, n_2, n_3, n_4) &= n_1^3+n_2^3+\alpha n_3^3+\alpha n_4^3 \\
		&= \big(n_1^3+\alpha 	n_{23}^3+\alpha n_4^3\big) + \big( n_2^3+\alpha n_3^3+\alpha (-n_{23})^3 \big).
	\end{split}\end{equation}
	On $\Omega_*$, $0 < |n_2| \leq 1$ and $n_2\in \mathbb{Z}_\lambda$, so  $\lambda^{-1}\leq|n_2|\leq1$. Meanwhile, thanks to \eqref{eq:Temp4}, we know $|n_3| \gtrsim |n_1| \gtrsim N \gg |n_2|$. Hence,
	\begin{equation}\label{0920-1}
		|n_2^3+\alpha n_3^3+\alpha (-n_{23})^3| = 3\alpha |n_2||n_{3}+c_1n_2||n_3+c_2n_2|
		\geq 3\alpha |n_2| \Big| \frac12 n_3 \Big|^2\geq \frac{3\alpha}{4\lambda} |n_3|^2. 
	\end{equation}
	Previously, we obtained a lower bound for $|n_1^3+\alpha n_{23}^3+\alpha n_4^3| $ in (\ref{eq:Temp2}). Meanwhile, an upper bound for this term is also available. In fact, since 
	\[
		|n_4 + c_2 n_1| \leq |n_4 + c_1 n_1| + |c_2 - c_1| |n_1| \leq 	\frac{\kappa_\alpha}{\lambda} + \frac{1}{3} \rho_\alpha |n_1|,
	\]
	we obtain 
	\begin{equation}\label{eq:R-resonant-upper}
		\begin{aligned}
			|n_1^3+\alpha n_{23}^3+\alpha n_4^3| &= |-3\alpha n_1(n_4+c_1n_1)(n_4+c_2n_1)|\\
			&\le 3\alpha |n_1| \frac{\kappa_\alpha}{\lambda}\Big(\frac{\kappa_\alpha}{\lambda} + \frac13 \rho_\alpha |n_1|\Big) \\
			&\leq 3\alpha |n_1|\frac{\kappa_\alpha^2}{\lambda^2} + \alpha \rho_\alpha \kappa_\alpha \frac{|n_1|^2}{\lambda}
			\leq 2\alpha\rho_\alpha \kappa_\alpha \frac{|n_1|^2}{\lambda}.
		\end{aligned}
	\end{equation}
	Plugging (\ref{0920-1}) and (\ref{eq:R-resonant-upper}) into (\ref{decomp_fm1}) yields 
	\begin{equation}\label{0920-2}
		|H_\alpha(n_1, n_2, n_3, n_4)| \geq \frac{3\alpha}{4\lambda} |n_3|^2 - 2\alpha\rho_\alpha \kappa_\alpha \frac{|n_1|^2}{\lambda}.
	\end{equation}
	Based on the definition of $\kappa_\alpha$ in (\ref{eq:kappa-alpha-choice}) and the established estimate (\ref{eq:Temp4}), we have
	\begin{equation}\label{0920-3}
		2\alpha\rho_\alpha \kappa_\alpha \frac{|n_1|^2}{\lambda} \leq 
		2\alpha\rho_\alpha \frac{1}{16 C_*^2 \rho_\alpha} \frac{|C_* n_3|^2}{\lambda} \leq \frac{\alpha n_3^2}{8\lambda}.
	\end{equation}
	The combination of (\ref{0920-3}) and (\ref{0920-2}) yields the desired estimate (\ref{eq:Temp3}).
	
	Finally, putting (\ref{eq:Temp3}) into (\ref{sigma4P_bdd1})
	gives 
	\begin{equation}\label{0920-4}
		\big| \sigma^{P}_{4}(n_1,n_2,n_3,n_4) \big|  \lesssim
		\frac{\lambda^{2+2s_\alpha+\epsilon}
			|n_1|^{2+2s_\alpha+\epsilon}m^2(n_1)}{\alpha n_3^2} \chi_{\Omega_*}(n_1,n_2,n_3,n_4),
	\end{equation}
	which implies  (\ref{eq:sigma4_est}) since $|n_3| \geq C_*^{-1} |n_1|$ thanks to (\ref{eq:Temp4}).
\end{proof}

\section{Global well-posedness of \eqref{GeneralMajdaBiello} on $\mathbb{T}$}
\label{Sec, gwp}

\subsection{Rescaling and a variant local well-posedness result}

In Section \ref{Sec, pre}, we scaled the original system (\ref{GeneralMajdaBiello}) to obtain a scaled system (\ref{MajdaBielloScaling}). For an arbitrary time $T$, the original IVP \eqref{GeneralMajdaBiello} is well-posed on $[0,T]$ if and only if the scaled system (\ref{MajdaBielloScaling}) is well-posed on $[0,\lambda^3 T]$. 
By the local well-posedness result in \cite{Oh2009}, the $I$-operator preserves the local well-posedness; more precisely, we have the following lemma.
\begin{lemma}\label{LWPresult}
	Let $\alpha\in\mathcal{A}$ and $\lambda \geq 1$. For any $s \in (s_{\alpha}, 1)$ and $\delta > 0$, there exists a constant $\varepsilon_*>0$, depending on $\alpha$, $s$ and $\delta$, such that the lifespan of the local solution $(u^\lambda, v^\lambda)$ to the system \eqref{MajdaBielloScaling} is at least $1$ if its initial data $(u_0^\lambda, v_0^\lambda)\in H_0^s(\mathbb{T}_\lambda) \times H^s(\mathbb{T}_\lambda)$ satisfies 
	\begin{equation}\label{LWPcondition}
		\lambda^{s_{\alpha}+\delta}\,\big\| (I u_0^\lambda, I v_0^\lambda) \big\|_{H_0^1(\mathbb{T}_\lambda) \times H^1(\mathbb{T}_\lambda)} \le \varepsilon_*.
	\end{equation}	
	Moreover, $(u^\lambda,v^\lambda)$ belongs to $ Y^s_{\lambda}(\mathbb{T}_\lambda\times [0,1])\times Y^s_{\alpha,\lambda}(\mathbb{T}_\lambda\times [0,1])$  and satisfies  
\begin{equation}\label{ContractionBall}
	\sup_{t\in[0,1]}\|(Iu^\lambda,Iv^\lambda)(t)\|_{H_0^1 \times H^1}\le \|(Iu^\lambda,Iv^\lambda)\|_{Y^1_{\lambda}\times Y^1_{\alpha,\lambda}}\le 2C_1  \big\| (I u_0^\lambda, I v_0^\lambda) \big\|_{H_0^1 \times H^1},
\end{equation}
where $C_1$ is the constant from \eqref{First_Linear_Estimate} in Lemma~\ref{lemma1}.
\end{lemma}
\begin{proof}
Firstly, according to the Duhamel principle and the notations in Section \ref{Subsec, basic}, the solution to the system (\ref{MajdaBielloScaling}) for $t\in[0,1]$ can be regarded as the 
restriction of the solution to the following integral equations on the time interval $[0,1]$ due to the cut-off function $\eta$.
\begin{equation}\label{Duh_soln_T}\left\{ 
\begin{aligned} 
	u^\lambda(t) &= \eta(t) \Big( S_{1,\lambda}(t) u_{0}^\lambda - \frac12\int_{0}^{t} S_{1,\lambda}(t - \tau) \partial_x (v^\lambda v^\lambda)(\tau) \,d\tau \Big),\\ 
	v^\lambda(t) &= \eta(t) \Big(S_{\alpha,\lambda}(t) v_{0}^\lambda - \int_{0}^{t} S_{\alpha,\lambda}(t - \tau) \partial_x(u^\lambda v^\lambda)(\tau) \,d\tau\Big).
\end{aligned}
\right.\end{equation}
By applying the $I$-operator to both equations respectively, it reduces to solving 
\begin{equation}\label{integ form}\left\{ 
\begin{aligned} 
I u^\lambda(t) &= \eta(t) \Big( S_{1,\lambda}(t) I u_{0}^\lambda - \frac12\int_{0}^{t} S_{1,\lambda}(t - \tau) I \partial_x (v^\lambda v^\lambda)(\tau) \,d\tau \Big),\\ 
I v^\lambda(t) &= \eta(t) \Big(S_{\alpha,\lambda}(t) I v_{0}^\lambda - \int_{0}^{t} S_{\alpha,\lambda}(t - \tau) I \partial_x(u^\lambda v^\lambda)(\tau) \,d\tau\Big). 
\end{aligned}
\right.\end{equation}
Define 
\[f^\lambda(t) = I u^\lambda(t), \quad f_0^\lambda = I u_0^\lambda, \quad g^\lambda(t) = I v^\lambda(t), \quad g_0^\lambda = I v_0^\lambda.\]
By the properties of $m_{N,s}$ described above, the operator $I$ has a bounded inverse operator which maps $H^{1}(\mathbb{T}_\lambda)$ to $H^{s}(\mathbb{T}_\lambda)$. Hence, (\ref{integ form}) is equivalent to 
\begin{equation}\label{integ form_I}\left\{ 
\begin{aligned} 
	f^\lambda(t) &= \eta(t) \Big( S_{1,\lambda}(t) I u_{0}^\lambda - \frac12\int_{0}^{t} S_{1,\lambda}(t - \tau) I \partial_x (v^\lambda v^\lambda)(\tau) \,d\tau \Big),\\ 
	g^\lambda(t) &= \eta(t) \Big(S_{\alpha,\lambda}(t) I v_{0}^\lambda - \int_{0}^{t} S_{\alpha,\lambda}(t - \tau) I \partial_x(u^\lambda v^\lambda)(\tau) \,d\tau\Big),
\end{aligned}
\right.\end{equation}
where $u^\lambda = I^{-1} f^{\lambda}$ and $v^\lambda = I^{-1} g^\lambda$.
This suggests considering the map 
\[\Phi(f^\lambda, g^\lambda) \triangleq \big(\Phi_{1}(f^\lambda,g^\lambda), \Phi_{2}(f^\lambda, g^\lambda)\big),\] 
where 
\begin{equation}\label{contra map}
\left\{\begin{aligned}
	\Phi_{1}(f^\lambda, g^\lambda) &:= \eta(t) \Big( S_{1,\lambda}(t) f_{0}^\lambda - \frac12\int_{0}^{t} S_{1,\lambda}(t - \tau) I \partial_x (v^\lambda v^\lambda)(\tau) \,d\tau \Big),\\
	\Phi_{2}(f^\lambda, g^\lambda) &:= \eta(t) \Big(S_{\alpha,\lambda}(t) g_{0}^\lambda - \int_{0}^{t} S_{\alpha,\lambda}(t - \tau) I \partial_x [(u^\lambda v^\lambda)](\tau) \,d\tau\Big).
\end{aligned}\right. \end{equation}
We aim to show $\Phi $ is a contraction mapping in a ball in an appropriate Banach space, which will imply that the fixed point of $\Phi$ is the desired solution to the system (\ref{MajdaBielloScaling}) for $0\leq t\leq 1$.

Recalling that the mean of $u(\cdot, t)$ over $\mathbb{T}_\lambda$ is conserved in time, so the initial condition $u_0^\lambda\in H_0^s(\mathbb{T}_\lambda)$ implies that $u^\lambda(\cdot, t)\in H_0^s(\mathbb{T}_\lambda)$ for all time. In addition, based on the relation $(f^\lambda, g^\lambda) = (Iu^\lambda, Iv^\lambda)$, we know that the mean of $f^\lambda(\cdot, t)$ over $\mathbb{T}_\lambda$ also vanishes for all time since 
\[
	\widehat{f^\lambda}(0,t) = \widehat{Iu^\lambda}(0,t) = \widehat{u^\lambda}(0,t) = 0.
\]
Based on this observation, we define the space for $(f^\lambda, g^\lambda)$ to be $\mathcal{Y}_\lambda$:
\[
	\mathcal{Y}_\lambda = \big\{ (f^\lambda, g^\lambda)\in Y^1_\lambda\times Y^1_{\alpha,\lambda}:\,  \widehat{f^\lambda}(0,t) = 0, \,\forall\, t \big\}.
\]
Any element $(f^\lambda, g^\lambda)$ in $\mathcal{Y}_\lambda$ is equipped with the norm 
\[
	\|(f^\lambda, g^\lambda)\|_{\mathcal{Y}_\lambda} := \|f^\lambda\|_{Y^1_\lambda} + \|g^\lambda\|_{Y^1_{\alpha,\lambda}}.
\]
For convenience, we denote 
\[
r_0 = \|(Iu^\lambda_0,Iv^\lambda_0)\|_{H_0^1(\mathbb{T}_\lambda)\times H^1(\mathbb{T}_\lambda)}
\]
and then consider the ball
\begin{equation}\label{soln_space_mt}
B(2C_1 r_0) = \left\{(f^\lambda, g^\lambda) \in \mathcal{Y}_{\lambda}: \|(f^\lambda, g^\lambda)\|_{Y^1_{\lambda}\times Y^{1}_{\alpha,\lambda}} \leq 2C_1 r_0 \right\}.
\end{equation}
We will prove $\Phi$ is a contraction mapping on $B(2C_1 r_0)$ by choosing a suitable $\epsilon_{*}$ in (\ref{LWPcondition}) and a constant $C_1$ in (\ref{soln_space_mt}).
	
Fix any $(f^\lambda, g^\lambda)\in B(2C_1 r_0)\subset \mathcal{Y}_\lambda$. By applying Lemma \ref{lemma1} to (\ref{contra map}), we have
	\begin{equation}\label{Tempcon}
		\|\Phi(f^\lambda, g^\lambda)\|_{Y^{1}_{\lambda}\times Y^{1}_{\alpha,\lambda}} \le C_1 r_0 + C_2 \left(\|I\partial_x(v^{\lambda} v^{\lambda})\|_{Z^1_{\lambda}(\mathbb{T}_\lambda \times \mathbb{R})}+\|I\partial_x(u^\lambda v^\lambda)\|_{Z^1_{\alpha,\lambda}(\mathbb{T}_\lambda \times \mathbb{R})}\right),
	\end{equation}
	where $u^\lambda = I^{-1} f^{\lambda}$, $v^\lambda = I^{-1} g^\lambda$, and the constants $C_1$ and $C_2$ come from Lemma \ref{lemma1}. Then, using Lemma \ref{lemma2}, we deduce from (\ref{Tempcon}) that
	\[
	\|\Phi(f^\lambda, g^\lambda)\|_{Y^{1}_{\lambda}\times Y^{1}_{\alpha,\lambda}} \le C_1 r_0 + C_2 C_3 \lambda^{s_{\alpha}+\delta}\|(f^\lambda, g^\lambda)\|_{Y^1_{\lambda}\times Y^{1}_{\alpha,\lambda}}^2.
	\]
	Since $(f^{\lambda}, g^\lambda)\in B(2C_1 r_0)$, the above estimate implies that 
	\begin{equation*}
		\|\Phi(f^{\lambda}, g^{\lambda})\|_{Y^{1}_{\lambda}\times Y^{1}_{\alpha,\lambda}}\le C_1 r_0 + C_2 C_3 \lambda^{s_{\alpha}+\delta}(2C_1 r_0)^2 = C_1 r_0 + 4C_1^2C_2C_3\lambda^{s_{\alpha}+\delta} r_0^2.
	\end{equation*}
	Thanks to the smallness assumption (\ref{LWPcondition}), we know 
	\begin{equation}\label{Phi_bdd1}
			\|\Phi(f^{\lambda}, g^{\lambda})\|_{Y^{1}_{\lambda}\times Y^{1}_{\alpha,\lambda}}\le C_1 r_0 + 4C_1^2 C_2 C_3 r_0 \epsilon_*.
	\end{equation}
	By choosing
	\begin{equation}\label{small_sigr}
		\epsilon_* = \frac{1}{16C_1C_2C_3},
	\end{equation}
	then \eqref{Phi_bdd1} implies
	\begin{equation}\label{norm_bdd_init}
		\|\Phi(f^{\lambda}, g^{\lambda})\|_{Y^{1}_{\lambda}\times Y^{1}_{\alpha,\lambda}}\le 2C_1 r_0,
	\end{equation}
	which justifies $\Phi\big( B(2C_1 r_0) \big) \subseteq B(2C_1 r_0)$.
	
	Similarly, for any $(f_i, g_i)\in B(2C_1 r_0)$, $i=1,2$, the same argument gives
	\[
	\begin{aligned}
		\|\Phi(f_1, g_1) - \Phi(f_2, g_2)\|_{Y^1_{\lambda}\times Y^{1}_{\alpha,\lambda}}&\leq 8C_1 C_2 C_3\lambda^{s_{\alpha}+\delta} r_0 \|(f_1, g_1) - (f_2, g_2)\|_{Y^1_{\lambda}\times Y^{1}_{\alpha,\lambda}}\\
		&\le 8C_1 C_2 C_3 \epsilon_* \|(f_1, g_1) - (f_2, g_2)\|_{Y^1_{\lambda}\times Y^{1}_{\alpha,\lambda}} \\
		&\le \frac{1}{2}\|(f_1, g_1) - (f_2, g_2)\|_{Y^1_{\lambda}\times Y^{1}_{\alpha,\lambda}},
	\end{aligned}
	\]
	where we have again used \eqref{small_sigr}.
	
	In summary, by choosing $C_1$ and $\epsilon_*$ to be the constants in (\ref{First_Linear_Estimate}) and \eqref{small_sigr} respectively, $\Phi$ is a contraction mapping on $B(2C_1 r_0)$ and its fixed point, denoted as $(f^\lambda, g^\lambda)$, is the desired solution to the system (\ref{MajdaBielloScaling}) for $0\leq t\leq 1$. Moreover, by recalling $u^\lambda = I^{-1} f^{\lambda}$ and $v^\lambda = I^{-1} g^\lambda$, \eqref{norm_bdd_init} implies
	 \[
	 	\|(Iu^\lambda,Iv^\lambda)\|_{Y^1_{\lambda}\times Y^1_{\alpha,\lambda}}\le 2C_1  \big\| (I u_0^\lambda, I v_0^\lambda) \big\|_{H_0^1(\mathbb{T}_\lambda) \times H^1(\mathbb{T}_\lambda)},
	 \]
	 which proves \eqref{ContractionBall}.
\end{proof}

For fixed initial data $(u_0,v_0)\in H_0^s(\mathbb{T})\times H^s(\mathbb{T})$, the smallness condition \eqref{LWPcondition} should be satisfied in order to apply Lemma \ref{LWPresult}. A direct computation leads to the following estimate which will be used to guarantee (\ref{LWPcondition}).
\begin{lemma}\label{Iinequality}
	Let $\frac12\leq s<1$, $\lambda\geq1$, and $N\gg1$. Then, for
	every $(u_0,v_0)\in H_0^s(\mathbb{T})\times H^s(\mathbb{T})$, the estimate
	\begin{equation}\label{Iscale_est}
	\|(Iu_0^\lambda, Iv_0^\lambda)\|_{H_0^1(\mathbb{T}_\lambda)\times H^1(\mathbb{T}_\lambda)}
	\lesssim \lambda^{-\frac{3}{2}-s} N^{1-s} \|(u_0,v_0)\|_{H_0^s(\mathbb{T})\times H^s(\mathbb{T})}
	\end{equation}
	holds if and only if $\lambda^s \lesssim N^{1-s}$.
\end{lemma}

\begin{proof}
	It suffices to prove the estimate for the $u$-component:
	\begin{equation}\label{Iscale_est_u}
		\| Iu_0^\lambda \|_{H_0^1(\mathbb{T}_\lambda)}
		\lesssim \lambda^{-\frac{3}{2}-s} N^{1-s} \|u_0\|_{H_0^s(\mathbb{T})} \quad \text{if and only if}\quad  \lambda^s \lesssim N^{1-s},
	\end{equation}
	since the $v$-component can be treated in the same way.
	By the definition of the scaling and of the $I$-operator, the inequality for $u_0$ is equivalent to
	\[
	\big\| \langle n\rangle m( n)\, \lambda^{-1}\widehat{u}_0(\lambda n)\,  \big\|_{\ell^2(\mathbb{Z}_\lambda)}
	\lesssim \lambda^{-\frac{3}{2}-s} N^{1-s} \big\| \langle n\rangle^s \widehat{u}_0(n) \big\|_{\ell^2(\mathbb{Z})},
	\]
	which is
	\[
		\frac{1}{\lambda} \sum_{n\in \mathbb{Z}_{\lambda}} (1+|n|)^2 m^2(n) \lambda^{-2} |\widehat{u}_0(\lambda n)|^2 
		\lesssim \lambda^{-3-2s} N^{2-2s} \sum_{n\in\mathbb{Z}} (1+|n|)^{2s} |\widehat{u}_0(n)|^2,
	\]
	i.e.
	\begin{equation}\label{Iscale_est_u1}
	\lambda^{-3} \sum_{n\in \mathbb{Z}} \Big(1+\Big|\frac{n}{\lambda}\Big|\Big)^2 m^2\Big(\frac{n}{\lambda}\Big) |\widehat{u}_0(n)|^2 
	\lesssim \lambda^{-3-2s} N^{2-2s} \sum_{n\in\mathbb{Z}} (1+|n|)^{2s} |\widehat{u}_0(n)|^2.
	\end{equation}
	After comparing the coefficient before $|\widehat{u}_0(n)|^2$, it suffices to establish the following multiplier bound:
	\begin{equation}\label{Temp0004}
		\lambda^{2s}\Big(1+\Big|\frac{n}{\lambda}\Big|\Big)^2 m^2\Big(\frac{n}{\lambda}\Big)
		\lesssim N^{2-2s}(1+|n|)^{2s} \qquad \text{for all } n\in\mathbb{Z}.
	\end{equation}
	
	\noindent\textbf{Sufficiency.} Assume $\lambda^s \lesssim N^{1-s}$. We verify \eqref{Temp0004} by considering three frequency ranges.
	
	\begin{itemize}
		\item {\bf Case 1: $|n| \ge 2\lambda N$.}
	Here the multiplier satisfies $m^2\big(\frac{n}{\lambda}\big) = N^{2-2s}\big|\frac{n}{\lambda}\big|^{2s-2}$. Hence
	\[
	\lambda^{2s}\Big(1+\Big|\frac{n}{\lambda}\Big|\Big)^2 m^2\Big(\frac{n}{\lambda}\Big)
	\lesssim |n|^{2s} N^{2-2s}
	\lesssim N^{2-2s}(1+|n|)^{2s}.
	\]
	
		\item {\bf Case 2: $\lambda \le |n| \le 2\lambda N$.}
	In this region $m\big(\frac{n}{\lambda}\big) \le m(1) \lesssim 1$ and $1+\frac{|n|}{\lambda} \le 2\frac{|n|}{\lambda}$. Therefore,
	\[
	\lambda^{2s}\Big(1+\Big|\frac{n}{\lambda}\Big|\Big)^2 m^2\Big(\frac{n}{\lambda}\Big)
	\le \lambda^{2s}\Big(\frac{2|n|}{\lambda}\Big)^2 m(1)^2
	\lesssim 4\lambda^{2s-2}|n|^2 .
	\]
	Using the assumption $\lambda \le |n| \le 2\lambda N$ and $\frac12\le s<1$, we obtain
	\[
	4\lambda^{2s-2}|n|^2
	\le 4\lambda^{2s-2} |n|^{2-2s}|n|^{2s}
	\le 4 \lambda^{2s-2}(2\lambda N)^{2-2s}(1+|n|)^{2s}\le 8N^{2-2s}(1+|n|)^{2s}.
	\]
	
		\item {\bf Case 3: $|n| \le \lambda$.} 
	Then $m\big(\frac{n}{\lambda}\big) = 1$ and $1+\frac{|n|}{\lambda} \le 2$. Consequently,
	\[
	\lambda^{2s}\Big(1+\Big|\frac{n}{\lambda}\Big|\Big)^2 m^2\Big(\frac{n}{\lambda}\Big)
	\le 4\lambda^{2s}
	\le 4N^{2-2s}
	\le 4N^{2-2s}(1+|n|)^{2s},
	\]
	where we used the assumption that $\lambda^{s}\lesssim N^{1-s}$ in the middle inequality.
	\end{itemize}
	
	Combining the three cases justifies (\ref{Temp0004}) for all $n$, which implies the desired inequality (\ref{Iscale_est_u}).
	
	\noindent\textbf{Necessity.} Conversely, suppose the estimate of the lemma holds. Choose the particular initial data
	\[
	u_0(x) = e^{ix}.
	\]
	Then 
	$\widehat{u}_0(n) = \delta(n-1)$, where $\delta$ refers to the Dirac delta.
	Substituting this into (\ref{Iscale_est_u1}) yields
	\[
	\Big(1+\frac{1}{\lambda}\Big)^2 m^2\Big(\frac{1}{\lambda}\Big)
	\lesssim \lambda^{-2s} N^{2-2s} 2^{2s}.
	\]
	Thus $\lambda^{2s} \lesssim N^{2-2s}$, i.e.\ $\lambda^s \lesssim N^{1-s}$.
\end{proof}

\subsection{Bounds for the modified-energy corrections at each time}
\begin{lemma}\label{Sigma3Estimate}
	Let $\alpha\in\mathcal{A}$, $s\in (s_\alpha, 1)$
	and  
	\begin{equation}\label{eq:ConditionOfLamandN2}
		1\ll \lambda \lesssim N^{\frac{2(1-s)}{3}}.
	\end{equation}
	Let $\sigma_3$ be as defined in \eqref{sigma3}. Then for every fixed time $t$,
	\begin{equation}\label{constant4}
		 | \Lambda_3(\sigma_3; u^\lambda, v^\lambda, v^\lambda)(t) | \le C_8 \|Iu^\lambda(t)\|_{H^1} \|Iv^\lambda(t)\|_{H^1}^2,
	\end{equation}
where $u^{\lambda}(\cdot, t)$ has zero mean on $\mathbb{T}_\lambda$ and
$C_8$ depends only on $\alpha$ and $s$.
\end{lemma}

\begin{proof}
	We refer the readers to \cite[Lemma~5.9]{OhEJDE} for the proof. 
	In fact, Oh \cite[Lemma~5.9]{OhEJDE} assumes the relation
		\[
		\lambda\sim N^{\frac{2(1-s)}{3+2s-2s_\alpha-\epsilon}}, 
		\]
		between $\lambda$ and $N$ for any sufficiently small positive $\epsilon$. However, an inspection of the proof shows that the weaker condition \eqref{eq:ConditionOfLamandN2} suffices to justify (\ref{constant4}).
\end{proof}

Similar to Lemma \ref{Sigma3Estimate}, we will also estimate the quadrilinear functional $|\Lambda_4(\sigma_4^P;u^\lambda,u^\lambda,v^\lambda,v^\lambda)(t)|$ for any fixed time $t$. Before carrying out this estimate, we first present one technical result which will be used to bound the quadrilinear functional in Lemma \ref{lem:fixed-time-bound}.

\begin{lemma}\label{lem:admissible-frequency-sum}
	Let $A>0$ and $a\in\mathbb{R}\setminus\{0\}$. For any $n\in\mathbb{Z}_\lambda$, define 
	\begin{equation}\label{Theta_def}
		\Theta_{A}(n;a):=
		\left\{\theta\in\mathbb{R}:
		|\theta|\leq\frac{A}{\lambda}
		\ \text{and}\
		an+\theta\in\mathbb{Z}_\lambda\right\}.
	\end{equation}
	Then for any $\lambda\geq 1$, $f\in\ell^2(\mathbb{Z}_\lambda)$ and $m\in\mathbb{Z}_\lambda$, we have 
	\begin{equation}\label{eq:finite-sum}
		\bigg\|
		\sum_{\theta\in\Theta_{A}(n;a)}
		\big|f(an+m+\theta)\big|
		\bigg\|_{\ell_n^2(\mathbb{Z}_\lambda)}
		\leq C(A,a)\|f\|_{\ell^2(\mathbb{Z}_\lambda)},
	\end{equation}
	where
	\[
		C(A,a):=
		\left[
		(2A+1)\left(\frac{2A}{|a|}+1\right)
		\right]^{1/2}.
	\]
	In particular, $C(A,a)$ is independent of $\lambda$, $f$ and $m$.
\end{lemma}

\begin{proof}
	For each fixed $n\in\mathbb{Z}_\lambda$, the frequencies $an+\theta$ with
	$\theta\in\Theta_A(n;a)$ lie in an interval of length $2A/\lambda$.
	Since the spacing of $\mathbb{Z}_\lambda$ is $1/\lambda$, it follows that
	\begin{equation}\label{eq:theta-cardinality}
		\sup_{n\in\mathbb{Z}_\lambda}\#\Theta_{A}(n;a)
		\leq 2A+1.
	\end{equation}
	As a result, we apply the Cauchy-Schwarz inequality and (\ref{eq:theta-cardinality}) to find
	\[\begin{split}
		\bigg\|
		\sum_{\theta\in\Theta_{A}(n;a)}
		\big|f(an+m+\theta)\big|
		\bigg\|_{\ell_n^2(\mathbb{Z}_\lambda)}^2 
		& = \frac{1}{2\pi\lambda} \sum_{n\in \mathbb{Z}_{\lambda}} \bigg( \sum_{\theta\in\Theta_{A}(n;a)}
		\big|f(an+m+\theta)\big| \bigg)^2 \\
		&\leq \frac{1}{2\pi\lambda} \sum_{n\in \mathbb{Z}_{\lambda}} (2A+1) \sum_{\theta\in\Theta_{A}(n;a)}
		\big|f(an+m+\theta)\big|^2 \\
		& = \frac{2A+1}{2\pi\lambda} \sum_{n\in \mathbb{Z}_{\lambda}} \sum_{\theta\in\Theta_{A}(n;a)}
		\big|f(an+m+\theta)\big|^2.
	\end{split}\]
	So (\ref{eq:finite-sum}) is reduced to the justification of the following inequality:
	\[
	\sum_{n\in \mathbb{Z}_{\lambda}} \sum_{\theta\in\Theta_{A}(n;a)}
	\big|f(an+m+\theta)\big|^2 
	\leq 
	\left(\frac{2A}{|a|}+1\right)  \sum_{n'\in \mathbb{Z}_{\lambda}} |f(n')|^2,
	\]
	which is further reduced to showing that for any $n'\in\mathbb{Z}_\lambda$, 
	\[
	\#\big\{(n,\theta): n\in\mathbb{Z}_{\lambda},\, \theta\in\Theta_{A}(n;a),\, an+m+\theta = n'
	\big\} \leq \frac{2A}{|a|} + 1.
	\]
	Noticing that once $n$ is fixed, then $\theta$ is uniquely determined by $\theta = n'-an-m$, so it is equivalent to showing that for any $n'\in\mathbb{Z}_\lambda$,
	\begin{equation}\label{eq:theta-multiplicity}
		\#\left\{n\in\mathbb{Z}_\lambda:
		n'=an+m+\theta\ \text{for some }\theta\in\Theta_{A}(n;a)
		\right\}
		\leq \frac{2A}{|a|}+1.
	\end{equation}
	In fact, for any fixed $n'\in\mathbb{Z}_\lambda$, the relation
	\[
	n'=an+m+\theta,
	\qquad \theta\in\Theta_{A}(n;a),
	\]
	implies
	\[
	\left|n-\frac{n'-m}{a}\right| = \Big| - \frac{\theta}{a}\Big|
	\leq\frac{A}{|a|\lambda}.
	\]
	The possible values of $n$ therefore lie in an interval of length
	$2A/(|a|\lambda)$. Since $n\in\mathbb{Z}_\lambda$, the number of possible values of $n$ is at most 
	$\frac{2A}{|a|}+1$, which verifies (\ref{eq:theta-multiplicity}).
\end{proof}

Now we are ready to show the pointwise bound for 
$|\Lambda_4(\sigma_4^P;u^\lambda,u^\lambda,v^\lambda,v^\lambda)(t)|$ at each fixed time $t$.

\begin{lemma}\label{lem:fixed-time-bound}
	Let $\alpha\in\mathcal{A}$, $s\in (s_\alpha, 1)$ and 
	\begin{equation}\label{eq:Lemma1ConditionOfLamandN}
		1\ll \lambda\lesssim N^{1-s}.
	\end{equation}
	Let $\sigma_4^P$ be as defined in \eqref{eq:sigma4P}. 
	Then for any fixed time $t$, we have
	\begin{equation}\label{eq:fixed-time-bound}
		\left|
		\Lambda_4(\sigma_4^P;u^\lambda,u^\lambda,v^\lambda,v^\lambda)(t)
		\right|
		\le C_9
		\|Iu^\lambda(t)\|_{H^1}^2 \|Iv^\lambda(t)\|_{H^1}^2,
	\end{equation}
	where the constant $C_9$ only depends on $\alpha$ and $s$.
\end{lemma}

\begin{proof}
	It suffices to prove the following $L^2$ multiplier
	bound: for arbitrary functions $f_i\in L^2(\mathbb{T}_\lambda)$,
	\begin{equation}\label{eq:fixed-time-reduced}
		\left|
		\Lambda_4\left(
		\frac{\sigma_4^P(n_1, n_2, n_3, n_4)}{\prod_{i=1}^4\langle n_i\rangle m(n_i)};
		f_1,f_2,f_3,f_4
		\right)
		\right|
		\lesssim
		\prod_{i=1}^4\|f_i\|_{L_x^2}.
	\end{equation}
	Indeed, applying \eqref{eq:fixed-time-reduced} with
	$f_1=f_2=\langle\nabla_x\rangle Iu^\lambda(t)$ and $f_3=f_4=\langle\nabla_x\rangle Iv^\lambda(t)$ gives
	\eqref{eq:fixed-time-bound}.
	By Lemma~\ref{lemma:sigma4} with the fixed $\epsilon = 2s-2s_\alpha$, we have 
\begin{equation}\label{0921-2}
	\left|
	\frac{\sigma_4^P(n_1, n_2, n_3, n_4)}{\prod_{i=1}^4\langle n_i\rangle m(n_i)}
	\right|
	\lesssim
	\frac{\lambda^{2+2s} |n_1|^{2s}m^{2}(n_1)}
	{\prod_{i=1}^4\langle n_i\rangle m(n_i)} \chi_{\Omega_*}(n_1, n_2, n_3, n_4).
\end{equation}
	On $\Omega_*$, \eqref{eq:Temp4} gives
	$|n_1|\sim|n_3|\sim|n_4|\gtrsim N$, and hence
	$m(n_1)\sim m(n_3)\sim m(n_4)$. Moreover, $|n_2|\leq1$
	implies $m(n_2)=1$. Therefore, (\ref{0921-2}) implies that
	\begin{equation}\label{0921-1}
	\left|
	\frac{\sigma_4^P(n_1, n_2, n_3, n_4)}{\prod_{i=1}^4\langle n_i\rangle m(n_i)}
	\right|
\lesssim
	\frac{\lambda^{2+2s} |n_1|^{2s}\chi_{\Omega_*}}{ |n_1|^3 \langle n_2\rangle m(n_1)}
\lesssim
	\frac{\lambda^{2+2s} \chi_{\Omega_*}}{N^{1-s}|n_1|^{2-s}\langle n_2\rangle},
	\end{equation}
	where the second inequality is due to the relation
	$m(n_1) \sim N^{1-s} |n_1|^{s-1}$. 
	Then it follows from (\ref{0921-1}) and $|n_1| \gtrsim N$ that
	\begin{equation}\label{0922-1}
		\left|
		\frac{\sigma_4^P(n_1, n_2, n_3, n_4)}{\prod_{i=1}^4\langle n_i\rangle m(n_i)}
		\right| \lesssim 
		\frac{ \lambda^{2+2s} \chi_{\Omega_*}}{N^{3-2s}\langle n_2\rangle}.
	\end{equation}
Meanwhile, on $\Omega_*$, either $|n_4 + c_1 n_1|\leq\kappa_\alpha/\lambda$ or $|n_4 + c_2 n_1|\leq\kappa_\alpha/\lambda$. We only consider the first case since the other case is analogous. Thus, in what follows, the left hand side of (\ref{eq:fixed-time-reduced}) refers to the contribution from the first case. Then it follows from (\ref{0922-1}) that 
\begin{equation}\label{0923-1}
	\begin{aligned}
		\text{LHS of } \eqref{eq:fixed-time-reduced}
		&\lesssim \frac{\lambda^{2+2s}}{N^{3-2s}}
		\Big(\frac{1}{2\pi \lambda}\Big)^3
		\sum_{n_1\in \mathbb{Z}_\lambda} |\widehat{f}_1(n_1)|
		\Big(\sum_{n_2\in \mathbb{Z}_\lambda} \langle n_2\rangle^{-1} |\widehat{f}_2(n_2)| \\
		&\qquad
		\sum_{n_4\in \mathbb{Z}_\lambda: |n_4 + c_1 n_1|\leq \kappa_\alpha/\lambda} |\widehat{f}_4(n_4)| |\widehat{f}_3(-n_1-n_2-n_4)| \Big).
	\end{aligned}
\end{equation}

For any fixed $n_1$, we perform a change of variable $n_4\to \theta$ by
\[
	\theta =: n_4 + c_1 n_1.
\]
Then for any $n_4\in\mathbb{Z}_\lambda$ such that $|n_4 + c_1 n_1|\leq \kappa_\alpha/\lambda$, we know $|\theta|\leq \kappa_\alpha/\lambda$ and $-c_1 n_1 + \theta = n_4\in\mathbb{Z}_\lambda$. Hence, $\theta\in \Theta_{\kappa_\alpha}(n_1; -c_1)$ which is defined as in (\ref{Theta_def}) with $A = \kappa_\alpha$ and $a=-c_1$, that is,
\begin{equation}\label{Thetan1}
	\Theta_{\kappa_\alpha}(n_1; -c_1) =
	\left\{\theta\in\mathbb R:\ |\theta|\leq\frac{\kappa_\alpha}{\lambda}
	\ \text{ and }\ -c_1 n_1+\theta\in\mathbb Z_\lambda\right\}.
\end{equation}
So (\ref{0923-1}) is converted to 
\begin{equation}\label{0923-2}\begin{split}
	\text{LHS of } \eqref{eq:fixed-time-reduced} &\lesssim \frac{\lambda^{2+2s}}{N^{3-2s}}
	\Big(\frac{1}{2\pi \lambda}\Big)^3
	\sum_{n_1\in \mathbb{Z}_\lambda} |\widehat{f}_1(n_1)|
	\Big(\sum_{n_2\in \mathbb{Z}_\lambda} \langle n_2\rangle^{-1} |\widehat{f}_2(n_2)| \\
	&\qquad
	\sum_{\theta\in\Theta_{\kappa_\alpha}(n_1;-c_1)} |\widehat{f}_4(-c_1 n_1 + \theta)| |\widehat{f}_3(-n_1-n_2+c_1 n_1-\theta)| \Big).
\end{split}\end{equation}
For notation convenience, we define $F_i$ to be the function such that $\widehat{F_i}=|\widehat{f_i}|$. Then we interchange the order of the summation for $n_1$ and $n_2$ in (\ref{0923-2}) to obtain 
\begin{equation}\label{0923-3}\begin{split}
	\text{LHS of } \eqref{eq:fixed-time-reduced} &\lesssim \frac{\lambda^{2+2s}}{N^{3-2s}}
	\Big(\frac{1}{2\pi \lambda}\Big)^3
	\sum_{n_2\in \mathbb{Z}_\lambda} \langle n_2\rangle^{-1} \widehat{F}_2(n_2)
	\Big(\sum_{n_1\in \mathbb{Z}_\lambda} \widehat{F}_1(n_1)
	\\
	&\qquad
	\sum_{\theta\in\Theta_{\kappa_\alpha}(n_1;-c_1)} \widehat{F}_4(-c_1 n_1 + \theta) \widehat{F}_3(-n_1-n_2+c_1 n_1-\theta) \Big).
\end{split}\end{equation}

In order to bound the right hand side of (\ref{0923-3}), we first  factor out the $\ell^{\infty}$ norm of $\widehat{F}_3$ and then the $\ell^{1}$ norm of $\langle n_2\rangle^{-1}\widehat{F}_2(n_2)$ to find
\[
	\text{LHS of } \eqref{eq:fixed-time-reduced}
	\lesssim 
	\frac{\lambda^{2+2s}}{N^{3-2s}}
	\Big(\frac{1}{2\pi \lambda}\Big)^2 \|\widehat{F_3}\|_{\ell^\infty}\|\langle n_2\rangle^{-1}\widehat{F_2}(n_2)\|_{\ell^1_{n_2}} 
	\sum_{n_1\in \mathbb{Z}_\lambda} \widehat{F}_1(n_1)
	\sum_{\theta\in\Theta_{\kappa_\alpha}(n_1;-c_1)} \widehat{F}_4(-c_1 n_1 + \theta).
\]
We continue to perform the Cauchy-Schwarz inequality in $n_1$ to obtain 
\[
	\text{LHS of } \eqref{eq:fixed-time-reduced}
	\lesssim 
	\frac{\lambda^{2+2s}}{N^{3-2s}}
	\frac{1}{2\pi \lambda} \|\widehat{F_3}\|_{\ell^\infty}\|\langle n_2\rangle^{-1}\widehat{F_2}(n_2)\|_{\ell^1_{n_2}} 
	\|\widehat{F}_1(n_1)\|_{\ell^2_{n_1}}
	\Big\|
	\sum_{\theta\in\Theta_{\kappa_\alpha}(n_1;-c_1)} \widehat{F}_4(-c_1 n_1 + \theta)\Big\|_{\ell^2_{n_1}}.
\]
Now we apply Lemma \ref{lem:admissible-frequency-sum} (with $A=\kappa_\alpha$, $a=-c_1$ and $m=0$) to the last term  above to get 
\begin{equation}\label{0923-4}
	\text{LHS of } \eqref{eq:fixed-time-reduced}
	\lesssim 
	\frac{\lambda^{2+2s}}{N^{3-2s}}
	\frac{1}{2\pi \lambda} \|\widehat{F_3}\|_{\ell^\infty}\|\langle n_2\rangle^{-1}\widehat{F_2}(n_2)\|_{\ell^1_{n_2}} 
	\|\widehat{F}_1\|_{\ell^2}
	\|\widehat{F}_4\|_{\ell^2}.
\end{equation}

Noticing 
\[
\|\widehat{F_3}\|_{\ell^\infty(\mathbb{Z}_\lambda)}\le \|F_3\|_{L_x^1(\mathbb{T}_\lambda)}\le \lambda^{1/2}\|F_3\|_{L_x^2(\mathbb{T}_\lambda)}
= \lambda^{1/2}\|\widehat F_3\|_{\ell^2(\mathbb{Z}_\lambda)},
\]
and
\[
\|\langle n_2\rangle^{-1}\widehat{F_2}\|_{\ell^1_{n_2}(\mathbb{Z}_\lambda)}\le \|\widehat{F_2} \|_{\ell^2_{n_2}(\mathbb{Z}_\lambda)},
\]	
then it follows from (\ref{0923-4}) that
\[
	\text{LHS of } \eqref{eq:fixed-time-reduced}
	\lesssim 
	\frac{\lambda^{\frac32+2s}}{N^{3-2s}}
	\|\widehat{F}_3\|_{\ell^2(\mathbb{Z}_\lambda)}
	\|\widehat{F}_2\|_{\ell^2(\mathbb{Z}_\lambda)}
	\|\widehat{F}_1\|_{\ell^2(\mathbb{Z}_\lambda)}
	\|\widehat{F}_4\|_{\ell^2(\mathbb{Z}_\lambda)}.
\]
So according to the definition of $F_i (1\leq i\leq 4)$ and the Plancherel identity, the above estimate yields 
	\begin{equation}\label{0923-5}
	\text{LHS of } \eqref{eq:fixed-time-reduced}\lesssim 
	\frac{\lambda^{\frac{3}{2}+2s}}{N^{3-2s}}\prod_{i=1}^4 \| \widehat f_i\|_{\ell^2(\mathbb{Z}_\lambda)} 	
	= \frac{\lambda^{\frac{3}{2}+2s}}{N^{3-2s}}\prod_{i=1}^4 \| f_i\|_{L_x^2(\mathbb{T}_\lambda)}.
	\end{equation}
	
	Finally, due to the assumption (\ref{eq:Lemma1ConditionOfLamandN}), we have
	\[
		\lambda^{\frac{3}{2}+2s} \lesssim N^{(\frac32+2s)(1-s)} \lesssim N^{3-2s}, 
	\]
	so it follows from (\ref{0923-5}) that 
	\[
		\text{LHS of } \eqref{eq:fixed-time-reduced}\lesssim \prod_{i=1}^4 \| f_i\|_{L_x^2(\mathbb{T}_\lambda)} = 
		\text{RHS of } \eqref{eq:fixed-time-reduced}.
	\]
	Hence, \eqref{eq:fixed-time-reduced} is justified, so is \eqref{eq:fixed-time-bound}.
\end{proof}

\subsection{Almost conservation of the third modified energy $E^{(3)}$}
In this section, we will show that the growth rate of the third modified energy $E^{(3)}$ can be controlled effectively. Firstly, according to the time derivative (\ref{deri_E3}) of $E^{(3)}$, we know $\frac{d}{dt} E^{(3)}(t)$ is determined by two quadrilinear functionals and four quintilinear functionals. So before carrying out the main result of this section, we first present several estimates on the quadrilinear and quintilinear functionals. The proofs of Lemma \ref{M41Estimate} and Lemma \ref{M42Estimate} will be postponed to Section \ref{Sec, multi-est}, and the proof of Lemma \ref{lem:M5j-estimates} will be carried out in Section \ref{Sec, quint-est}.

\begin{lemma}\label{M41Estimate}
Let $\alpha\in\mathcal{A}$, $s\in (s_\alpha, 1)$ and $\lambda\gg 1$.
Let $M_{41}$ and $\Omega_*$ be as defined in (\ref{Multiplier41}) and (\ref{eq:bad-region-def}). Suppose $u^{\lambda}$ has zero mean on $\mathbb{T}_\lambda$. 
Then for any $\epsilon\in (0, 2s-2s_\alpha)$, the following quadrilinear estimate holds:
\begin{equation}\label{constant51}
\left| \int_0^1 \Lambda_4(M_{41}(1-\chi_{\Omega_*}); u^\lambda, u^\lambda, v^\lambda, v^\lambda) dt \right| \le C_{10}\,D_1(\lambda,N,\epsilon) \|Iu^\lambda\|_{X^{1,\frac12}_{\lambda}}^2 \|Iv^\lambda\|_{X^{1,\frac12}_{\alpha,\lambda}}^2,
\end{equation}
where $C_{10}>0$ only depends on $\alpha$, $s$ and $\epsilon$, and
\begin{equation}\label{D1decay}
	D_1(\lambda,N,\epsilon) = N^{-2+2s_{\alpha}+\epsilon} \lambda^{2s_{\alpha}+\epsilon}.
\end{equation}
\end{lemma}

\begin{lemma}\label{M42Estimate} 
Let $\alpha\in\mathcal{A}$, $s\in (s_\alpha, 1)$ and $\lambda\gg 1$.
Let $M_{42}$ be as defined in (\ref{Multiplier42}).
Then for any $\epsilon\in (0, 2s-2s_\alpha)$, the following quadrilinear estimate holds 
\begin{equation}\label{constant52}
	\left| \int_0^1 \Lambda_4\bigl(M_{42};v^\lambda,v^\lambda,v^\lambda,v^\lambda\bigr) \,dt \right| \leq C_{11}D_1(\lambda,N,\epsilon) \|Iv^\lambda\|_{X^{1,\frac12}_{\alpha,\lambda}}^4,
 \end{equation} 
where $C_{11}>0$ depends only on $\alpha$, $s$ and $\epsilon$, and $D_1(\lambda,N,\epsilon)$ is as defined in \eqref{D1decay}. 
\end{lemma}

\begin{lemma}\label{lem:M5j-estimates}
	Let $\alpha\in\mathcal{A}$, $s\in (s_\alpha, 1)$ and 
	\begin{equation}\label{eq:condNandlam}
		1\ll \lambda \lesssim N^{\frac13}.
	\end{equation}
	Let $M^P_{5j}$, $1\le j\le4$, be as defined in \eqref{M5jMultipliers}, and assume that $u^\lambda$ has zero mean over $\mathbb{T}_\lambda$. Then for any $\epsilon\in (0, 2s-2s_\alpha)$,
	the following quintilinear estimates hold:
	\begin{equation}\label{eq:M51-est}
		\left|\int_0^1\Lambda_5(M^P_{51};v^\lambda,u^\lambda,v^\lambda,v^\lambda,v^\lambda)\,dt\right|
		\le C_{12}D_2(\lambda,N,\epsilon)
		\|Iu^\lambda\|_{X^{1,\frac12}_{\lambda}}
		\|Iv^\lambda\|_{X^{1,\frac12}_{\alpha,\lambda}}^4,
	\end{equation}
	\begin{equation}\label{eq:M52-est}
		\left|\int_0^1\Lambda_5(M^P_{52};u^\lambda,v^\lambda,v^\lambda,v^\lambda,v^\lambda)\,dt\right|
		\le C_{12}D_2(\lambda,N,\epsilon)
		\|Iu^\lambda\|_{X^{1,\frac12}_{\lambda}}
		\|Iv^\lambda\|_{X^{1,\frac12}_{\alpha,\lambda}}^4,
	\end{equation}
	\begin{equation}\label{eq:M53-est}
		\left|\int_0^1\Lambda_5(M^P_{53};u^\lambda,u^\lambda,u^\lambda,v^\lambda,v^\lambda)\,dt\right|
		\le C_{12}D_2(\lambda,N,\epsilon)
		\|Iu^\lambda\|_{X^{1,\frac12}_{\lambda}}^3
		\|Iv^\lambda\|_{X^{1,\frac12}_{\alpha,\lambda}}^2,
	\end{equation}
	\begin{equation}\label{eq:M54-est}
		\left|\int_0^1\Lambda_5(M^P_{54};u^\lambda,u^\lambda,v^\lambda,u^\lambda,v^\lambda)\,dt\right|
		\le C_{12}D_2(\lambda,N,\epsilon)
		\|Iu^\lambda\|_{X^{1,\frac12}_{\lambda}}^3
		\|Iv^\lambda\|_{X^{1,\frac12}_{\alpha,\lambda}}^2,
	\end{equation}
where $C_{12}>0$ only depends on $\alpha$, $s$ and $\epsilon$, and
\begin{equation}\label{eq:D2-def}
	D_2(\lambda,N,\epsilon)
	:= \lambda^{\frac12} D_1(\lambda,N,\epsilon) = N^{-2+2s_{\alpha}+\epsilon}\lambda^{\frac12+2s_{\alpha}+\epsilon}.
\end{equation}
\end{lemma}

Based on Lemma \ref{M41Estimate}--Lemma \ref{lem:M5j-estimates}, we are ready to establish the almost conservation property of $E^{(3)}$.

\begin{prop}\label{Prop, growth_E3}
	Let $\alpha\in\mathcal{A}$, $s\in (s_\alpha, 1)$ and $1\ll \lambda \lesssim N^{\frac13}$. Assume that $u^{\lambda}$ has zero mean on $\mathbb{T}_\lambda$. Then for any $\epsilon\in (0, 2s-2s_\alpha)$, we have
	\begin{equation}\label{est_E3}\begin{split}
			& \big|E^{(3)}(1) - E^{(3)}(0)\big| = \big| E^{(3)}(u^\lambda, v^\lambda)(1) - E^{(3)}(u^\lambda, v^\lambda)(0) \big| \\
			\leq\, &\, C_{13} D_1(\lambda, N, \epsilon) 
			\bigg( \|Iu^\lambda\|_{X^{1,\frac12}_{\lambda}} +  \|Iv^\lambda\|_{X^{1,\frac12}_{\alpha,\lambda}} \bigg)^{4}
			+
			C_{13} D_2(\lambda, N, \epsilon) 
			\bigg( \|Iu^\lambda\|_{X^{1,\frac12}_{\lambda}} +  \|Iv^\lambda\|_{X^{1,\frac12}_{\alpha,\lambda}} \bigg)^{5},
		\end{split}\end{equation}
	where $C_{13}>0$ only depends on $\alpha$, $s$ and $\epsilon$, and $D_1$ and $D_2$ are as defined in (\ref{D1decay}) and (\ref{eq:D2-def}). 
\end{prop}
\begin{proof}
	By integrating $\frac{d}{dt} E^{(3)}(t)$ in (\ref{deri_E3}) over the time interval $[0,1]$, we conclude (\ref{est_E3}) by adding (\ref{constant51}), (\ref{constant52}), (\ref{eq:M51-est})--(\ref{eq:M54-est}) together.
\end{proof}

\subsection{Global well-posedness}
\label{Subsec, gwp}

In this section, we will discuss the main result, Theorem \ref{Theorem}, of this paper.

\begin{proof}[Proof of Theorem \ref{Theorem}]
We aim to prove the solution $(u,v)$ of the Majda-Biello system (\ref{GeneralMajdaBiello}) exists globally, that is to prove the solution exists on the time interval $[0,T^{*}]$ for any $T^{*}>0$. Equivalently, it reduces to showing the existence of the
scaled solution $(u^{\lambda}, v^{\lambda})$ of \eqref{MajdaBielloScaling} on $[0,\lambda^3T^*]$ for some $\lambda\geq 1$. Our strategy is to take advantage of Lemma \ref{LWPresult} and an induction argument to extend the lifespan of the local solution by time $1$ for each step with the total steps $K =: \left\lceil\lambda^3T^*\right\rceil$.

Since the initial data $u_0\in H^{s}_{0}(\mathbb{T})$, the mean of $u_0$ over $\mathbb{T}$ is zero. Then thanks to the conservation of $\int_{\mathbb{T}} u(x,t) \, dx$ in (\ref{mass_con}), we know the means of $u(\cdot, t)$ and $u^{\lambda}(\cdot, t)$ both vanish for any $t$, that is $u(\cdot, t) \in H^{s}_{0}(\mathbb{T})$ and $u^{\lambda}(\cdot, t) \in H^{s}_{0}(\mathbb{T}_\lambda)$. 
For convenience of notations, let
\[
w_0 := (u_0,v_0),\quad w(t) :=  (u(t),v(t));\qquad w_0^\lambda := (u_0^\lambda,v_0^\lambda),\quad w^\lambda(t) :=  (u^\lambda(t),v^\lambda(t));
\]
and
\[
 \varepsilon_{\lambda,0} := \|w^{\lambda}_0 \|_{L^2(\mathbb{T}_\lambda) \times L^2(\mathbb{T}_{\lambda})},\quad \varepsilon^I_{\lambda,1} := \| Iw^\lambda_0 \|_{H^1(\mathbb{T}_\lambda)\times H^1(\mathbb{T}_\lambda) }.
\]

For $\alpha$ with $s_\alpha \geq 1$, the global well-posedness when $s\geq 1$ had already been known, so without loss of generality, we assume $s_\alpha < 1$, i.e. $\alpha\in \mathcal{A}$, and assume $s\in (s_\alpha, 1)$. For any $\epsilon \in(0, 2s-2s_\alpha)$, by noticing
$s_\alpha = [1+\nu(\rho_\alpha)]/2$ and $s_\alpha < s < 1$, we have
\[
	0 < \epsilon < 2s-2s_\alpha < 1-\nu(\rho_\alpha),
\]
which satisfies the assumptions of Proposition \ref{Prop, growth_E3}. Moreover, we define 
\[
C_{\max} := \max\{1,2C_1,C_6,C_7,C_8,C_9,C_{13}\},
\]
where the constant \(C_1\) is from \eqref{First_Linear_Estimate}, 
$C_6, C_7$ are from \eqref{constant1} and \eqref{constant2}, \(C_8\) is from \eqref{constant4}; \(C_9\) is from Lemma~\ref{lem:fixed-time-bound}; and \(C_{13}\) is from Proposition \ref{Prop, growth_E3}.

For a large positive $N$ whose value will be determined later, we choose $\lambda$ by setting
\begin{equation}\label{RelationofLambdaN}
	\lambda = N^{\frac{2-2s}{3+2s-2s_{\alpha}-\epsilon}}.
\end{equation}
Then it is straightforward to verify that this choice of $\lambda$ satisfies condition~\eqref{eq:ConditionOfLamandN2} in Lemma~\ref{Sigma3Estimate}, 
condition~\eqref{eq:Lemma1ConditionOfLamandN} in
Lemma~\ref{lem:fixed-time-bound} and the condition $\lambda\lesssim N^{\frac13}$ in Proposition \ref{Prop, growth_E3}. In addition, 
(\ref{RelationofLambdaN}) also implies $\lambda^{s} \lesssim N^{1-s}$. Hence, according to Lemma \ref{Iinequality}, we have 
\begin{equation}\label{scaled_IH1norm}
\epsilon_{\lambda,1}^{I} 
\lesssim \lambda^{-\frac{3}{2}-s} N^{1-s} M_0,
\end{equation}
where $M_0 =: \|(u_0,v_0)\|_{H_0^s(\mathbb{T})\times H^s(\mathbb{T})}$.

Based on the choice of $\lambda$ in (\ref{RelationofLambdaN}) and the estimate (\ref{scaled_IH1norm}), we have 
\begin{equation}\label{init_I1norm}
	\epsilon_{\lambda,1}^{I} \lesssim N^{- \frac{2s_\alpha+\epsilon}{3+2s-2s_\alpha-\epsilon}(1-s)} M_0.
\end{equation}
Since the power on $N$ is negative in the above estimate, we can choose $N$ to be sufficiently large such that 
\begin{equation}\label{ConditionC.3.}
	12\, C_{\max}^4 \varepsilon_{\lambda,1}^I \leq 1, \quad 36\, C_{\max}^6 (\varepsilon_{\lambda,1}^I)^2 \leq 1
\end{equation}
In fact, as long as $N$ is chosen such that 
\begin{equation}\label{N_choice1}
	N^{\frac{2s_\alpha+\epsilon}{3+2s-2s_\alpha-\epsilon}(1-s)} \gtrsim 12 \, C_{\max}^4  M_0,
\end{equation}
then (\ref{ConditionC.3.}) is fulfilled. Meanwhile, it is readily seen that 
\begin{equation}\label{ConditionC.2.}
	\varepsilon_{\lambda,0} \le \varepsilon_{\lambda,1}^I\leq 1.
\end{equation}
Since the system \eqref{MajdaBielloScaling} possesses the \(L^2\) conservation property, we have
\begin{equation}\label{L2con_est}
	\| Iw^\lambda(\cdot, t) \|_{L^2 \times L^2} \le \| w^\lambda(\cdot, t) \|_{L^2 \times L^2} = \| w^\lambda(\cdot, 0) \|_{L^2 \times L^2} = \varepsilon_{\lambda,0} \leq 1.
\end{equation}
Thus, we are allowed to apply Lemma~\ref{lemma6}. 

Next, we aim to verify the condition \eqref{LWPcondition} in Lemma~\ref{LWPresult} with 
\[\delta =: \frac{\epsilon}{4}.\] 
Again, based on (\ref{scaled_IH1norm}) and then (\ref{RelationofLambdaN}), we find
\begin{equation}\label{est_0915}
	\lambda^{s_{\alpha}+\delta} \varepsilon^I_{\lambda,1} \lesssim \lambda^{-\frac32-s+s_\alpha+\delta} N^{1-s} M_0 
	= N^{-\frac{2\delta}{3+2s-2s_\alpha-\epsilon} (1-s)} M_0,
\end{equation}
where we used the relation $\delta = \epsilon/4$ in the last equality. Since the power on $N$ is negative in (\ref{est_0915}), we can choose $N$ sufficiently large so that 
\begin{equation}\label{ConditionC.1.}
	\lambda^{s_{\alpha}+\delta} \varepsilon^I_{\lambda,1} \le \frac{\varepsilon_*}{3C_{\max}},
\end{equation}
where \(\varepsilon_*\) is the small constant in \eqref{LWPcondition} in Lemma~\ref{LWPresult}.

We will prove by induction that the solution extends to $[0,K]$, where
$K =: \left\lceil\lambda^3T^*\right\rceil$. 
In view of
Lemma~\ref{LWPresult} and the estimate (\ref{ConditionC.1.}), it suffices to prove that
\begin{equation}\label{eq:iteration-target}
	\|Iw^\lambda(k)\|_{H_0^1\times H^1}
	\leq 3C_{\max}\varepsilon_{\lambda,1}^I,
	\qquad k\in\{0,1,\ldots,K\}.
\end{equation}
Indeed, whenever \eqref{eq:iteration-target} holds at an integer time $k<K$,
Condition~(\ref{ConditionC.1.}) gives
\[
	\lambda^{s_\alpha+\delta}\|Iw^\lambda(k)\|_{H_0^1\times H^1}
	\leq 3C_{\max}\lambda^{s_\alpha+\delta}\varepsilon_{\lambda,1}^I
	\leq\varepsilon_*,
\]
so Lemma~\ref{LWPresult} extends the solution from time $k$ to time $k+1$.

By \eqref{constant1} in Lemma~\ref{lemma6} and the conservation of the $L^2$-norm,
\begin{equation}\label{norm_E1_bdd}
    \|I w^{\lambda}(t)\|_{H^1_0\times H^1}^2 \leq 
    C_{\max} \|I w^{\lambda}(t)\|_{L^2\times L^2}^2 + C_{\max} E^{(1)}(t) 
    \leq C_{\max} \varepsilon_{\lambda,0}^2 + C_{\max} E^{(1)}(t),
\end{equation}
where we used (\ref{L2con_est}) in the second inequality.
Thus, it remains to control $E^{(1)}(t)$. Since the available increment
estimate for $E^{(1)}(t)$ does not provide sufficient decay, we compare it
with $E^{(3)}(t)$, whose increment is much smaller. By the triangle inequality,
\begin{equation}\label{energy_tri}
	    |E^{(1)}(t)| \leq |E^{(1)}(t) - E^{(3)}(t)| + |E^{(3)}(t) - E^{(3)}(0)| + |E^{(3)}(0) - E^{(1)}(0)| + |E^{(1)}(0)|.
\end{equation}
It follows from \eqref{energies} that
\[
    |E^{(1)}(t) - E^{(3)}(t)| \le  |\Lambda_3(\sigma_3;u^\lambda,v^\lambda,v^\lambda)|+|\Lambda_{4}(\sigma_4^{P};u^\lambda,u^\lambda,v^\lambda,v^\lambda)|.
\]
Hence, for each fixed time $t$, Lemma~\ref{Sigma3Estimate} and Lemma~\ref{lem:fixed-time-bound} yield
\begin{equation}\label{eq:fixedtimeboundiniteration}
|E^{(1)}(t) - E^{(3)}(t)|\le C_{\max}\|Iw^\lambda(t) \|_{H_0^1\times H^1}^3 + C_{\max}\|Iw^\lambda(t) \|_{H_0^1\times H^1}^4.
\end{equation}
In particular, at $t=0$ we have
\begin{equation}\label{eq:fixedtimeboundfortequ0}
|E^{(1)}(0) - E^{(3)}(0)|\le C_{\max} (\varepsilon_{\lambda,1}^I)^3+ C_{\max} (\varepsilon_{\lambda,1}^I)^4.
\end{equation}
Moreover, Lemma~\ref{lemma6} gives
\begin{equation}\label{eq:energyestimatefortequ0}
|E^{(1)}(0)| \leq C_{\max} (\varepsilon_{\lambda,1}^I)^2.
\end{equation}
Combining \eqref{energy_tri}, \eqref{eq:fixedtimeboundiniteration},
\eqref{eq:fixedtimeboundfortequ0}, \eqref{eq:energyestimatefortequ0} and
\eqref{norm_E1_bdd}, we obtain
\begin{equation}\label{eq:BasicBootstrapInequ}
	\begin{aligned}
		\| I w^\lambda(t) \|_{H_0^1\times H^1}^2
	&\leq C_{\max} (\varepsilon_{\lambda,0})^2 + C^2_{\max} (\varepsilon_{\lambda,1}^I)^2 + C_{\max}^2 (\varepsilon_{\lambda,1}^I)^3 + C_{\max}^2 (\varepsilon_{\lambda,1}^I)^4\\
		&+C_{\max}^2 \| I w^\lambda(t) \|_{H_0^1\times H^1}^3+C_{\max}^2 \| I w^\lambda(t) \|_{H_0^1\times H^1}^4\\
		&+C_{\max}|E^{(3)}(t)-E^{(3)}(0)|.
	\end{aligned}
\end{equation}

Now we prove \eqref{eq:iteration-target} by induction on $k$. The case $k=0$ follows
immediately from
\[
	\|Iw^\lambda(0)\|_{H_0^1\times H^1}=\varepsilon_{\lambda,1}^I.
\]
Now let $1\leq k\leq K$ and suppose that \eqref{eq:iteration-target} has
been established up to time $k-1$. In particular,
\begin{equation}\label{eq:k-1timeBoundCond}
	\|Iw^\lambda(k-1)\|_{H_0^1\times H^1}
	\leq 3C_{\max}\varepsilon_{\lambda,1}^I.
\end{equation}
By \eqref{eq:k-1timeBoundCond} and the estimate (\ref{ConditionC.1.}), the smallness
condition (\ref{LWPcondition}) in Lemma~\ref{LWPresult} is satisfied at time $k-1$. Therefore,
the solution extends to $[k-1,k]$. By 
\eqref{ContractionBall} and \eqref{eq:k-1timeBoundCond}, it holds that
\begin{equation}\label{crudebound}
\|Iw^\lambda(k)\|_{H_0^1\times H^1}
	\leq
	\|Iw^\lambda\|_{
	Y^1_\lambda(\mathbb T_\lambda\times[k-1,k])\times
	Y^1_{\alpha,\lambda}(\mathbb T_\lambda\times[k-1,k])}
	\leq C_{\max}\|Iw^\lambda(k-1)\|_{H_0^1\times H^1}
	\leq 3C_{\max}^2\varepsilon_{\lambda,1}^I.
\end{equation}
The bound in (\ref{crudebound}) almost satisfies the induction (\ref{eq:iteration-target}) for $k$. It remains to adjust the constant $C_{\max}^2$ on the right hand side of (\ref{crudebound}) to be $C_{\max}$.

Taking $t=k$ in \eqref{eq:BasicBootstrapInequ} and using \eqref{crudebound} to bound the cubic and quartic terms by quadratic terms, we obtain
\begin{equation*}
		\begin{aligned}
		\| I w^\lambda(k) \|_{H_0^1\times H^1}^2
		&\leq C_{\max} (\varepsilon_{\lambda,0})^2 + C^2_{\max} (\varepsilon_{\lambda,1}^I)^2 + C_{\max}^2 (\varepsilon_{\lambda,1}^I)^3 + C_{\max}^2 (\varepsilon_{\lambda,1}^I)^4\\
		&+3C_{\max}^4 \varepsilon_{\lambda,1}^I\,\| I w^\lambda(k) \|_{H_0^1\times H^1}^2+9C_{\max}^6 (\varepsilon_{\lambda,1}^I)^2\,\| I w^\lambda(k) \|_{H_0^1\times H^1}^2\\
		&+C_{\max}|E^{(3)}(k)-E^{(3)}(0)|.
	\end{aligned}
\end{equation*}
Estimates (\ref{ConditionC.3.}) and (\ref{ConditionC.2.}) allow the two terms containing
$\|Iw^\lambda(k)\|_{H_0^1\times H^1}^2$ on the right hand side to be
absorbed into the left hand side. Consequently,
 \begin{equation}\label{eq:lastformineq}
 		\| I w^\lambda(k) \|_{H_0^1\times H^1}^2
 		\leq 8 C^2_{\max} (\varepsilon_{\lambda,1}^I)^2 +2C_{\max}|E^{(3)}(k)-E^{(3)}(0)|.
 \end{equation}
To estimate $|E^{(3)}(k)-E^{(3)}(0)| $, we apply Proposition \ref{Prop, growth_E3} on each interval $[j-1,j]$, $1\leq j\leq k$.
The combination of the induction hypothesis and the time-translated version of
\eqref{crudebound} (with $k$ being replaced by $j$) gives
\begin{equation}\label{eq:engryincreasevicosity}
\begin{split}
	|E^{(3)}(j)-E^{(3)}(j-1)| & \leq 
	C_{\max} D_1(\lambda, N, \epsilon) \|Iw^\lambda\|_{
		Y^1_\lambda(\mathbb T_\lambda\times[j-1,j])\times
		Y^1_{\alpha,\lambda}(\mathbb T_\lambda\times[j-1,j])}^4 \\
	&\quad + C_{\max} D_2(\lambda, N, \epsilon) \|Iw^\lambda\|_{
		Y^1_\lambda(\mathbb T_\lambda\times[j-1,j])\times
		Y^1_{\alpha,\lambda}(\mathbb T_\lambda\times[j-1,j])}^5
	\\
	&\leq 81 C_{\max}^9 D_1(\lambda, N, \epsilon) (\varepsilon_{\lambda,1}^I)^4
	+ 243C_{\max}^{11} D_2(\lambda, N, \epsilon)(\varepsilon_{\lambda,1}^I)^5.
\end{split}\end{equation}
Summing \eqref{eq:engryincreasevicosity} over $1\leq j\leq k$, we find
\begin{equation*}
	\begin{aligned}
		|E^{(3)}(k)-E^{(3)}(0)|
		&\le \sum_{j=1}^{k}|E^{(3)}(j)-E^{(3)}(j-1)|\\
		&\le k\left[ 81C_{\max}^9D_1(\lambda, N, \epsilon)(\varepsilon_{\lambda,1}^I)^4
		+ 243 C_{\max}^{11} D_2(\lambda, N, \epsilon)(\varepsilon_{\lambda,1}^I)^5\right].
	\end{aligned}
\end{equation*}
Substituting this estimate into \eqref{eq:lastformineq}, we obtain
\begin{equation}\label{ind_k}
\begin{split}
	\|Iw^\lambda(k)\|_{H_0^1\times H^1}^2
	&\le \left(8
	+ 2k\left[ 81 C_{\max}^{8}D_1(\lambda, N, \epsilon)(\varepsilon_{\lambda,1}^I)^2
	+ 243C_{\max}^{10}D_2(\lambda,N,\epsilon)(\varepsilon_{\lambda,1}^I)^3\right]\right) (C_{\max} \varepsilon_{\lambda,1}^I)^2.
\end{split}
\end{equation}
Therefore, if
\begin{equation}\label{bootstrap-red}
	2k\left[81 C_{\max}^{8}D_1(\lambda, N, \epsilon)(\varepsilon_{\lambda,1}^I)^2
	+ 243C_{\max}^{10}D_2(\lambda, N, \epsilon)(\varepsilon_{\lambda,1}^I)^3\right]\le 1,
\end{equation}
then it follows from (\ref{ind_k}) that 
$\|Iw^\lambda(k)\|_{H_0^1\times H^1}\le 3C_{\max}\varepsilon_{\lambda,1}^I$,
which closes the induction at time $k$.

We now determine $N$ in order for (\ref{bootstrap-red}) to hold. With
$K=\left\lceil\lambda^3T^*\right\rceil$,
condition~\eqref{bootstrap-red} becomes
\begin{equation}\label{total_est1}
	\left\lceil\lambda^3T^*\right\rceil
	\left(162C_{\max}^{8}D_1(\lambda, N, \epsilon)(\varepsilon_{\lambda,1}^I)^2
	+486C_{\max}^{10}D_2(\lambda, N, \epsilon)(\varepsilon_{\lambda,1}^I)^3\right)\le1.
\end{equation}
Condition~(\ref{ConditionC.1.}) implies
\[
	3C_{\max}\lambda^{s_\alpha+\delta}\varepsilon_{\lambda,1}^I\le \varepsilon_*,
\]
so 
\begin{equation}\label{init_normI1est}
	\varepsilon_{\lambda,1}^I \lesssim \lambda^{-(s_\alpha+\delta)}
\end{equation}
and (\ref{total_est1}) reduces to 
\begin{equation}\label{reduction1}
	\lambda^3 T^{*}\left(
	D_1(\lambda, N, \epsilon)\lambda^{-2(s_\alpha+\delta)}
	+ D_2(\lambda, N, \epsilon)\lambda^{-3(s_\alpha+\delta)}
	\right)\ll1.
\end{equation}
According to the definition (\ref{eq:D2-def}) of $D_2$, we know 
\[\begin{split}
	D_2(\lambda, N, \epsilon)\lambda^{-3(s_\alpha+\delta)} 
	= \lambda^{\frac12} D_1(\lambda, N, \epsilon)\lambda^{-3(s_\alpha+\delta)} \leq D_1(\lambda, N, \epsilon)\lambda^{-2(s_\alpha+\delta)},
\end{split}\]
where the last inequality is due to $\lambda\geq 1$ and $s_\alpha\geq \frac12$. Therefore, (\ref{reduction1}) further boils down to 
\begin{equation}\label{reduction2}
	\lambda^3 T^{*} D_1(\lambda, N, \epsilon) \lambda^{-2(s_\alpha+\delta)} \ll 1.
\end{equation}
Since $D_1(\lambda, N, \epsilon) = N^{-2+2s_{\alpha}+\epsilon}\lambda^{2s_{\alpha}+\epsilon}$ and $\delta = \epsilon/4$, the above criterion (\ref{reduction2}) becomes 
\begin{equation}\label{reduction3}
	\lambda^{3 + \frac{\epsilon}{2}} N^{-2+2s_\alpha+\epsilon} T^{*} \ll 1.
\end{equation}

Due to the choice of $\lambda$ in (\ref{RelationofLambdaN}), (\ref{reduction3}) is equivalent to 
\begin{equation}\label{reduction4}
	N^{(3 + \frac{\epsilon}{2}) \frac{2-2s}{3+2s-2s_{\alpha}-\epsilon} } N^{-2+2s_\alpha+\epsilon} \ll (T^*)^{-1}.
\end{equation}
Noticing that 
\begin{equation}\label{iden}
	\Big(3 + \frac{\epsilon}{2}\Big) \frac{2-2s}{3+2s-2s_{\alpha}-\epsilon} + (-2+2s_\alpha+\epsilon) = \frac{(1-s)\epsilon - (5-2s_\alpha-\epsilon)(2s-2s_\alpha-\epsilon)}{3+2s-2s_{\alpha}-\epsilon},
\end{equation}
we choose $\epsilon = s-s_\alpha$ so that 
\[
	\text{the right hand side of (\ref{iden})} = -\frac{(s-s_\alpha)(4-s_\alpha)}{3+s-s_{\alpha}} < 0.
\]
As a result, (\ref{reduction4}) becomes 
\begin{equation}\label{reduction5}
	N^{-\frac{(s-s_\alpha)(4-s_\alpha)}{3+s-s_{\alpha}}} \ll (T^*)^{-1},
\end{equation}
which is equivalent to 
\begin{equation}\label{Temp0006-red}
	N \gg (T^*)^{\frac{3+s-s_{\alpha}}{(s-s_\alpha)(4-s_\alpha)}}.
\end{equation}
Since $s>s_\alpha$, (\ref{Temp0006-red}) holds when $N$ is chosen to be a sufficiently large number.

As a summary, by choosing \(N\) sufficiently large so that \eqref{ConditionC.3.},
\eqref{ConditionC.1.}, and \eqref{Temp0006-red} hold simultaneously, then the induction for (\ref{eq:iteration-target}) closes and the scaled solution $(u^{\lambda}, v^{\lambda})$ of (\ref{MajdaBielloScaling}) exists on
\([0,K]\supset[0,\lambda^3T^*]\). Scaling back yields a solution $(u,v)$ of the original system (\ref{GeneralMajdaBiello}) such that
\[
(u,v)\in C\bigl([0,T^*];H^s_0(\mathbb T)\times H^s(\mathbb T)\bigr).
\]
Uniqueness in the local resolution class follows by successive
applications of the local uniqueness result. For each fixed $T^*>0$,
the parameters can be chosen uniformly for initial data in a bounded
neighborhood of the given datum, and concatenating the local
continuous-dependence estimates yields continuous dependence on
$[0,T^*]$. Negative times are treated using the symmetry
$(u,v)(x,t)\mapsto(u,v)(-x,-t)$. Since \(T^*>0\) is arbitrary, the global well-posedness is justified.

\end{proof}

\section{Quadrilinear estimates}
\label{Sec, multi-est}

\subsection{Preliminary quadrilinear estimates on resonant sets}

The goal of this section is to justify Lemma \ref{M41Estimate} and Lemma \ref{M42Estimate}. 
We first demonstrate two preliminary quadrilinear estimates (\ref{eq:resonant-quadrilinear-41}) and (\ref{eq:resonant-quadrilinear-42}) on resonant sets
$\Delta$ and $\Delta'$ below:
\begin{equation}\begin{split}\label{res_set}
	\Delta&:=\left\{(n_1,n_2,n_3,n_4)\in\Gamma_4^\lambda:
	|n_4+c_1n_1|\leq\frac{\kappa_\alpha}{\lambda}
	\text{ or }|n_4+c_2n_1|\leq\frac{\kappa_\alpha}{\lambda}\right\},\\
	\Delta'&:=\left\{(n_1,n_2,n_3,n_4)\in\Gamma_4^\lambda:
	|n_2+c_1n_{14}|\leq\frac{\kappa_\alpha}{\lambda}
	\text{ or }|n_2+c_2n_{14}|\leq\frac{\kappa_\alpha}{\lambda}\right\},
\end{split}\end{equation}
where $c_1, c_2$ are as defined in (\ref{c1c2}), and $\kappa_\alpha$ is as defined in \eqref{eq:kappa-alpha-choice}.
These two preliminary quadrilinear estimates will be applied later to tackle the most challenging case on resonant sets in Lemma \ref{M41Estimate} and Lemma \ref{M42Estimate} respectively.
We point out that the lattice frequencies in both $\Delta$ and $\Delta'$ are sparse. Taking the first set $\Delta$ as an example,
for any fixed $n_1\in\mathbb{Z}_\lambda$, there exists at most one possible choice for $n_4$ under the restriction 
\[|n_4 + c_j n_1| \leq \frac{\kappa_\alpha}{\lambda}, \qquad j=1,2,\]
 due to the fact that $\kappa_\alpha\leq \frac18$. It is this sparsity that provides an additional decay factor $\lambda$ in the denominator on the right hand side of (\ref{eq:resonant-quadrilinear-41}) and (\ref{eq:resonant-quadrilinear-42}).

\begin{lemma}\label{lem:resonant-quadrilinear}
Let $\alpha\in\mathcal A$, $\lambda\gg1$. Let the functions $H_\alpha$ and $H_\alpha'$ be as defined in (\ref{eq:quartic-resonance-H}) and (\ref{eq:quartic-resonance-Hprime}). Let the resonant sets $\Delta$ and $\Delta'$ be as defined in (\ref{res_set}).
\begin{itemize}
 \item[(a)] Let $L_i = \tau_i - n_i^3$ for $ i\in\{1,2\}$
and $L_i = \tau_i - \alpha n_i^3$ for $i\in\{3,4\}$. Then
\begin{equation}\label{eq:resonant-quadrilinear-41}
	  \int_{\sum\limits_{i=1}^{4}\tau_i=0}\int_{\Delta}
	  \langle H_\alpha \rangle^{\frac12} \frac{\prod_{i=1}^4\widetilde f_i(n_i,\tau_i)}
	  {\prod_{i=1}^4\langle L_i\rangle^{1/2}}
	  \leq \frac{C}{\lambda} \prod_{i=1}^4\|f_i\|_{L^2(\mathbb{T}_\lambda\times \mathbb{R})},
 \end{equation}
where $\widetilde{f_i}$ represents the space-time Fourier transform of $f_i$, and the constant $C$ only depends on $\alpha$.

 \item[(b)]
 Let $L_i = \tau_i - \alpha n_i^3$ for $i\in\{1,2,3,4\}$. Then similar to (\ref{eq:resonant-quadrilinear-41}), we have
 \begin{equation}\label{eq:resonant-quadrilinear-42}
 	\int_{\sum\limits_{i=1}^{4}\tau_i=0}\int_{\Delta'}
	\langle H_\alpha' \rangle^{\frac12}
 	\frac{\prod_{i=1}^4\widetilde f_i(n_i,\tau_i)}
 	{\prod_{i=1}^4\langle L_i\rangle^{1/2}}
 	\leq \frac{C}{\lambda}
 	\prod_{i=1}^4\|f_i\|_{L^2(\mathbb{T}_\lambda\times \mathbb{R})}.
 \end{equation}
\end{itemize}
\end{lemma}

\begin{proof}[Proof of (\ref{eq:resonant-quadrilinear-41}) in  Lemma \ref{lem:resonant-quadrilinear}]

Without loss of generality, we may assume all $\{\widetilde{f}_i\}_{i=1}^{4}$ are nonnegative functions.
For frequencies in $\Delta$, either $|n_4 + c_1 n_1| \leq \kappa_\alpha/\lambda$ or $|n_4 + c_2 n_1| \leq \kappa_\alpha/\lambda$. In the following, we will only consider the case $|n_4+c_1n_1|\leq\kappa_\alpha/\lambda$, since the second case is analogous. 

Let $\theta = n_4 + c_1 n_1$. Then 
\begin{equation}\label{cov_th}
	 \begin{cases}
	 n_3=-c_2n_1-n_2-\theta,\\
	 n_4=-c_1n_1+\theta.
	 \end{cases}
\end{equation}
Meanwhile, it is readily seen that $\theta\in\Theta_{\kappa_\alpha}(n_1;-c_1)$ which was defined in \eqref{Thetan1}. 
For notation convenience, we will drop the subscript $\kappa_\alpha$ and simply write $\Theta_{\kappa_\alpha}(n_1;-c_1)$ as $\Theta(n_1;-c_1)$ in the rest of this proof.

Denote $L_M = \max\limits_{1\leq i\leq4}\langle L_i\rangle$. Then 
$L_M \sim\sum\limits_{i=1}^4\langle L_i\rangle$.
Since $\sum\limits_{i=1}^4\tau_i=0$, we have
\[
	 L_{M} \sim \sum_{i=1}^4\langle L_i\rangle
	 \geq\Big\langle\sum_{i=1}^4L_i\Big\rangle
	 = \langle n_1^3 + n_2^3 + \alpha n_3^3 + \alpha n_4^3\rangle 
	 = \langle H_\alpha\rangle.
\]
We partition the integration domain according to which modulation is
largest among $\{\langle L_i\rangle\}_{i=1}^{4}$. We detail the case $\langle L_1\rangle=L_M$; the same spatial
estimate applies in the other cases, assigning the $L_t^2$ norm to
the factor with largest modulation and the $L_t^6$ norm to each of
the remaining factors, see \eqref{F_def} and \eqref{0925-6}.

Define functions $\{F_i\}_{i=1}^{4}$ whose space-time Fourier transforms $\{\widetilde{F}_i\}_{i=1}^{4}$ satisfy
\begin{equation}\label{F_def}
	 \widetilde F_1(n_1,\tau_1)=\widetilde f_1(n_1,\tau_1),\qquad
	 \widetilde F_i(n_i,\tau_i)=
	 \frac{\widetilde f_i(n_i,\tau_i)}{\langle L_i\rangle^{1/2}},
	 \quad 2\leq i\leq4.
\end{equation}
Then it follows from $\langle L_1\rangle = L_{M} \gtrsim \langle H_\alpha\rangle$ and the change of variable (\ref{cov_th}) that 
\begin{equation}\label{0925-1}
 \begin{aligned}
 \text{LHS of \eqref{eq:resonant-quadrilinear-41}}\lesssim &
 \int_{\sum\limits_{i=1}^{4} \tau_i = 0}\Big(\frac{1}{2\pi\lambda}\Big)^3
 \sum_{n_1,n_2\in\mathbb Z_\lambda}
 \sum_{\theta\in\Theta(n_1;-c_1)}\\
 &\qquad\widetilde F_1(n_1,\tau_1)\widetilde F_2(n_2,\tau_2)
 \widetilde F_3(-c_2n_1-n_2-\theta,\tau_3)
 \widetilde F_4(-c_1n_1+\theta,\tau_4).
 \end{aligned}
\end{equation}
Applying the Plancherel identity in time to (\ref{0925-1}) yields 
\begin{equation}\label{0925-0}
\begin{aligned}
	\text{LHS of \eqref{eq:resonant-quadrilinear-41}}\lesssim &
	\int_{\mathbb{R}}\Big(\frac{1}{2\pi\lambda}\Big)^3
	\sum_{n_1,n_2\in\mathbb Z_\lambda}
	\sum_{\theta\in\Theta(n_1;-c_1)}\\
	&\qquad\widehat F_1(n_1, t)\widehat F_2(n_2, t)
	\widehat F_3(-c_2n_1-n_2-\theta, t)
	\widehat F_4(-c_1n_1+\theta, t) \,dt,
\end{aligned}
\end{equation}
where $\widehat{F}_i$ represents the spatial Fourier transform of $F_i$ for $i\in\{1,2,3,4\}$. Now we reorganize the summation in (\ref{0925-0}) to obtain 
\begin{equation}\label{0925-2}
\begin{aligned}
	\text{LHS of \eqref{eq:resonant-quadrilinear-41}}\lesssim &
	\int_{\mathbb{R}}\Big(\frac{1}{2\pi\lambda}\Big)^3
	\sum_{n_1\in\mathbb{Z}_\lambda} \widehat F_1(n_1, t) \bigg( \sum_{\theta\in\Theta(n_1;-c_1)} \widehat F_4(-c_1n_1+\theta, t) \\
	& \qquad \qquad \Big[ \sum_{n_2\in\mathbb{Z}_\lambda} \widehat F_2(n_2, t)
	\widehat F_3(-c_2n_1-n_2-\theta, t) \Big]\bigg) \,dt.
\end{aligned}
\end{equation}

By summing $n_2$ in (\ref{0925-2}) via the Cauchy-Schwarz inequality, and noticing that 
\[\begin{split}
	&\frac{1}{2\pi\lambda} \sum_{n_2\in\mathbb{Z}_\lambda} \widehat F_2(n_2, t)
	\widehat F_3(-c_2n_1-n_2-\theta, t) \\
 	\le\,\,& \|\widehat F_2(n_2,t)\|_{\ell^2_{n_2}}
	\|\widehat F_3(-c_2n_1-n_2-\theta,t)\|_{\ell^2_{n_2}} 
	= \|\widehat F_2(\cdot, t)\|_{\ell^2} \|\widehat F_3(\cdot, t)\|_{\ell^2},
\end{split}\]
then we obtain 
\begin{equation}\label{0925-4}
\begin{aligned}
	& \text{LHS of \eqref{eq:resonant-quadrilinear-41}} \\
	\lesssim\quad & \frac{1}{\lambda^2}
	\int_{\mathbb R} 
	\|\widehat F_2(\cdot, t)\|_{\ell^2} \|\widehat F_3(\cdot, t)\|_{\ell^2}
	\sum_{n_1\in\mathbb Z_\lambda} |\widehat F_1(n_1,t)| \bigg(\sum_{\theta\in\Theta(n_1;-c_1)}
	|\widehat F_4(-c_1n_1+\theta,t)|\bigg) \,dt.
\end{aligned}
\end{equation}
Now we apply the Cauchy-Schwarz inequality in $n_1$ to find 
\begin{equation}\label{0925-5}\begin{split}
	& \frac{1}{2\pi\lambda}\sum_{n_1\in\mathbb Z_\lambda} |\widehat F_1(n_1,t)| \bigg(\sum_{\theta\in\Theta(n_1;-c_1)}
	|\widehat F_4(-c_1n_1+\theta,t)|\bigg)  \\
	\leq\,\, & \|\widehat F_1(\cdot, t)\|_{\ell^2}
	\Big\|\sum_{\theta\in\Theta(n_1;-c_1)}
	|\widehat F_4(-c_1n_1+\theta,t)|\Big\|_{\ell^2_{n_1}} 
	\lesssim \|\widehat F_1(\cdot, t)\|_{\ell^2} \|\widehat F_4(\cdot, t)\|_{\ell^2},
\end{split}\end{equation}
where the last inequality is due to Lemma~\ref{lem:admissible-frequency-sum} with
$A=\kappa_\alpha$, $a=-c_1$ and $m=0$. 

Plugging (\ref{0925-5}) into (\ref{0925-4}) and applying the H\"older inequality in time yields 
\begin{equation}\label{0925-6}
\begin{aligned}
	\text{LHS of \eqref{eq:resonant-quadrilinear-41}}
	\lesssim &
	\frac{1}{\lambda}
	\|\widehat F_1\|_{L_t^2\ell^2_{n_1}}
	\|\widehat F_2\|_{L_t^6\ell^2}
	\|\widehat F_3\|_{L_t^6\ell^2}
	\|\widehat F_4\|_{L_t^6\ell^2}.
\end{aligned}
\end{equation}
Finally, by Plancherel's identity and Lemma~\ref{lemma3},
\begin{equation}\label{0925-7}
	 \|\widehat F_1\|_{L_t^2\ell^2_{n_1}}\lesssim\|f_1\|_{L^2_{x,t}},
	 \qquad
	 \|\widehat F_i\|_{L_t^6\ell^2_{n_i}}\lesssim\|f_i\|_{L^2_{x,t}},
	 \quad 2\leq i\leq4.
\end{equation}
Consequently, the combination of (\ref{0925-7}) and (\ref{0925-6}) gives 
\begin{equation}\label{0925-8}
	\text{LHS of \eqref{eq:resonant-quadrilinear-41}}
	\lesssim \frac{1}{\lambda}
	\prod_{i=1}^4\|f_i\|_{L^2_{x,t}},
\end{equation}
which justifies (\ref{eq:resonant-quadrilinear-41}). 
\end{proof}

\begin{proof}[Proof of (\ref{eq:resonant-quadrilinear-42}) in  Lemma \ref{lem:resonant-quadrilinear}]
In this case, we can also assume $\{\widetilde f_i\}_{i=1}^{4}$ are nonnegative.
For frequencies in $\Delta'$, either $|n_2 + c_1 n_{14}| \leq \kappa_\alpha/\lambda$ or $|n_2 + c_2 n_{14}| \leq \kappa_\alpha/\lambda$. In the following, we will only consider the case $|n_2 + c_1 n_{14}| \leq \kappa_\alpha/\lambda$ since the other case is analogous. By replacing $n_{14}$ with $-n_2-n_3$ and taking advantage of the fact $c_1+c_2=1$, we find 
\[
	\frac{\kappa_\alpha}{\lambda} \geq |n_2 + c_1 n_{14}| = |c_2 n_2 - c_1 n_3|,
\]
which implies that 
\[
	\Big| n_2 - \frac{c_1}{c_2} n_3 \Big| \leq \frac{\kappa_\alpha}{c_2 \lambda}.
\]
Let $\theta = n_2 - \frac{c_1}{c_2} n_3$. Then 
\begin{equation}\label{cov-0926}
\begin{cases}
	n_2 = \dfrac{c_1}{c_2}n_3+\theta,\\
	n_4 = -n_1-\dfrac{n_3}{c_2}-\theta.
\end{cases}
\end{equation}
Meanwhile, one can check that $\theta\in \Theta_{\frac{\kappa_\alpha}{c_2}}(n_3; \frac{c_1}{c_2})$ which was defined in \eqref{Theta_def} \big(with $A = \frac{\kappa_\alpha}{c_2}$, $n=n_3$ and $a=\frac{c_1}{c_2}$\big).

Since $\sum\limits_{i=1}^{4} L_i = - H'_\alpha$, we have 
$L_M \gtrsim \langle H'_\alpha \rangle$.
Again, without loss of generality, we assume $\langle L_1\rangle = L_M$ and define $\{F_i\}_{i=1}^{4}$ as in (\ref{F_def}). Then applying the change of variable (\ref{cov-0926}) to (\ref{eq:resonant-quadrilinear-42}) yields 
\begin{equation}\label{0926-1}
\begin{aligned}
	\text{LHS of \eqref{eq:resonant-quadrilinear-42}}\lesssim &
	\int_{\sum\tau_i=0}\Big(\frac{1}{2\pi\lambda}\Big)^3
	\sum_{n_1,n_3\in\mathbb Z_\lambda}
	\sum_{\theta\in\Theta_{\frac{\kappa_\alpha}{c_2}}(n_3;\frac{c_1}{c_2})}\\
	&\qquad\widetilde F_1(n_1,\tau_1)\widetilde F_3(n_3,\tau_3)
	\widetilde F_2\Big(\frac{c_1}{c_2}n_3+\theta,\tau_2\Big)
	\widetilde F_4\Big(-n_1-\frac{n_3}{c_2}-\theta,\tau_4\Big).
\end{aligned}
\end{equation}
Applying the Plancherel identity in time to (\ref{0926-1}) yields 
\begin{equation}\label{0926-0}
\begin{aligned}
	\text{LHS of \eqref{eq:resonant-quadrilinear-42}}\lesssim &
	\int_{\mathbb{R}}\Big(\frac{1}{2\pi\lambda}\Big)^3
	\sum_{n_1,n_3\in\mathbb Z_\lambda}
	\sum_{\theta\in\Theta_{\frac{\kappa_\alpha}{c_2}}(n_3;\frac{c_1}{c_2})}\\
	&\qquad\widehat F_1(n_1, t)\widehat F_3(n_3, t)
	\widehat F_2\Big(\frac{c_1}{c_2}n_3+\theta, t\Big)
	\widehat F_4\Big(-n_1-\frac{n_3}{c_2}-\theta, t\Big) \,dt,
\end{aligned}
\end{equation}
where $\widehat{F}_i$ represents the spatial Fourier transform of $F_i$ for $i=1,2,3,4$.
Now we reorganize the summation in (\ref{0926-0}) to obtain 
\begin{equation}\label{0926-2}
\begin{aligned}
	\text{LHS of \eqref{eq:resonant-quadrilinear-42}}\lesssim &
	\int_{\mathbb{R}}\Big(\frac{1}{2\pi\lambda}\Big)^3
	\sum_{n_3\in\mathbb{Z}_\lambda} \widehat F_3(n_3, t) \bigg( \sum_{\theta\in\Theta_{\frac{\kappa_\alpha}{c_2}}(n_3;\frac{c_1}{c_2})} \widehat F_2\Big(\frac{c_1}{c_2}n_3+\theta, t\Big) \\
	& \qquad \qquad \Big[ \sum_{n_1\in\mathbb{Z}_\lambda} \widehat F_1(n_1, t)
	\widehat F_4\Big(-n_1-\frac{n_3}{c_2}-\theta, t\Big)\Big]\bigg)  \,dt.
\end{aligned}
\end{equation}

The rest of the proof follows the same strategy as that from  (\ref{0925-2}) to (\ref{0925-8}). Hence, (\ref{eq:resonant-quadrilinear-42}) is verified.
\end{proof}

\subsection{Proof of Lemma \ref{M41Estimate}}
\label{Subsec, pf_M41}

\subsubsection{Reduction to pointwise bounds of a multiplier}
\label{Subsubsec, Reduction_pb_m}

Let $\alpha\in\mathcal{A}$, $s\in (s_\alpha, 1)$ and $\lambda \gg 1$. We recall (\ref{constant51}) below:
\begin{equation}\label{M41_estp}
	\left| \int_0^1 \Lambda_4(M_{41}(1-\chi_{\Omega_*}); u^\lambda, u^\lambda, v^\lambda, v^\lambda) dt \right| \le C_{10}\,D_1(\lambda,N,\epsilon) \|Iu^\lambda\|_{X^{1,\frac12}_{\lambda}}^2 \|Iv^\lambda\|_{X^{1,\frac12}_{\alpha,\lambda}}^2,
\end{equation}
where $C_{10}>0$ only depends on $\alpha$, $s$ and $\epsilon$, 
$\Omega_*$ is the set as defined in (\ref{eq:bad-region-def}), and \[
	D_1(\lambda,N,\epsilon) = N^{-2+2s_{\alpha}+\epsilon} \lambda^{2s_{\alpha}+\epsilon}.
\] 
We are going to perform some standard reductions. Firstly, since we are dealing with non-endpoint case for which $s\in (s_\alpha, 1)$ and $\epsilon\in (0, 2s - 2s_\alpha)$, it suffices to consider the whole time region $\mathbb{R}$ although the original time integral domain in (\ref{M41_estp}) is restricted to $[0,1]$, 
thanks to the argument in  \cite[pp.~664--665]{Iteam2001} or \cite[Lemma 4.6]{Oh2007}.
That is to say, (\ref{M41_estp}) boils down to the following estimate:
\begin{equation}\label{0929-4}
	\left| \int_{\mathbb{R}} \Lambda_4(M_{41}(1-\chi_{\Omega_*}); u^\lambda, u^\lambda, v^\lambda, v^\lambda) dt \right| \le C_{10} D_1(\lambda,N,\epsilon) \|Iu^\lambda\|_{X^{1,\frac12}_{\lambda}}^2 \|Iv^\lambda\|_{X^{1,\frac12}_{\alpha,\lambda}}^2.
\end{equation}


On the other hand, in the region
\begin{equation}\label{Omega_neg1}
	\Omega_{-1} := \Big\{(n_1,n_2,n_3,n_4)\in \Gamma_4^\lambda: \max_{1\leq i\leq 4} |n_i|\ll N \Big\},
\end{equation}
we have $m(n_i)=1$ for all $i=1,2,3,4$, which implies that 
\begin{equation*}
	M_{41} = 2 i n_{23} \quad\text{on}\quad \Omega_{-1}.
\end{equation*}
The change of variables $(n_1,n_2,n_3,n_4)\mapsto(n_2,n_1,n_4,n_3)$ preserves $\Omega_{-1}$, the normalized counting measure, and the product of the Fourier factors; hence we may average the original integral with its relabeled form.
Consequently, the contribution of
$\Lambda_4(M_{41};u^\lambda,u^\lambda,v^\lambda,v^\lambda)$ from $\Omega_{-1}$ is
\[
\int_{\Gamma_4^\lambda\cap \Omega_{-1}} 2 i n_{23} \, \widehat{u^\lambda} \,\widehat{u^\lambda} \,\widehat{v^\lambda}\,\widehat{v^\lambda} = \int_{\Gamma_4^\lambda\cap \Omega_{-1}} i[n_{14}+n_{23}]\, \widehat{u^\lambda}\, \widehat{u^\lambda} \,\widehat{v^\lambda} \,\widehat{v^\lambda} = 0,
\]
where $\widehat{u^\lambda}$ and $\widehat{v^\lambda}$ refer to the spatial Fourier transforms of $u^\lambda$ and $v^\lambda$ respectively.
Therefore, it remains to prove the estimate \eqref{0929-4} only on the spatial domain
\begin{equation}\label{omega0}
	\Omega_0 := \Gamma_4^\lambda \setminus \big(\Omega_{-1} \cup \Omega_{*}\big).
\end{equation}
That is to justify 
\begin{equation}\label{0929-6}
	\left| \int_{\mathbb{R}} \Lambda_4(M_{41}\chi_{\Omega_0}; u^\lambda, u^\lambda, v^\lambda, v^\lambda) \,dt \right|
	\le C_{10} D_1(\lambda,N,\epsilon) \|Iu^\lambda\|_{X^{1,\frac12}_{\lambda}}^2 \|Iv^\lambda\|_{X^{1,\frac12}_{\alpha,\lambda}}^2.
\end{equation}

Based on the definition of $\Lambda_4$ and the Fourier transform in time, (\ref{0929-6}) is converted as 
\begin{equation}\label{1003-1}\begin{split}
	&\int_{\sum\limits_{i=1}^{4}\tau_i=0}\int_{\Omega_0} M_{41}(n_1, n_2, n_3, n_4) \widetilde{u^\lambda}(n_1, \tau_1)\, \widetilde{u^\lambda}(n_2,\tau_2)\, \widetilde{v^\lambda}(n_3, \tau_3)\, \widetilde{v^\lambda}(n_4, \tau_4)  \\
	\lesssim \,\,& D_1(\lambda, N, \epsilon) \big\| \langle n\rangle m(n) \langle \tau - n^3\rangle^{\frac12} \widetilde{u^\lambda}(n,\tau) \big\|_{\ell_n^2(\mathbb Z_\lambda)L_\tau^2(\mathbb R)}^2
	\big\| \langle n\rangle m(n) \langle \tau - \alpha n^3\rangle^{\frac12} \widetilde{v^\lambda}(n,\tau) \big\|_{\ell_n^2(\mathbb Z_\lambda)L_\tau^2(\mathbb R)}^2.
\end{split}\end{equation}
Define functions $\{f_i\}_{i=1}^{4}$ such that their space-time Fourier transforms $\{\widetilde f_i\}_{i=1}^{4}$ satisfy
\[
\widetilde f_i(n_i,\tau_i)=
\begin{cases}
	\langle n_i\rangle m(n_i)\langle L_i\rangle^{\frac12} \widetilde{u^\lambda}(n_i,\tau_i), &i=1,2,\\
	\langle n_i\rangle m(n_i)\langle L_i\rangle^{\frac12} \widetilde{v^\lambda}(n_i,\tau_i), &i=3,4,
\end{cases}
\]
where 
\[
	L_i = \tau_i - n_i^3 \quad \text{for}\quad i=1,2, \qquad
	\text{and}\qquad L_i = \tau_i - \alpha n_i^3 \quad \text{for} \quad i=3,4.
\]
Then (\ref{1003-1}) is reduced to 
\begin{equation}\label{1003-2}
	\int_{\sum\limits_{i=1}^{4} \tau_i=0}\int_{\Omega_0}
	M_{41}(n_1,n_2,n_3,n_4) \prod_{i=1}^{4} \frac{\widetilde f_i(n_i,\tau_i)}{\langle n_i\rangle m(n_i) \langle L_i\rangle^{1/2}}
	\lesssim D_1(\lambda, N,\epsilon)
	\prod_{i=1}^4\|f_i\|_{L^2(\mathbb{T}_\lambda\times \mathbb{R})}.
\end{equation}
Denote a multiplier $\mathcal{M}$ on $\Gamma_4^\lambda$ as below:
\begin{equation}\label{multi_M}
	\mathcal{M}(n_1,n_2,n_3,n_4) := \frac{|M_{41}(n_1,n_2,n_3,n_4)|}{\prod_{i=1}^{4} \langle n_i\rangle m(n_i)}.
\end{equation}
Then it suffices to prove 
\begin{equation}\label{1003-4}
	\int_{\sum\limits_{i=1}^{4} \tau_i=0}\int_{\Omega_0} \mathcal{M}(n_1,n_2,n_3,n_4)
	\frac{\prod_{i=1}^4 |\widetilde f_i(n_i,\tau_i)|}
	{\prod_{i=1}^4\langle L_i\rangle^{1/2}}
	\lesssim D_1(\lambda, N,\epsilon)
	\prod_{i=1}^4\|f_i\|_{L^2(\mathbb{T}_\lambda\times \mathbb{R})}.
\end{equation}
Without loss of generality, we may assume all $\{\widetilde{f}_i\}_{i=1}^{4}$ are nonnegative functions.

After the reduction to the spatial domain $\Omega_0$ on the LHS of (\ref{1003-4}), we will further reduce this estimate to be the problem of obtaining pointwise bounds of certain multipliers.

\begin{itemize}
	\item For any subset $A\subset \Omega_0$, if we can establish the following pointwise bound for $\mathcal{M}$ on $A$:
	\begin{equation}\label{MultiTrans}
		\mathcal{M}(n_1,n_2,n_3,n_4) \lesssim D_1(\lambda,N,\epsilon) \quad \text{on} \quad A,
	\end{equation}
	then \eqref{1003-4}, when being restricted on $A\subset \Omega_0$, follows from the H\"older inequality and Lemma \ref{lemma4}. As a result, the integral estimate (\ref{1003-4}) on $A$ is converted to be the pointwise bound problem (\ref{MultiTrans}).

	\item In the above analysis, the estimate (\ref{MultiTrans}) is essential. Unfortunately, this estimate may not hold on resonant set $\Delta$ which was defined in (\ref{res_set}). So we need to take advantage of the finer estimate (\ref{eq:resonant-quadrilinear-41}) in Lemma \ref{lem:resonant-quadrilinear} to handle this situation. More precisely, for any subset $B\subset \Omega_0\cap \Delta$, if we prove that 
	\begin{equation}\label{MultiTrans_weak}
		\mathcal{M}(n_1,n_2,n_3,n_4) \lesssim \lambda \langle H_\alpha\rangle^{\frac12} D_1(\lambda,N,\epsilon) \quad \text{on} \quad B,
	\end{equation}
	then \eqref{1003-4}, when being restricted on $B \subset \Omega_0\cap \Delta$, follows from (\ref{eq:resonant-quadrilinear-41}) in Lemma \ref{lem:resonant-quadrilinear}. Hence, the integral estimate (\ref{1003-4}) on $B$ boils down to the pointwise bound problem (\ref{MultiTrans_weak}).
	
\end{itemize}

\subsubsection{Estimates in the high-low frequency regions}
\label{Subsubsec, M41_high-low}
We first rewrite $M_{41}$ in \eqref{Multiplier41} as below:
\begin{equation}\label{M41_comp}
	M_{41}(n_1,n_2,n_3,n_4) = 2i n_{23} \frac{R^I(n_1,n_{23},n_4)}{R(n_1,n_{23},n_4)} , \qquad \forall\, (n_1,n_2,n_3,n_4)\in \Gamma_4^{\lambda},
\end{equation}
where
\begin{equation}\label{DefOfRandRI}
\begin{split}
	& R^I(x,y,z) := x^3 m^2(x) + \alpha y^3 m^2(y) + \alpha z^3 m^2(z), \\
	& R(x,y,z) := x^3 + \alpha y^3 + \alpha z^3,
\end{split}
\qquad\forall\, x+y+z=0.
\end{equation}
Then it follows from (\ref{multi_M}) that
\begin{equation}\label{M_comp}
	\mathcal{M}(n_1, n_2, n_3, n_4) 
	= \frac{|M_{41}(n_1,n_2,n_3,n_4)|}{\prod_{i=1}^{4} \langle n_i\rangle m(n_i)} 
	= \frac{2 |n_{23}| | R^{I}(n_1, n_{23}, n_4)|}{|R(n_1, n_{23}, n_4)| \prod_{i=1}^{4} \langle n_i\rangle m(n_i)}.
\end{equation}

So in order to provide a pointwise bound for $\mathcal{M}$, the challenging region is when $R(n_1, n_{23}, n_4)$ is near $0$. Thus, we split $\Omega_0$ into four subregions:
\begin{equation}\begin{split}\label{Omega_reg}
	&\Omega_1 = \{ \Omega_0: |n_{23}| \ll N \}, \hspace{3.3cm}
	\Omega_2 = \{ \Omega_0: |n_{23}| \gtrsim N,\ |n_1| \gtrsim N \gg |n_4| \},  \\
	&\Omega_3 = \{\Omega_0: |n_{23}| \gtrsim N,\ |n_4| \gtrsim N \gg |n_1| \},  \qquad
	\Omega_4 = \{\Omega_0: |n_1|,\ |n_4|,\ |n_{23}| \gtrsim N \}. 
\end{split}\end{equation}
In the regions $\Omega_1$, $\Omega_2$ and $\Omega_3$, one of the frequencies $n_1$, $n_{23}$ and $n_4$ is far smaller than $N$ while $\max_{1\leq i\leq 4} |n_i| \gtrsim N$, so these regions are referred to as high-low frequency regions. In contrast, all the frequencies $n_1$, $n_{23}$ and $n_4$ in $\Omega_4$ are at least comparable to $N$, so we call it a high-high frequency region. In this section, we will bound $\mathcal{M}$ pointwise on the high-low frequency regions, while the estimate on the high-high frequency region will be carried out in Section \ref{Subsubsec, est_hh_freq}.

We begin with some notations. 
\begin{itemize}
	
	\item For any $i,j$, we denote $N_i$ and $N_{ij}$ to be the absolute values of the frequencies $n_i$ and $n_{ij}$ respectively:
	\[N_i := |n_i|, \qquad N_{ij} := |n_{ij}| = |n_i+n_j|.\]

	\item Plugging $z=-x-y$ into the function $R$ in (\ref{DefOfRandRI}) yields
	\begin{equation}\label{R_decom1}
	R(x,y,z) = x^3 + \alpha y^3 + \alpha z^3 = -3\alpha x(y+c_1 x)(y+c_2 x),
	\end{equation}
	where $c_1 = \frac12 - \frac{\rho_\alpha}{6}$ and $c_2 = \frac12 + \frac{\rho_\alpha}{6}$ are the two constants in (\ref{c1c2}) such that
	\begin{equation}\label{c1_c2}
	c_1 + c_2 = 1, \qquad \alpha(c_1^3 + c_2^3) = 1.
	\end{equation}
	Since $y$ and $z$ are symmetric in $R(x,y,z)$, it then follows from (\ref{R_decom1}) that 
	\begin{equation}\label{R_decom_z}
	R(x,y,z) = -3\alpha x(z+c_1 x)(z+c_2 x).
	\end{equation}
	
	\item For any $p,q\in \mathbb{Z}_\lambda$, we define 
	\begin{equation}\label{Rp_def}
		R_{p}(q) = -3\alpha p (q+c_1 p)(q+c_2 p) = p^3 + \alpha q^3 - \alpha(p+q)^3.
	\end{equation}
	Then it follows from (\ref{R_decom1}) and (\ref{R_decom_z}) that 
	\begin{equation}\label{R_fun}
		R(x,y,z) = R_{x}(y) = R_{x}(z).
	\end{equation} 
	In particular, $R(n_1, n_{23}, n_{4}) = R_{n_1}(n_{23}) = R_{n_1}(n_4)$.
\end{itemize}

Next, we will justify the pointwise bound estimate (\ref{MultiTrans}) for the regions $\Omega_1$, $\Omega_2$ and $\Omega_3$.

\textbf{Estimate in $\Omega_1$}. We further split $\Omega_1$ into two regions $\Omega_{11}$ and $\Omega_{12}$:
\[
	\Omega_{11} = \{ \Omega_1: |n_1| \ll N \}, \qquad	\Omega_{12} = \{ \Omega_1: |n_1| \gtrsim N \}.
\]
\begin{itemize}
	\item In $\Omega_{11}$, from $N_{23}\ll N$, $N_1\ll N$ and the condition $\sum_{i=1}^{4}n_i = 0$, we find $N_4\ll N$. Consequently, $m(n_1) = m(n_{23}) = m(n_4) = 1$ and $R^{I}(n_1, n_{23}, n_4) = R(n_1, n_{23}, n_4)$, which implies that 
	\[
		|M_{41}| = 2|n_{23}| = 2|n_{14}|,
	\]
	and 
	\[
	\begin{aligned}
		\mathcal{M} \lesssim \frac{|M_{41}|}{\prod_{i=1}^{4} m(n_i)\langle n_i\rangle} &\lesssim \frac{|n_{14}|}{m(n_2)m(n_3)\prod_{i=1}^4\langle n_i \rangle} \lesssim \frac{1}{N_2 m(N_2) N_3 m(N_3)}.
	\end{aligned}
	\]
	Meanwhile, since $\max\limits_{1\leq i\leq 4} |n_i| \gtrsim N$ on $\Omega_0$, $N_2\sim N_3\gtrsim N$. 
	Therefore,
	\[
		\mathcal{M} \le \frac{1}{N_2m(N_2)N_3m(N_3)}\lesssim \frac{N^{2s}}{N_2^sN_3^s N^2}
		\lesssim N^{-2}. 
	\]
	
	\item In $\Omega_{12}$, the relation $n_1+n_{23}+n_4 = 0$ together with $N_{23}\ll N\lesssim N_1$ yields $N_4\sim N_1\gtrsim N$. Consequently, 
	$|R^I(n_1,n_{23},n_4)| \lesssim N_1^3 m^2(n_1)$
	and
	\begin{equation}
		|R(n_1,n_{23},n_4)| = |R_{n_1}(n_{23})| = |3\alpha n_{1}(n_{23}+c_1n_1)(n_{23}+c_2n_1)| \sim N_1^3.
	\end{equation}
	Thus, it follows from (\ref{M_comp}) that
	\[
		\begin{aligned}
			\mathcal{M}
			&\lesssim \frac{N_{23} N_1^3 m^2(n_1)}{N_1^3 \cdot N_1 N_4 m(n_1)m(n_4)\cdot \langle N_2\rangle \langle N_3\rangle  m(n_2)m(n_3) }\\
			&\lesssim \frac{1}{N_1 N_4}\cdot \frac{N_{23}}{\langle N_2\rangle \langle N_3\rangle m(n_2)m(n_3)}
			\lesssim N^{-2}. 
		\end{aligned}
	\]
\end{itemize}

\textbf{Estimate in $\Omega_2$.} We split $\Omega_2$ into three subregions based on the sizes of $|n_2|$ and $|n_3|$:
\[
	\Omega_{21} = \{ \Omega_2: |n_2|, |n_3| \gtrsim N \}, \quad
	\Omega_{22} = \{ \Omega_2: |n_2| \gtrsim N \gg |n_3| \}, \quad
	\Omega_{23} = \{ \Omega_2: |n_3| \gtrsim N \gg |n_2| \}.
\]
Since the indices $2$ and $3$ are symmetric in $\Omega_2$, the estimate for $\Omega_{23}$ follows directly from that for $\Omega_{22}$ by swapping $n_2$ and $n_3$; we therefore only detail the cases $\Omega_{21}$ and $\Omega_{22}$.
\begin{itemize}
	\item In $\Omega_{21}$, the conditions give $N_1\sim N_{23}\gtrsim N\gg N_4$. Hence, 
	\[
		N_1^3m^2(N_1)\sim N_{23}^3m^2(N_{23})\gg N_4^3m^2(N_4),
	\] 
	which implies that
	\[
		|R^I(n_1,n_{23},n_4)|\lesssim N_1^3m^2(N_1),
	\]
	and
	\[
		|R(n_1,n_{23},n_4)| = |R_{n_1}(n_{4})| = |3\alpha 	n_{1}(n_{4}+c_1n_1)(n_{4}+c_2n_1)| \sim N_1^3.
	\]
	Then it follows from (\ref{M_comp}) that 
	\[
		\mathcal{M} \lesssim \frac{N_{23}N_1^3m^2(N_1)}{N_1^3\cdot N_1m(N_1)\cdot N_2N_3m(N_2)m(N_3)\cdot \langle N_4\rangle } 
		\lesssim \frac{1}{N_2N_3m(N_2)m(N_3)}.
	\]
	Meanwhile, since $N_2, N_3 \gtrsim N$, we have 
	\[
		\begin{aligned}
			\mathcal{M} 
			&\lesssim \frac{1}{N_2N_3m(N_2)m(N_3)}\lesssim N^{-2}.
		\end{aligned}
	\]
	
	\item In $\Omega_{22}$, we easily deduce $N_1\sim N_2\sim N_{23}\gtrsim N$ and $N_3, N_4\ll N$. As a result,
	\[
	|R^I(n_1,n_{23},n_4)|\lesssim N_1^3m^2(N_1),\qquad |R(n_1,n_{23},n_4)| \sim N_1^3.
	\]
	Therefore, it follows from (\ref{M_comp}) that 
	\[
			\mathcal{M} \lesssim \frac{N_{23}N_1^3m^2(N_1)}{N_1^3\cdot N_1N_2m(N_1)m(N_2)\cdot \langle N_3\rangle \langle N_4\rangle}
			\lesssim \frac{1}{N_2}\lesssim N^{-1}.
	\]
\end{itemize}

\textbf{Estimate in $\Omega_3$.}
Since $u$ has zero mean, we only need to consider the integral region on which $n_1 \neq 0$. From $N_{23} \sim N_4 \gg N_1$, we then obtain
\begin{equation}\label{0928-0}
	|R(n_1,n_{23},n_4)| = |R_{n_1}(n_4)| = 3\alpha \bigl|n_1(n_4 + c_1 n_1)(n_4 + c_2 n_1)\bigr| \sim N_1 N_4^2.
\end{equation}
On the other hand, $R^I(n_1,n_{23},n_4)$ can be rewritten as 
\begin{equation}\label{0928-1}
\begin{split}
	R^I(n_1,n_{23},n_4) &= n_1^3 m^2(n_1) + \alpha \bigl[n_{23}^3 m^2(n_{23}) + n_4^3 m^2(n_4)\bigr] \\
	&= n_1^3 m^2(n_1) + \alpha \bigl[ n_{23}^3 m^2(n_{23}) - (n_{23}+n_1)^3 m^2(n_{23}+n_1) \bigr].
\end{split}
\end{equation}
Since $N_1 \ll N_{23}$, applying the mean value theorem yields
\[
	\bigl| n_{23}^3 m^2(n_{23}) - (n_{23}+n_1)^3 m^2(n_{23}+n_1) \bigr| \lesssim N^{2-2s} N_{23}^{2s} N_1.
\]
Therefore, it follows from (\ref{0928-1}) that 
\begin{equation}\label{0928-2}
	|R^I(n_1,n_{23},n_4)| \lesssim N_1^3 m^2(N_1) + \alpha N^{2-2s} N_{23}^{2s} N_1 \lesssim N^{2-2s} N_{23}^{2s} N_1.
\end{equation}
Inserting (\ref{0928-0}) and (\ref{0928-2}) into (\ref{M_comp}) gives
\begin{equation}\label{0928-3}\begin{split}
	\mathcal{M} &\lesssim \frac{N_{23} \cdot N^{2-2s} N_{23}^{2s} N_1}{N_1 N_4^2 \cdot m(N_1) \langle N_1\rangle \cdot \langle N_2\rangle  \langle N_3 \rangle  m(N_2) m(N_3) \cdot N_4 m(N_4)} \\
	&\lesssim \frac{N^{2-2s} N_{23}^{2s}}{\langle N_2\rangle \langle N_3\rangle  m(N_2) m(N_3) \cdot N_4^2 m(N_4)},
\end{split}\end{equation}
where in the last inequality we used the fact that $m(N_1)\langle N_1\rangle \gtrsim 1$ for any $N_1$.
By symmetry of $n_2$ and $n_3$ in $\Omega_3$, we may assume $N_2 \ge N_3$, so that $N_2 \gtrsim N$ and
\[
	m(n_2) m(n_3) \langle N_2\rangle \langle N_3\rangle \gtrsim N_2m(N_2) \gtrsim N^{1-s} N_2^s.
\]
Substituting this into (\ref{0928-3}) gives
\begin{equation}\label{0929-01}
	\mathcal{M} \lesssim \frac{N^{2-2s} N_{4}^{2s}}
	{N_4^2 m(N_4) \cdot N^{1-s} N_2^s}
	\lesssim \frac{1}{N_4^{1-s} N_2^{s}}
	\lesssim N^{-1} .
\end{equation}

So far, we finished the desired estimates on the high-low frequency regions $\Omega_1$, $\Omega_2$ and $\Omega_3$. The conclusion on these regions is 
\begin{equation}\label{C1need1}
	\mathcal{M}(n_1,n_2,n_3,n_4) \lesssim N^{-1} \quad \text{on} \quad \Omega_1 \cup \Omega_2 \cup \Omega_3.
\end{equation}
According to the definition of $D_1$ and the fact that $s_\alpha\geq \frac12$, we know
\[
	\frac{D_1(\lambda, N, \epsilon)}{N^{-1}} = N^{-1+2s_\alpha+\epsilon} \lambda^{2s_\alpha+\epsilon} \geq 1.
\]
As a result, (\ref{MultiTrans}) is verified for $A = \Omega_1 \cup \Omega_2 \cup \Omega_3$ thanks to (\ref{C1need1}).

Next, we will turn to treat the high-high frequency region $\Omega_4$.

\subsubsection{Estimates in the high-high frequency regions}
\label{Subsubsec, est_hh_freq}

\textbf{Estimate in $\Omega_4$.} We divide $\Omega_4$ into the following four subregions:
\begin{align*}
	&\Omega_{41} = \{ \Omega_4: |n_1| \sim |n_{23}| \gg |n_4| \gtrsim N \},
	&\Omega_{42} = \{ \Omega_4: |n_1| \sim |n_4| \gg |n_{23}| \gtrsim N \},\\
		&\Omega_{43} = \{ \Omega_4: |n_4| \sim |n_{23}| \gg |n_1| \gtrsim N \},
			&\Omega_{44} = \{ \Omega_4: |n_1| \sim |n_4| \sim |n_{23}| \gtrsim N \}.
\end{align*}
\begin{itemize}
	\item In $\Omega_{41}$ and $\Omega_{42}$, it is readily seen that
	\[
		|R^I(n_1,n_{23},n_4)|\lesssim N_1^3m^2(N_1), \qquad 
		|R(n_1,n_{23},n_4)|\sim N_1^3.
	\]
	Meanwhile, $N_{23} \lesssim N_1$ on $\Omega_{41} \cup \Omega_{42}$, so it follows from (\ref{M_comp}) that 
	\begin{equation}\label{0928-4}
		\mathcal{M} \lesssim \frac{N_{23} N_1^3m^2(N_1)}{N_1^3 \cdot N_1 m(N_1) \cdot \prod_{i=2}^{4} \langle N_i\rangle m(N_i)} \lesssim \frac{1}{\prod_{i=2}^{4} \langle N_i\rangle m(N_i) }.
	\end{equation}
	Since $\max\{N_2, N_3, N_4\} \gtrsim N_1 \gtrsim N$, (\ref{0928-4}) implies that $\mathcal{M}\lesssim N^{-1}$.
	
	\item The region $\Omega_{43}$ is completely analogous to $\Omega_3$ except that $|n_1| \ll N$ in $\Omega_3$ while $|n_1| \gtrsim N$ in $\Omega_{43}$. But the previous argument from (\ref{0928-0}) to (\ref{0929-01}) for the region $\Omega_3$ worked for any $n_1\neq 0$, so the same argument is applicable here on the region $\Omega_{43}$ and we obtain $\mathcal{M} \lesssim N^{-1}$.
\end{itemize}
The estimates on $\Omega_{41}$, $\Omega_{42}$ and $\Omega_{43}$ can be stated as 
\[
	\mathcal{M}(n_1,n_2,n_3,n_4) \lesssim N^{-1} \quad \text{on} \quad \Omega_{41} \cup \Omega_{42} \cup \Omega_{43},
\]
which confirms (\ref{MultiTrans}) with $A = \Omega_{41} \cup \Omega_{42} \cup \Omega_{43}$.

Next, we will tackle the most challenging region $\Omega_{44}$ in which the resonance may occur.

\textbf{Estimate in $\Omega_{44}$.} 
Since $N_1\sim N_4\sim N_{23}$ in $\Omega_{44}$, for $R^I(n_1,n_{23},n_4)$ we have
\begin{equation}\label{BoundofRI}
	|R^I(n_1,n_{23},n_4)| \lesssim N_1^3 m^2(N_1),
\end{equation}
and
\begin{equation}\label{M44}\begin{split}
    \mathcal{M}(n_1,n_2,n_3,n_4) &\lesssim \frac{N_1^4 m^2(N_1)}{ |R(n_1,n_{23},n_4) | \prod_{i=1}^{4} (m(n_i)\langle n_i\rangle)} \\
    &\lesssim \frac{N_1^2}{|R(n_1,n_{23},n_4)| \langle N_2\rangle \langle N_3\rangle m(N_2) m(N_3)}.
\end{split}\end{equation}
The most challenging part in $\Omega_{44}$ is the resonant region $\Delta$ defined as in (\ref{res_set}). We further decompose $\Delta$ into two subsets: $\Delta = \Delta_1 \cup \Delta_2$, where
\begin{equation}\label{Delta_j}
	\Delta_j := \left\{ (n_1,n_2,n_3,n_4)\in \Gamma_4^\lambda: |n_4 + c_j n_1| \leq \frac{\kappa_\alpha}{\lambda} \right\}, \qquad 
	j=1,2.
\end{equation}
Here $\kappa_\alpha$ is the constant introduced in \eqref{eq:bad-region-def}. Since the frequencies in $\Omega_{44}$ satisfy $|n_1|\sim |n_4|\sim |n_{23}|$ which does not distinguish $|n_2|$ and $|n_3|$, we further divide $\Omega_{44}$ into the following subregions:
\begin{align*}
	&\qquad \qquad \qquad \qquad\qquad \mathcal{O}_0 := \Omega_{44} \setminus\Delta,\\[4pt]
	\mathcal{O}_{1}  &:= \Big\{\Omega_{44} \cap \Delta_1: |n_2|\sim |n_3| \Big\},  
	\qquad \qquad\,\,\, \widetilde{\mathcal{O}}_1 := \Big\{\Omega_{44} \cap \Delta_2: |n_2|\sim |n_3|\Big\}, \\[4pt]
	\mathcal{O}_{2}  &:= \Big\{\Omega_{44} \cap \Delta_1: |n_2|\gg |n_3| \Big\},
	\qquad\qquad\, \widetilde{\mathcal{O}}_2 := \Big\{\Omega_{44} \cap \Delta_2: |n_2|\gg |n_3|\Big\}, \\[4pt]
	\mathcal{O}_{3}  &:= \Big\{\Omega_{44} \cap \Delta_1: |n_3|\gg |n_2|\Big\},
	\qquad\qquad\, \widetilde{\mathcal{O}}_3 := \Big\{\Omega_{44} \cap \Delta_2: |n_3|\gg |n_2|\Big\}.
\end{align*}
Since $\widetilde{\mathcal{O}}_{1}$, $\widetilde{\mathcal{O}}_{2}$ and
$\widetilde{\mathcal{O}}_3$ are analogous to
$\mathcal{O}_{1}$, $\mathcal{O}_{2}$ and $\mathcal{O}_{3}$, respectively, it suffices to consider $\mathcal{O}_{0}$, $\mathcal{O}_{1}$, $\mathcal{O}_{2}$ and $\mathcal{O}_{3}$ in the sequel.

\begin{itemize}
	\item \textbf{Estimate in $\mathcal{O}_0$.} 
	This region does not have resonance as $|n_4 + c_j n_1| > \kappa_\alpha/\lambda$ for $j=1,2$. Meanwhile,
	\[
		|n_4 + c_1 n_1| + |n_4 + c_2 n_1| \geq |c_2 - c_1| |n_1| \sim N_1.
	\]
	As a result, 
	\[
		 |R(n_1,n_{23},n_4)| = |R_{n_1}(n_4)| = 3\alpha |n_1||n_4+c_1n_1||n_4+c_2n_1| \gtrsim \frac{N_1^2}{\lambda}.
	\]
    Combining the above estimate with (\ref{M44}), we obtain 
    \begin{equation}\label{0929-1}
        \mathcal{M} \lesssim \frac{\lambda}{\langle N_2\rangle \langle N_3\rangle m(N_2) m(N_3)}.
    \end{equation}
    Since $N_2 + N_3 \geq N_{23} \gtrsim N$, at least one of $N_2$ or $N_3$ is at least comparable to $N$. Thus, it follows from (\ref{0929-1}) that 
	\begin{equation}\label{C1need3}
		\mathcal{M}\lesssim  \lambda N^{-1} \quad\text{on}\quad \mathcal{O}_0.
	\end{equation}
	Again, based on the definition of $D_1$ and the fact that $s_\alpha\geq \frac12$, we have
	\[
	\frac{D_1(\lambda, N, \epsilon)}{\lambda N^{-1}} = N^{-1+2s_\alpha+\epsilon} \lambda^{2s_\alpha+\epsilon - 1} \geq 1.
	\]
	So (\ref{MultiTrans}) is proved with $A = \mathcal{O}_0$ because of (\ref{C1need3}).

	\item \textbf{Estimate in $\mathcal{O}_1$.} 
	In this region, 
	\[
		|n_4 + c_2 n_1| \geq |c_2-c_1| |n_1| - |n_4 + c_1 n_1| \geq \frac{\rho_\alpha}{3} N_1 - \frac{\kappa_\alpha}{\lambda} \gtrsim N_1,
	\]
	where we used the fact that $\rho_\alpha > 0$, i.e. $\alpha\neq 4$, due to the assumption $\alpha\in\mathcal{A}$. On the other hand, 
	\[
		|n_4 + c_1 n_1| = N_1 \Big| c_1 + \frac{n_4}{n_1} \Big| = N_1 \Big| c_1 - \frac{-\lambda n_4}{\lambda n_1} \Big|,
	\]
	where both $\lambda n_4$ and $\lambda n_1$ are integers since $n_1, n_4\in\mathbb{Z}_\lambda$.
	Based on Definition \ref{Def,mbti}, for any $\epsilon > 0$, there exists some constant $K>0$, depending only on $c_1$ and $\epsilon$, such that 
	\[
		\Big| c_1 - \frac{-\lambda n_4}{\lambda n_1} \Big| \gtrsim \frac{K}{(\lambda N_1)^{2+\nu(c_1)+\epsilon}} =  \frac{K}{(\lambda N_1)^{2+\nu(\rho_\alpha)+\epsilon}},
	\]
	where we used the fact that $\nu(c_1) = \nu(\rho_\alpha)$ in the last equality. The combination of the above three relations yields 
	\[
		|n_4 + c_2 n_1|  |n_4 + c_1 n_1| \gtrsim \lambda^{-2-\nu(\rho_\alpha)-\epsilon} N_1^{-\nu(\rho_\alpha)-\epsilon}.
	\]
	So 
	\[
		|R(n_1,n_{23},n_4)| = |R_{n_1}(n_4)| = 3\alpha |n_1||n_4 + c_2 n_1| |n_4 + c_1 n_1|
		\gtrsim \lambda^{-2-\nu(\rho_\alpha)-\epsilon} N_1^{1-\nu(\rho_\alpha)-\epsilon}.
	\]
	Since $s_{\alpha} = \frac{1+\nu(\rho_\alpha)}{2}$, this yields
	\begin{equation}\label{resonance_R}
		|R(n_1,n_{23},n_4)| \gtrsim \lambda^{-1-2s_{\alpha}-\epsilon} N_1^{2-2s_{\alpha}-\epsilon}.
	\end{equation}
	Plugging (\ref{resonance_R}) into (\ref{M44}) gives 
	\begin{equation}\label{0929-2}
		\mathcal{M} \lesssim \frac{\lambda^{1+2s_{\alpha}+\epsilon} N_1^{2s_\alpha+\epsilon}}{\langle N_2\rangle \langle N_3\rangle m(N_2) m(N_3)}.
	\end{equation}
	Since $N_1 \sim N_{23}\gtrsim N$ and $N_2\sim N_3$, it holds that $N_2 \gtrsim N_1$ and $N_3 \gtrsim N_1$. Consequently, 
	\[
		\langle N_2\rangle \langle N_3\rangle m(N_2) m(N_3) \sim N_2^2 m^2(N_2) \gtrsim N_1^2 m^2(N_1) \sim N^{2-2s} N_1^{2s}.	
	\]
	Then we infer from (\ref{0929-2}) that 
	\begin{equation}\label{0929-3}
		\mathcal{M} \lesssim \frac{\lambda^{1+2s_{\alpha}+\epsilon}}{N^{2-2s} N_1^{2s-2s_\alpha-\epsilon}}.
	\end{equation}
	Thanks to the assumption $0<\epsilon < 2s-2s_\alpha$ in Lemma \ref{M41Estimate}, we know $N_1^{2s-2s_\alpha-\epsilon} \gtrsim N^{2s-2s_\alpha-\epsilon}$. Hence, it follows from (\ref{0929-3}) that 
	\begin{equation}\label{C1need4}
		\mathcal{M} \lesssim \frac{\lambda^{1+2s_{\alpha}+\epsilon}}{ N^{2-2s_\alpha-\epsilon}} = \lambda D_1(\lambda, N, \epsilon) \quad\text{on}\quad \mathcal{O}_1,
	\end{equation}
which justifies the pointwise bound (\ref{MultiTrans_weak}) with $B=\mathcal{O}_1 \subset \Omega_0\cap \Delta$.

\item \textbf{Estimate in $\mathcal{O}_2$.} 
Firstly, it follows from \eqref{M44} and \eqref{resonance_R} that
\begin{equation}\label{MO2}
	\mathcal{M}\lesssim \frac{N_1^2}{\lambda^{-2s_{\alpha}-1-\epsilon} N_1^{2-2s_{\alpha}-\epsilon} \cdot \langle N_2\rangle\, m(N_2)}\lesssim \frac{\lambda^{1+2s_{\alpha}+\epsilon} N_1^{2s_\alpha+\epsilon}}{\langle N_2 \rangle\, m(N_2)}.
\end{equation}
Then in the present case, it is readily seen that $N_2 \sim N_{23}\gtrsim N$ and $N_1\sim N_{23}\sim N_2$. As a result, 
we infer from (\ref{MO2}) that 
\begin{equation}\label{1004-1}
	\mathcal{M}\lesssim \frac{\lambda^{1+2s_{\alpha}+\epsilon} N_1^{2s_\alpha+\epsilon}}{N_1\, m(N_1)} \lesssim 
	\frac{\lambda^{1+2s_{\alpha}+\epsilon} N_1^{2s_\alpha+\epsilon-s}}{N^{1-s}}. 
\end{equation}

Since $n_{23} = - n_{14}$, $H_\alpha$ in (\ref{eq:quartic-resonance-H}) can be rewritten as
\begin{equation}\label{HalphaExp}
	\begin{aligned}
		H_\alpha(n_1,n_2,n_3,n_4) &= (n_1^3+\alpha n_4^3 +\alpha n_{23}^3) +( n_2^3 + \alpha n_3^3+\alpha n_{14}^3)\\
		& = R_{n_1}(n_4) + R_{n_2}(n_3).
	\end{aligned}
\end{equation}
Moreover, since $|n_4 + c_1 n_1| \leq \kappa_\alpha/\lambda$, it follows from the decomposition (\ref{Rp_def}) that 
\begin{equation}\label{1004-4}
	|R_{n_1}(n_4)| = 3\alpha | n_1(n_4 + c_1 n_1)(n_4 + c_2 n_1)| \lesssim \frac{\kappa_\alpha N_1^2}{\lambda},
\end{equation}
and
\[
	|R_{n_2}(n_3)| = |-3\alpha n_2(n_3+c_1n_2)(n_3+c_2n_2)|\gtrsim N_2^3.
\]
Consequently,
\begin{equation}\label{HO2}
	|H_\alpha|\gtrsim |R_{n_2}(n_3)|-|R_{n_1}(n_4)|\gtrsim N_2^3 -\frac{ \kappa_\alpha N_1^2}{\lambda} \gtrsim N_1^3.
\end{equation}
Based on (\ref{1004-1}), (\ref{HO2}) and the fact that $N_1\gtrsim N$, we find
\[
	\frac{\mathcal{M}}{\lambda \langle H_\alpha\rangle^{1/2}} \lesssim \frac{\lambda^{1+2s_\alpha+\epsilon} N_1^{2s_\alpha+\epsilon-s}}{N^{1-s}\lambda N_1^{3/2}} 
	\lesssim \frac{\lambda^{2s_\alpha+\epsilon}}{N^{\frac52 - 2s_\alpha - \epsilon}} \lesssim D_1(\lambda, N, \epsilon).
\]
Hence, we proved the pointwise bound (\ref{MultiTrans_weak}) with $B=\mathcal{O}_2 \subset \Omega_0\cap \Delta$.

\item \textbf{Estimate in $\mathcal{O}_3$.} 
According to (\ref{M44}) and \eqref{resonance_R}, we obtain 
\begin{equation}\label{1004-2}
	\mathcal{M} \lesssim \frac{N_1^2}{\lambda^{-1-2s_\alpha-\epsilon} N_1^{2-2s_\alpha-\epsilon} \langle N_3\rangle m(N_3)} \lesssim \frac{\lambda^{1+2s_\alpha+\epsilon} N_1^{2s_\alpha+\epsilon}}{\langle N_3\rangle m(N_3)}.
\end{equation}
In $\mathcal{O}_3$, we have $N_3\gtrsim N_{23}\gtrsim N$ and 
$N_1\sim N_{23} \sim N_3 \gtrsim N$. Then it follows from (\ref{1004-2}) that 
\begin{equation}\label{1004-3}
	\mathcal{M} \lesssim \frac{\lambda^{1+2s_\alpha+\epsilon} N_1^{2s_\alpha+\epsilon}}{\langle N_1\rangle m(N_1)} \lesssim \frac{\lambda^{1+2s_\alpha+\epsilon} N_1^{2s_\alpha+\epsilon-s}}{N^{1-s}}.
\end{equation}
Similar to (\ref{HalphaExp}) and (\ref{1004-4}), we know
\[H_\alpha(n_1,n_2,n_3,n_4) = R_{n_1}(n_4) + R_{n_2}(n_3),\]
where $|R_{n_1}(n_4)| \lesssim \frac{N_1^2}{\lambda}$. Thus, it remains to estimate $R_{n_2}(n_3)$. 

Since $u^\lambda$ has zero mean, we may restrict to $n_2\ne0$.
The exclusion of $\Omega_*$ then implies $N_2>1$ on $O_3$. As a result,
\begin{equation}\label{1004-5}
	|R_{n_2}(n_3)| = |-3\alpha n_2(n_3+c_1n_2)(n_3+c_2n_2)|\gtrsim N_3^2.
\end{equation}
Consequently, $|H_\alpha|$ admits the lower bound
\begin{equation}\label{HO3}
	|H_\alpha| \geq |R_{n_2}(n_3)|-|R_{n_1}(n_4)|\gtrsim N_3^2 -\frac{ N_1^2}{\lambda} \gtrsim N_1^2,
\end{equation}
where the last inequality is due to $N_1\sim N_3$ and $\lambda \gg 1$.
Based on (\ref{1004-3}) and (\ref{HO3}), we conclude that 
\[
	\frac{\mathcal{M}}{\lambda \langle H_\alpha\rangle^{1/2}} \lesssim \frac{\lambda^{2s_\alpha+\epsilon}}{N^{1-s} N_1^{1+s-2s_\alpha-\epsilon}}
	\lesssim \frac{\lambda^{2s_\alpha+\epsilon}}{N^{2 - 2s_\alpha - \epsilon}} = D_1(\lambda, N, \epsilon),
\]
where we used the fact that $1+s-2s_\alpha-\epsilon > 0$ and $N_1\gtrsim N$.
Hence, the pointwise bound (\ref{MultiTrans_weak}) is verified with $B=\mathcal{O}_3 \subset \Omega_0\cap \Delta$.
\end{itemize}

Combining all the above cases together finishes the proof of Lemma \ref{M41Estimate}.

\subsubsection{Necessity of the partial correction}
\begin{remark}\label{remark:O4} 
If we do not perform the partial correction and only consider the second modified energy $E^{(2)}$, then we need to deal with the region $\Omega_*$ which can be decomposed as $\Omega_* = \mathcal{O}_4 \cup \widetilde{\mathcal{O}}_4$, where 
\[
	\mathcal{O}_{4} := \Big\{\Delta_1: N_1 \gtrsim N, \, \frac{1}{\lambda} \leq N_2 \leq 1 \Big\},
	\qquad
	\widetilde{\mathcal{O}}_4 := \Big\{\Delta_2: N_1 \gtrsim N, \, \frac{1}{\lambda} \leq N_2 \leq 1\Big\},
\]
and $\Delta_1$ and $\Delta_2$ are as defined in (\ref{Delta_j}).
Below we only discuss
$\mathcal{O}_4$, since the estimate for $\widetilde{\mathcal{O}}_4$ is analogous.
In $\mathcal{O}_4$, we have $N_1\sim N_3\sim N_4 \gtrsim N$. Then it follows from \eqref{M44} and \eqref{resonance_R} that 
\begin{equation}\label{MO4}
	\mathcal{M} \lesssim \frac{\lambda^{1 + 2 s_\alpha + \epsilon} N_1^{2s_\alpha+\epsilon} }{N_3 m(N_3)} \lesssim \frac{\lambda^{1 + 2 s_\alpha + \epsilon} N_1^{2s_\alpha+\epsilon-s}}{N^{1-s}}.
\end{equation}
Since $\frac{1}{\lambda} \leq N_2 \leq 1$, the previous estimate (\ref{1004-5}) needs to be modified as 
$|R_{n_2}(n_3)| \gtrsim N_3^2/\lambda$.
Meanwhile, it follows from (\ref{1004-4}) that 
$|R_{n_1}(n_4)| \lesssim \kappa_\alpha N_1^2/\lambda$. Then according to Lemma \ref{lemma:sigma4} and the fact that $\mathcal{O}_4\subset \Omega_*$, we have
\begin{equation}\label{HO4}
	|H_\alpha(n_1,n_2,n_3,n_4)| \geq \frac{\alpha N_3^2}{4\lambda} \gtrsim \frac{N_1^2}{\lambda}.
\end{equation}
Based on (\ref{MO4}) and (\ref{HO4}), we obtain
\[
	\frac{\mathcal{M}}{\lambda \langle H_\alpha\rangle^{1/2}} \lesssim \frac{\lambda^{\frac12+2s_\alpha+\epsilon}}{N^{1-s} N_1^{1+s-2s_\alpha-\epsilon}}
	\lesssim \frac{\lambda^{\frac12+2s_\alpha+\epsilon}}{N^{2 - 2s_\alpha - \epsilon}} = \lambda^{\frac12} D_1(\lambda, N, \epsilon).
\]
Then the estimate (\ref{constant51}) in Lemma \ref{M41Estimate} can be modified as 
\begin{equation}\label{1004-7}
	\left| \int_0^1 \Lambda_4(M_{41}; u^\lambda, u^\lambda, v^\lambda, v^\lambda) dt \right| \le C_{10}\,\widetilde{D_1}(\lambda,N,\epsilon) \|Iu^\lambda\|_{X^{1,\frac12}_{\lambda}}^2 \|Iv^\lambda\|_{X^{1,\frac12}_{\alpha,\lambda}}^2,
\end{equation}
where 
\begin{equation}\label{1004-6}
	\widetilde{D_1}(\lambda,N,\epsilon) := \lambda^{\frac12} D_1(\lambda, N, \epsilon).
\end{equation}
Comparing with $D_1(\lambda,N,\epsilon)$, $\widetilde{D_1}(\lambda,N,\epsilon)$ has an additional factor $\lambda^{\frac12}$.

Using the second modified energy $E^{(2)}$, the global iteration requires
\[
\lambda^3 T^* \widetilde D_1(\lambda,N,\epsilon)
\bigl(\varepsilon^I_{\lambda,1}\bigr)^2 \ll 1.
\]
Choosing $\lambda$ as in \eqref{RelationofLambdaN} and taking $\delta=\epsilon/4$,
the local smallness condition gives
$\varepsilon^I_{\lambda,1}\lesssim\lambda^{-(s_\alpha+\delta)}$.
Thus, a sufficient condition for the iteration is
\[
T^*\lambda^{\frac72+\frac{\epsilon}{2}}
N^{-2+2s_\alpha+\epsilon}\ll 1.
\]
After substituting the choice \eqref{RelationofLambdaN} of $\lambda$, the exponent of $N$ is
\[
F(\epsilon)
:=
\left(\frac72+\frac{\epsilon}{2}\right)
\frac{2-2s}{3+2s-2s_\alpha-\epsilon}
-2+2s_\alpha+\epsilon.
\]
Since
\[
F(0)
=
\frac{-4s_\alpha^2+10s_\alpha+1-(11-4s_\alpha)s}
{3+2s-2s_\alpha},
\]
we have $F(0)<0$ whenever
\begin{equation}\label{eq:ss0}
	s >  \frac{-4s_{\alpha}^2+10s_{\alpha}+1}{11-4s_{\alpha}}.
\end{equation}
Under this strict inequality, we may choose
$\epsilon\in(0,2s-2s_\alpha)$ sufficiently small so that
$F(\epsilon)<0$. Taking $N$ sufficiently large then closes the
global iteration. Therefore, without the partial correction,
these estimates yield global well-posedness of the Majda-Biello systems (\ref{GeneralMajdaBiello}) for $s$ in \eqref{eq:ss0}.
\end{remark}

After establishing the desired result for the integral involving the multiplier $M_{41}(1-\chi_{\Omega_*})$ in Lemma \ref{M41Estimate}, we turn to treat the integral involving the multiplier $M_{42}$ in Lemma \ref{M42Estimate}. 

\subsection{Proof of Lemma \ref{M42Estimate}}

The proof of Lemma~\ref{M42Estimate} follows the analogous decomposition as that for Lemma~\ref{M41Estimate}. We retain the notation $\Omega_{-1}$ in (\ref{Omega_neg1}) and the notations $R^{I}$ and $R$ in (\ref{DefOfRandRI}) from the previous section. So it follows from \eqref{Multiplier42} that
\begin{equation}\label{M42_comp}
	M_{42}(n_1,n_2,n_3,n_4) = \frac i2 n_{14} \frac{R^I(n_{14},n_{2},n_3)}{R(n_{14},n_{2},n_3)} , \qquad \forall\, (n_1,n_2,n_3,n_4)\in \Gamma_4^{\lambda}.
\end{equation}
Moreover, in the region $\Omega_{-1}$, we have $m(n_i)=1$ for any $i \in \{1,2,3,4\}$, which guarantees that
\begin{equation*}
		M_{42} = \frac i2 n_{14} \quad\text{on}\quad \Omega_{-1}.
\end{equation*}
By the change of variables $(n_1,n_2,n_3,n_4)\mapsto(n_2,n_1,n_4,n_3)$,  we may average the original integral with its relabeled form.
Consequently, the contribution from $\Omega_{-1}$ to
$\Lambda_4(M_{42};v^\lambda,v^\lambda,v^\lambda,v^\lambda)$ is
\[
	\int_{ \Omega_{-1}}  M_{42} \,\, \widehat v^\lambda\,\widehat v^\lambda \,\widehat v^\lambda\,\widehat v^\lambda = \frac i2\int_{ \Omega_{-1}}  n_{14} \,\, \widehat v^\lambda\,\widehat v^\lambda \,\widehat v^\lambda\,\widehat v^\lambda = \frac i4\int_{\Omega_{-1}} [n_{14}+n_{23}]\,\, \widehat v^\lambda\,\widehat v^\lambda \,\widehat v^\lambda\,\widehat v^\lambda = 0.
\]
Therefore, it remains to prove the estimate \eqref{constant52} only on  
\[
	\Omega_0' := \Gamma_4^\lambda\setminus\Omega_{-1} = \Big\{(n_1,n_2,n_3,n_4)\in \Gamma_4^\lambda: \max_i |n_i|\gtrsim N \Big\}.
\]
That is to prove 
\begin{equation}\label{1005-1}
	\left| \int_{\mathbb{R}} \Lambda_4(M_{42}\chi_{\Omega_0'}; v^\lambda, v^\lambda, v^\lambda, v^\lambda) \,dt \right|
	\le C_{11} D_1(\lambda,N,\epsilon)  \|Iv^\lambda\|_{X^{1,\frac12}_{\alpha,\lambda}}^4.
\end{equation}
Similar to the derivation from (\ref{0929-6}) to (\ref{1003-4}), we find (\ref{1005-1}) is reduced to the following estimate:
\begin{equation}\label{1005-2}
	\int_{\sum\limits_{i=1}^{4} \tau_i=0}\int_{\Omega_0'} \mathcal{M}'(n_1,n_2,n_3,n_4)
	\frac{\prod_{i=1}^4 |\widetilde f_i(n_i,\tau_i)|}
	{\prod_{i=1}^4\langle L_i\rangle^{1/2}}
	\lesssim D_1(\lambda, N,\epsilon)
	\prod_{i=1}^4\|f_i\|_{L^2(\mathbb{T}_\lambda\times \mathbb{R})},
\end{equation}
where $\widetilde{f_i}$ represents the space-time Fourier transform of $f_i$ and $L_i = \tau_i - \alpha n_i^3$ for any $i\in\{1,2,3,4\}$, and
\begin{equation}\label{M_prime}
	\mathcal{M}'(n_1,n_2,n_3,n_4) := \frac{|M_{42}(n_1,n_2,n_3,n_4)|}{\prod_{i=1}^{4} \langle n_i\rangle m(n_i)}.
\end{equation}
Without loss of generality, we may assume all $\{\widetilde{f}_i\}_{i=1}^{4}$ are nonnegative functions. Next, we will prove (\ref{1005-2}) by decomposing $\Omega_0'$ further.

Comparing the expression (\ref{M42_comp}) for $M_{42}$ with the expression (\ref{M41_comp}) for $M_{41}$, the denominator for $M_{42}$ is $R(n_{14}, n_2, n_3)$ while that for $M_{41}$ is $R(n_1, n_{23}, n_{4})$. Since the decomposition in the proof of Lemma \ref{M41Estimate} is based on $\{N_1, N_{23}, N_4\}$, we will split the region $\Omega_0'$ based on $\{N_{14}, N_2, N_3\}$. Specifically, we partition $\Omega_0'$ into four regions:
\begin{align*}
	\Omega'_1 &= \{ \Omega_0': |n_{2}| \ll N \}, 
	& \Omega'_2 &= \{ \Omega_0': |n_{2}| \gtrsim N,\ |n_{14}| \gtrsim N \gg |n_3| \},  \\
	\Omega'_3 &= \{ \Omega_0': |n_{2}| \gtrsim N,\ |n_3| \gtrsim N \gg |n_{14}| \}, 
	& \Omega'_4 &= \{ \Omega_0': |n_{14}|,\ |n_2|,\ |n_{3}| \gtrsim N \}.
\end{align*}
The treatment for the regions $\Omega_1'$, $\Omega_2'$ and $\Omega_3'$ is similar to that for the regions $\Omega_1$, $\Omega_2$ and $\Omega_3$ in Section \ref{Subsubsec, M41_high-low}, so we omit the details and the remaining region is $\Omega_4'$. Again, we partition $\Omega_4'$ into four subregions according to the size of $\{N_{14}, N_2, N_3\}$:
\begin{align*}
	\Omega_{41}' &= \{ \Omega_4': |n_{14}| \sim |n_{2}| \gg |n_3| \gtrsim N \}, &
	\Omega_{42}' &= \{ \Omega_4': |n_{14}| \sim |n_3| \gg |n_{2}| \gtrsim N \},\\
	\Omega_{43}' &= \{ \Omega_4': |n_3| \sim |n_{2}| \gg |n_{14}| \gtrsim N \}, &
	\Omega_{44}' &= \{ \Omega_4': |n_{14}| \sim |n_2| \sim |n_{3}| \gtrsim N \}.
\end{align*}
We only consider the region $\Omega_{44}'$, since the remaining three regions $\Omega_{41}'$, $\Omega_{42}'$, and $\Omega_{43}'$ can be treated in a similar manner as that for the 
regions $\Omega_{41}$, $\Omega_{42}$ and $\Omega_{43}$ in Section \ref{Subsubsec, est_hh_freq}.

\textbf{Estimate in $\Omega_{44}'$.}
In this region, we have $N_{14}\sim N_2\sim N_3$, so
\begin{equation}\label{BoundofRI'}
	|R^I(n_{14},n_{2},n_3)| \lesssim N_{14}^3 m^2(N_{14}).
\end{equation}
Then the combination of (\ref{M_prime}), (\ref{M42_comp}) and (\ref{BoundofRI'}) yields 
\begin{equation}\label{M44'}
	\mathcal{M}'(n_1,n_2,n_3,n_4)\lesssim \frac{N_{14}^2}{\big|R(n_{14},n_{2},n_3)\big| \langle N_1\rangle \langle N_4\rangle m(N_1) m(N_4)}.
\end{equation}
From (\ref{M44'}), we see that the estimate on $\Omega_{44}'$ is challenging since it involves the resonant set $\Delta'$ (see  (\ref{res_set})) on which the denominator $|R(n_{14}, n_2, n_3)|$ in (\ref{M44'}) can be small. By splitting $\Delta'$ into two subsets: $\Delta' = \Delta_1' \cup \Delta_2'$, where 
\begin{align*}
	\Delta_1' &:= \bigl\{ (n_1,n_2,n_3,n_4)\in\Gamma_4^\lambda: |n_{2} + c_1 n_{14}| \leq \tfrac{\kappa_\alpha}{\lambda} \bigr\}, &
	\Delta_2' &:= \bigl\{ (n_1,n_2,n_3,n_4)\in\Gamma_4^\lambda: |n_2 + c_2 n_{14}| \leq \tfrac{\kappa_\alpha}{\lambda} \bigr\},
\end{align*}
we further decompose $\Omega_{44}'$ into seven subregions $\{\mathcal{O}_0', \mathcal{O}_1', \mathcal{O}_2', \mathcal{O}_3', \widetilde{\mathcal{O}}_1', \widetilde{\mathcal{O}}_2', \widetilde{\mathcal{O}}_3'\}$, where 
$\mathcal{O}_0' := \Omega_{44}' \setminus \Delta'$,
\begin{align*}
	\mathcal{O}_{1}' &:= \Big\{\Omega_{44}' \cap \Delta_1': |n_1|\sim |n_4| \Big\},  
	\qquad \qquad\,\,\, \widetilde{\mathcal{O}}_1' := \Big\{\Omega_{44}' \cap \Delta_2': |n_1|\sim |n_4|\Big\}, \\[4pt]
	\mathcal{O}_{2}'  &:= \Big\{\Omega_{44}' \cap \Delta_1': |n_1|\gg |n_4| \Big\},
	\qquad\qquad\, \widetilde{\mathcal{O}}_2' := \Big\{\Omega_{44}' \cap \Delta_2': |n_1|\gg |n_4|\Big\}, \\[4pt]
	\mathcal{O}_{3}' &:= \Big\{\Omega_{44}' \cap \Delta_1': |n_4|\gg |n_1|\Big\},
	\qquad\qquad\, \widetilde{\mathcal{O}}_3' := \Big\{\Omega_{44}' \cap \Delta_2': |n_4|\gg |n_1|\Big\}.
\end{align*}
Since $\widetilde{\mathcal{O}}_{1}'$, $\widetilde{\mathcal{O}}_{2}'$ and
$\widetilde{\mathcal{O}}_3'$ are analogous to
$\mathcal{O}_{1}'$, $\mathcal{O}_{2}'$ and $\mathcal{O}_{3}'$, respectively, it suffices to consider $\mathcal{O}_{0}'$, $\mathcal{O}_{1}'$, $\mathcal{O}_{2}'$ and $\mathcal{O}_{3}'$ in the sequel. The treatment for the first two regions is analogous to that for the regions $\mathcal{O}_0$ and $\mathcal{O}_1$ in Section \ref{Subsubsec, est_hh_freq}. Moreover, the variables $n_1$ and $n_4$ are symmetric, so the proof for the region $\mathcal{O}_2'$ is parallel to that for the region $\mathcal{O}_3'$. Hence, it remains to investigate the estimate on the region $\mathcal{O}_3'$.

\textbf{Estimate in $\mathcal{O}_3'$.}
Recalling the proof for the region $\mathcal{O}_3$ in Section \ref{Subsubsec, est_hh_freq}, the property that $N_2\geq 1$ played an essential role, and this property was due to the exclusion of $\Omega_*$. But in the current situation, we didn't exclude $\Omega_*$, so the proof might be quite different. The key advantage in the current case is the consistent dispersion coefficient $\alpha$ for the resonance function $H_\alpha'$ in (\ref{eq:quartic-resonance-Hprime}), which makes $H_\alpha'$ much simpler than $H_\alpha$ which is associated with different dispersion coefficients $1$ and $\alpha$ in (\ref{eq:quartic-resonance-H}). By restricting the estimate (\ref{1005-2}) to the region $\mathcal{O}_3'$, it reduces to proving 
\begin{equation}\label{1006-1}
	\int_{\sum\limits_{i=1}^{4} \tau_i=0}\int_{\mathcal{O}_3'} \mathcal{M}'(n_1,n_2,n_3,n_4)
	\frac{\prod_{i=1}^4\widetilde f_i(n_i,\tau_i)}
	{\prod_{i=1}^4\langle L_i\rangle^{1/2}}
	\lesssim D_1(\lambda, N,\epsilon)
	\prod_{i=1}^4\|f_i\|_{L^2(\mathbb{T}_\lambda\times \mathbb{R})}.
\end{equation}
Since $\mathcal{O}_3'\subset \Delta'$, we will take advantage of the estimate (\ref{eq:resonant-quadrilinear-42}) in Lemma \ref{lem:resonant-quadrilinear} to prove (\ref{1006-1}). More precisely, if we can prove that 
\begin{equation}\label{1006-2}
	\mathcal{M}' \lesssim \lambda\langle H_\alpha'\rangle^{\frac12} D_1(\lambda,N,\epsilon) \quad\text{on}\quad \mathcal{O}_3' \subset \Delta',
\end{equation}
then the combination of (\ref{1006-2}) and (\ref{eq:resonant-quadrilinear-42}) yields 
\[\begin{split}
	\text{LHS of (\ref{1006-1})} &\lesssim \int_{\sum\limits_{i=1}^{4} \tau_i=0}\int_{\mathcal{O}_3'} \lambda\langle H_\alpha'\rangle^{\frac12} D_1(\lambda,N,\epsilon) 
	\frac{\prod_{i=1}^4\widetilde f_i(n_i,\tau_i)}
	{\prod_{i=1}^4\langle L_i\rangle^{1/2}} \\
	& \leq \lambda D_1(\lambda,N,\epsilon) \int_{\sum\limits_{i=1}^{4} \tau_i=0}\int_{\Delta'} \langle H_\alpha'\rangle^{\frac12} \frac{\prod_{i=1}^4\widetilde f_i(n_i,\tau_i)}
	{\prod_{i=1}^4\langle L_i\rangle^{1/2}} \\
	&\lesssim D_1(\lambda, N,\epsilon)
	\prod_{i=1}^4\|f_i\|_{L^2(\mathbb{T}_\lambda\times \mathbb{R})},
\end{split}\]
which justifies (\ref{1006-1}). Hence, it boils down to proving (\ref{1006-2}).

According to (\ref{eq:quartic-resonance-Hprime}) and the fact that $n_{14} = -n_{23}$, we have  
\begin{equation}\label{1007-1}
	H_\alpha' = \alpha(n_1^3+n_2^3+n_3^3+n_4^3) 
	= \underbrace{(n_{14}^3 + \alpha n_2^3 + \alpha n_3^3)}_{R_{n_{14}}(n_2) } + \underbrace{(n_{23}^3 + \alpha n_1^3 + \alpha n_4^3)}_{R_{n_{23}}(n_1)},
\end{equation}
where we adopt the definition (\ref{Rp_def}) for $R_{p}(q)$. According to the factorization in (\ref{Rp_def}), we have 
\begin{align*}
	R_{n_{14}}(n_2) &= -3\alpha n_{14}(n_2+c_1 n_{14})(n_2+c_2 n_{14}), \\
	R_{n_{23}}(n_1) &= -3\alpha n_{23}(n_1+c_1 n_{23})(n_1+c_2 n_{23}).
\end{align*}
Since $|n_2 + c_1 n_{14}| \leq \kappa_\alpha/\lambda$ and $|n_{14}|\sim |n_2|\sim |n_3|$ on $\mathcal{O}_3'$, $|R_{n_{14}}(n_2)| \lesssim N_{14}^2/\lambda$. In addition, by plugging $n_{23} = -n_1-n_4$, we find
\[
	R_{n_{23}}(n_1) = 3\alpha n_{14} (c_2 n_1 - c_1 n_4)(c_1 n_1 - c_2 n_4).
\]
As $|n_4| \gg |n_1|$ on $\mathcal{O}_3'$, we obtain 
\[
	|R_{n_{23}}(n_1)| \sim N_4^3 \gtrsim N_{14}^3.
\]
Thus, it follows from (\ref{1007-1}) that 
\begin{equation}\label{1007-2}
	|H_\alpha'| \geq |R_{n_{23}}(n_1)| - |R_{n_{14}}(n_2)| \gtrsim N_{14}^3 - \frac{N_{14}^2}{\lambda}\gtrsim N_{14}^3.
\end{equation}
On the other hand, similar to the derivation of (\ref{resonance_R}), we obtain 
\begin{equation}\label{1007-3}
	|R(n_{14}, n_{2},n_{3})| \gtrsim \lambda^{-1-2s_{\alpha}-\epsilon} N_{14}^{2-2s_{\alpha}-\epsilon}.
\end{equation}
Plugging (\ref{1007-3}) into (\ref{M44'}) yields 
\[
	\mathcal{M}'(n_1, n_2, n_3, n_4) \lesssim \frac{\lambda^{1+2s_{\alpha}+\epsilon} N_{14}^{2s_\alpha+\epsilon}}{\langle N_1\rangle\langle N_4\rangle m(N_1) m(N_4)}.
\]
Since $N_4\sim N_{14} \gtrsim N$ on $\mathcal{O}_{3}'$, we have 
\[
	\langle N_1\rangle\langle N_4\rangle m(N_1) m(N_4) \gtrsim \langle N_4\rangle m(N_4) \sim N^{1-s} N_{14}^{s}.
\]
Thus, 
\begin{equation}\label{1007-4}
	\mathcal{M}'(n_1, n_2, n_3, n_4) \lesssim \frac{\lambda^{1+2s_{\alpha}+\epsilon} N_{14}^{2s_\alpha+\epsilon-s}}{N^{1-s}}.
\end{equation}
Based on (\ref{1007-2}) and (\ref{1007-4}), we conclude that 
\[\begin{split}
	\frac{\mathcal{M}'}{\lambda \langle H_\alpha'\rangle^{1/2}} 
	&\lesssim \frac{\lambda^{2s_{\alpha}+\epsilon}}{N^{1-s} N_{14}^{\frac32+s-2s_\alpha-\epsilon}} \lesssim \frac{\lambda^{2s_{\alpha}+\epsilon}}{ N^{\frac52-2s_\alpha-\epsilon}} \lesssim D_1(\lambda, N, \epsilon),
\end{split}\]
where we used the facts that $\frac32+s-2s_\alpha-\epsilon > 0$ and $N_{14} \gtrsim N$.

Thus, we finish the proof of Lemma \ref{M42Estimate}.

\section{Quintilinear estimates}\label{Sec, quint-est}

\subsection{Setup}\label{Subsec, setup}

The goal of this section is to justify the quintilinear estimates in Lemma \ref{lem:M5j-estimates}.
Recalling that the bad region $\Omega_*$ is defined as (see \eqref{eq:bad-region-def}):
\[
\Omega_{*} = \left\{ (n_1, n_2, n_3, n_4) \in \Gamma_4^\lambda : |n_1| \gtrsim N,\, 0 < |n_2| \le 1,\, |n_4 + c_1 n_1| \le \frac{\kappa_\alpha}{\lambda} \text{ or } |n_4 + c_2 n_1| \le \frac{\kappa_\alpha}{\lambda} \right\}.
\]
Throughout this section, we consider only the part of $\Omega_*$ determined by
\begin{equation}\label{one_compo}
	|n_4+c_1 n_1|\leq\frac{\kappa_\alpha}{\lambda}.
\end{equation}
The part of $\Omega_*$ determined by
\[
|n_4+c_2n_1|\leq\frac{\kappa_\alpha}{\lambda}
\]
can be treated in the same way and therefore will be omitted.

On the other hand, we recall the definition of the quintilinear multipliers $\{M_{5j}^{P}\}_{j=1}^{4}$ in \eqref{M5jMultipliers}. Since fixed constant coefficients do not affect the estimates below, we omit them throughout the remainder of the proof. With this convention, \eqref{M5jMultipliers} reads as
\begin{equation}\label{eq:M5j-recall}
	\begin{aligned}
		&M_{51}^P(n_1,\ldots,n_5)
		=n_{15}\sigma_4^P(n_{15},n_2,n_3,n_4),\quad
		M_{52}^P(n_1,\ldots,n_5)
		=n_{25}\sigma_4^P(n_1,n_{25},n_3,n_4),\\
		&M_{53}^P(n_1,\ldots,n_5)
		=n_{35}\sigma_4^P(n_1,n_2,n_{35},n_4),\quad
		M_{54}^P(n_1,\ldots,n_5)
		=n_{45}\sigma_4^P(n_1,n_2,n_3,n_{45}).
	\end{aligned}
\end{equation}
By \eqref{eq:sigma4P} and \eqref{M41_comp}, the quartic multiplier $\sigma_4^P$ in
\eqref{eq:M5j-recall} is given by
\begin{equation*}
	\sigma_4^P(n_1,n_2,n_3,n_4)
	=n_{23}
	\frac{R^I(n_1,n_{23},n_4)}
	{R(n_1,n_{23},n_4)}
	\frac{\chi_{\Omega_*}(n_1,n_2,n_3,n_4)}
	{H_\alpha(n_1,n_2,n_3,n_4)},
\end{equation*}
where $R(x,y,z)$ and $R^I(x,y,z)$ are as defined in \eqref{DefOfRandRI},
$\chi_{\Omega_*}$ is the characteristic function of the bad region $\Omega_{*}$,
and $H_\alpha$ is defined in \eqref{eq:quartic-resonance-H}, i.e.
\[
H_\alpha(n_1,n_2,n_3,n_4)
:=n_1^3+n_2^3+\alpha n_3^3+\alpha n_4^3.
\]

Then by Lemma \ref{lemma:sigma4}, on the support of $\sigma_4^P(n_1,n_2,n_3,n_4)$, it holds that
\begin{equation}\label{eq:inherited-frequency-relations}
	N_1\sim N_3\sim N_4\sim N_{23}\gtrsim N,
	\qquad
	\lambda^{-1}\leq|n_2|\le1.
\end{equation}
The multiplier also satisfies
\begin{equation}\label{eq:sigma4j-est}
	|\sigma_4^P(n_1,n_2,n_3,n_4)|
	\lesssim
	\lambda^{2+2s_{\alpha}+\epsilon}
	N_1^{2s_{\alpha}+\epsilon}m^2(n_1)\chi_{\Omega_*}(n_1,n_2,n_3,n_4).
\end{equation}
The lower bound of $|H_\alpha|$ has been obtained in Lemma \ref{lemma:sigma4}. We shall also use an upper bound for $|H_\alpha|$ on the bad region $\Omega_{*}$. By
\eqref{HalphaExp},
\[
H_\alpha(n_1,n_2,n_3,n_4)
=R(n_1,n_{23},n_4)+R(n_{2},-n_{23},n_3).
\]
Under this restriction, the proof of Lemma~\ref{lemma:sigma4} gives
\eqref{eq:R-resonant-upper}, namely
\begin{equation}\label{Rn1n4_UpperBound}
	\left|R(n_1,n_{23},n_4)\right|
	\leq 2\alpha\rho_\alpha \kappa_\alpha \frac{|n_1|^2}{\lambda}.
\end{equation}
Moreover, since $N_1\sim N_3\gtrsim N$ and $|n_2|\le1$ on the bad region, we have
\begin{equation*}
	\left|R(n_{2},-n_{23},n_3)\right|
	\le
	3\alpha|n_2|\bigl(|n_3|+|c_1n_2|\bigr)\bigl(|n_3|+|c_2n_2|\bigr)
	\lesssim N_3^2.
\end{equation*}
It then follows that
\begin{equation}\label{eq:H-bad-upper}
	|H_\alpha(n_1,n_2,n_3,n_4)|\le |R(n_1,n_{23},n_4)|+|R(n_{2},-n_{23},n_3)|
	\lesssim N_3^2
	\quad\text{on}\quad \Omega_{*}.
\end{equation}
Meanwhile, we also recall the lower bound for $|R(n_1,n_{23},n_4)|$ in \eqref{eq:Temp2}:
\begin{equation}\label{Rn1n4_LowerBound}
|R(n_1,n_{23},n_4)|\gtrsim\lambda^{-1-2s_\alpha-\epsilon}|n_1|^{2-2s_\alpha-\epsilon} \quad\text{on}\quad \Omega_*.
\end{equation}

For $j \in \{1,2,3,4\}$, we define the dispersion vector
\begin{equation*}
	(\rho_{1,j},\rho_{2,j},\rho_{3,j},\rho_{4,j},\rho_{5,j})
	:=
	\begin{cases}
		(\alpha,1,\alpha,\alpha,\alpha), & j=1,\\
		(1,\alpha,\alpha,\alpha,\alpha), & j=2,\\
		(1,1,1,\alpha,\alpha), & j=3,\\
		(1,1,\alpha,1,\alpha), & j=4.
	\end{cases}
\end{equation*}
For $i\in\{1,2,3,4,5\}$ and $j\in\{1,2,3,4\}$, let 
\begin{equation*}
	L_{i,j}:=\tau_i-\rho_{i,j}n_i^3
	\quad \text{and} \quad 
	L_{M,j}
	:=\max_{1\leq i\leq5}\langle L_{i,j}\rangle.
\end{equation*}
Denote the frequency vector $\vec{n} = (n_1, n_2, n_3, n_4, n_5)$ and the temporal vector $\vec{\tau} = (\tau_1, \tau_2, \tau_3, \tau_4, \tau_5)$. For $j\in\{1,2,3,4\}$, we also define
\begin{equation}\label{eq:Kj-LMj-def}
	K_j(\vec{n}):=\sum_{i=1}^5\rho_{i,j}n_i^3.
\end{equation}
Then on the hyperplane 
$\{(\vec{n}, \vec{\tau}): \sum_{i=1}^5 n_i=0, \, \sum_{i=1}^5\tau_i=0\}$, 
$K_j(\vec{n})=-\sum_{i=1}^5L_{i,j}$, which implies that 
\begin{equation}\label{eq:LM-Kj}
	L_{M,j} \gtrsim \sum_{i=1}^5\langle L_{i,j}\rangle \gtrsim 
	\langle K_j(\vec{n})\rangle,
	\qquad 1\leq j\leq4.
\end{equation}

For each $j\in\{1,2,3,4\}$, let
\begin{equation}\label{Gamma_5bad}
	\Gamma_{5,j}^{\mathrm{Bad}}
	:=
	\left\{\vec{n}\in\Gamma_5^\lambda:
	(n_1,\ldots,n_{j-1},n_{j5},n_{j+1},\ldots,n_4)
	\in\Omega_*\right\}.
\end{equation}
For example, 
\[\begin{split}
	\Gamma_{5,1}^{\mathrm{Bad}}
	&=
	\left\{\vec{n}\in\Gamma_5^\lambda:
	(n_{15}, n_{2}, n_{3}, n_{4})
	\in\Omega_*\right\}, \\
	\Gamma_{5,2}^{\mathrm{Bad}}
	&=
	\left\{\vec{n}\in\Gamma_5^\lambda:
	(n_1, n_{25}, n_{3}, n_{4})
	\in\Omega_*\right\},
\end{split}\]
$\Gamma_{5,3}^{\mathrm{Bad}}$ and $\Gamma_{5,4}^{\mathrm{Bad}}$ can be understood similarly. 
Define
\begin{equation}\label{eq:Lj}
	\mathcal L_j
	:=\int_{\sum\limits_{i=1}^{5}\tau_i = 0}
	\int_{\Gamma_{5,j}^{\mathrm{Bad}}}
	\frac{
		M^P_{5j}(n_1,\cdots,n_5)
		\prod_{i=1}^{5}\widetilde f_i(n_i,\tau_i)
	}{
		\prod_{i=1}^5
		\langle n_i\rangle m(n_i)
		\langle L_{i,j}\rangle^{1/2}
	},
	\qquad\forall\, j\in\{1,2,3,4\}.
\end{equation}
By an analogous argument as the reduction from (\ref{M41_estp}) to (\ref{1003-2}), we find that the estimates \eqref{eq:M51-est}--\eqref{eq:M54-est} boil down to 
\begin{equation}\label{eq:M5j-target}
	|\mathcal L_j|
	\lesssim
	D_2(\lambda,N,\epsilon)\prod_{i=1}^{5}\|f_i\|_{L_{x,t}^{2,2}}, \quad \forall\, j\in\{1,2,3,4\}.
\end{equation}
Without loss of generality, we may assume that $\widetilde f_i(n_i,\tau_i)$ in (\ref{eq:Lj}) are all nonnegative for $1\leq i\leq5$.

\subsection{Proof of \eqref{eq:M51-est}}
\label{Subsec, quin_est1}

As in the previous discussion, the estimate (\ref{eq:M51-est}) is reduced to (\ref{eq:M5j-target}) with $j=1$, where $\mathcal{L}_{1}$ is given by (\ref{eq:Lj}). 
Recalling 
$M_{51}^P(n_1,\ldots,n_5) = n_{15} \sigma_4^P(n_{15},n_2,n_3,n_4)$ 
in (\ref{eq:M5j-recall}), it then follows from \eqref{eq:inherited-frequency-relations} that 
\begin{equation}\label{FreofRelInM51P}
	N_{15}\sim N_3\sim N_4\sim N_{23}\gtrsim N,
	\qquad \lambda^{-1}\leq|n_2|\le1.
\end{equation}
Let 
\begin{equation}\label{1010-1}
	\mathscr{M}_{51}(\vec{n}) = \frac{ |M_{51}^{P}(\vec{n})|}{\prod_{i=1}^{5} \langle n_i\rangle m(n_i)}, \qquad\forall\, \vec{n}=(n_1,n_2,n_3,n_4,n_5)\in \Gamma_{5,1}^{\text{Bad}}.
\end{equation}
Then (\ref{eq:M5j-target}) with $j=1$ boils down to 
\begin{equation}\label{1010-2}
	\int_{\sum\limits_{i=1}^{5}\tau_i = 0}
	\int_{\Gamma_{5,1}^{\mathrm{Bad}}}
	\mathscr{M}_{51}(\vec{n})
	\prod_{i=1}^{5} \frac{\widetilde f_i(n_i,\tau_i)}{\langle L_{i,1}\rangle^{1/2}}		
	\lesssim D_2(\lambda,N,\epsilon)\prod_{i=1}^{5}\|f_i\|_{L_{x,t}^{2}}.
\end{equation}
Plugging (\ref{eq:sigma4j-est}) into (\ref{eq:M5j-recall}) yields 
\begin{equation}\label{1010-3}
	|M_{51}^{P}(\vec{n})| \lesssim \lambda^{2+2s_{\alpha}+\epsilon}
	N_{15}^{1+2s_{\alpha}+\epsilon}m^2(n_{15})\chi_{\Omega_*}(n_{15},n_2,n_3,n_4).
\end{equation}
Based on the relation (\ref{FreofRelInM51P}), 
\begin{equation}\label{1010-4}
	\langle n_3\rangle\langle n_4\rangle m(n_3)m(n_4) \sim N_{15}^2 m^2(n_{15}) \quad\text{and}\quad \langle n_2\rangle m(n_2) \sim 1.
\end{equation}
Then we combine (\ref{1010-1}), (\ref{1010-3}) and (\ref{1010-4}) to obtain that for any $ \vec{n} \in \Gamma_{5,1}^{\mathrm{Bad}}$,
\begin{equation}\label{eq:L1estTemp1}
	\mathscr{M}_{51}(\vec{n})
	\lesssim 
	\frac{
		\lambda^{2+2s_{\alpha}+\epsilon}N_{15}^{-1+2s_{\alpha}+\epsilon}
	}{
		\langle n_1\rangle\langle n_5\rangle m(n_1)m(n_5)
	}.
\end{equation}
Since the relative sizes of $|n_1|$ and $|n_5|$ are not determined at this point, we will divide the proof into two cases depending on the sizes of $|n_1|$ and $|n_5|$. 

Before that, we perform a change of variable on $\Omega_*$ to take advantage of its sparsity. In fact, the frequencies $(n_1,n_2,n_3,n_4,n_5)$ in $\Gamma_{5,1}^{\mathrm{Bad}}$ satisfy $(n_{15}, n_2, n_3, n_4)\in\Omega_*$, so $|n_4 + c_1 n_{15}| \leq \frac{\kappa_\alpha}{\lambda}$ due to the assumption (\ref{one_compo}).
Setting $\theta = n_4 + c_1 n_{15}$, then
\begin{equation}\label{eq:M51-case1-parametrization}
	n_4=-c_1n_{15}+\theta,
	\qquad
	n_3=-c_2n_{15}-n_2-\theta.
\end{equation}
Meanwhile, based on the definition (\ref{Theta_def}), we have $\theta\in \Theta_{\kappa_\alpha}(n_{15};-c_1)$, where  
\[
	\Theta_{\kappa_\alpha}(n_{15};-c_1) = 	
	\left\{\theta\in\mathbb{R}:
	|\theta|\leq\frac{\kappa_\alpha}{\lambda}
	\ \text{and}\
	-c_1n_{15}+\theta\in\mathbb{Z}_\lambda\right\}.
\]
Since $\kappa_\alpha\leq \frac18$, the set $\Theta_{\kappa_\alpha}(n_{15};-c_1)$ contains at most one element. 
For convenience of notation, we will suppress the subscript and abbreviate $\Theta_{\kappa_\alpha}(n_{15};-c_1)$ as  $\Theta(n_{15};-c_1)$ in the rest of this proof. 

\textbf{Case 1.} $|n_1|\sim |n_5|$.

The condition $|n_1|\sim |n_5|$, together with
\eqref{FreofRelInM51P}, gives $N_1\sim N_5\gtrsim N_{15} \gtrsim N$.
Hence, by \eqref{eq:L1estTemp1} and the assumption $\epsilon\in (0, 2s-2s_\alpha)$, we find
\begin{equation}\label{1011-4}
	\mathscr{M}_{51}(\vec{n}) \lesssim \frac{\lambda^{2+2s_{\alpha}+\epsilon} N_{15}^{-1+2s_{\alpha}+\epsilon}}{N_1^s N_5^s \,N^{2-2s}}
	\lesssim \frac{\lambda^{2+2s_{\alpha}+\epsilon}}{N^{3-2s_\alpha-\epsilon}}.
\end{equation}
As a result, (\ref{1010-2}) is reduced to 
\[
	\frac{\lambda^{2+2s_{\alpha}+\epsilon}}{N^{3-2s_\alpha-\epsilon}} 
	\int_{\sum\limits_{i=1}^{5}\tau_i = 0}
	\int_{\Gamma_{5,1}^{\mathrm{Bad}}}
	\prod_{i=1}^{5} \frac{\widetilde f_i(n_i,\tau_i)}{\langle L_{i,1}\rangle^{1/2}}		
	\lesssim D_2(\lambda,N,\epsilon)\prod_{i=1}^{5}\|f_i\|_{L_{x,t}^{2}}.
\]
Define the functions $\{F_i\}_{i=1}^{5}$ such that their space-time Fourier transforms $\{\widetilde{F}_i\}_{i=1}^{5}$ satisfy
\begin{equation}\label{1011-1}
	\widetilde F_i(n_i,\tau_i)
	=
	\frac{\widetilde f_i(n_i,\tau_i)}{\langle L_{i,1}\rangle^{1/2}},
	\qquad 1\leq i\leq5.
\end{equation}
Then based on the formula (\ref{eq:D2-def}) for $D_2(\lambda,N,\epsilon)$, the above estimate is equivalent to 
\begin{equation}\label{1013-11}
	\int_{\sum\limits_{i=1}^{5}\tau_i = 0}
	\int_{\Gamma_{5,1}^{\mathrm{Bad}}}
	\prod_{i=1}^{5} \widetilde F_i(n_i,\tau_i)
	\lesssim \lambda^{-\frac32} N \prod_{i=1}^{5}\|f_i\|_{L_{x,t}^{2}},
\end{equation}

Taking the inverse Fourier transform in the time variable, (\ref{1013-11}) is converted to
\begin{equation}\label{1010-5}
	\int_{\mathbb{R}}
	\int_{\Gamma_{5,1}^{\mathrm{Bad}}}
	\prod_{i=1}^{5}  \widehat F_i(n_i, t)  \,dt
	\lesssim \lambda^{-\frac32} N \prod_{i=1}^{5}\|f_i\|_{L_{x,t}^{2}},
\end{equation}
where $\widehat{F}_{i}$ represents the Fourier transform of $F_i$ in the spatial variable. 

Using the variable $\theta$ in \eqref{eq:M51-case1-parametrization}, we find
\begin{equation}\label{1011-7}\begin{split}
	\text{LHS of (\ref{1010-5})} = \Big(\frac{1}{2\pi\lambda}\Big)^4 \int_{\mathbb{R}} &\sum_{n_1\in\mathbb{Z}_\lambda} \sum_{\frac{1}{\lambda}\leq |n_2|\leq 1} \sum_{n_5\in\mathbb{Z}_\lambda}  \sum_{\theta\in\Theta(n_{15};-c_1)} \bigg[\bigl| \widehat{F}_1(n_1, t) \bigr| \bigl| \widehat{F}_2(n_2, t) \bigr| \\
	& \bigl| \widehat{F}_5(n_5, t) \bigr| \bigl| \widehat{F}_3(-c_2 n_{15}-n_2-\theta, t) \bigr| \bigl| \widehat{F}_4(-c_1 n_{15}+\theta, t) \bigr|\bigg] \,dt,
\end{split}\end{equation}
By applying another change of variable $n_5\to k := n_5+n_1$, we obtain
\[\begin{split}
	\text{LHS of (\ref{1010-5})} = \Big(\frac{1}{2\pi\lambda}\Big)^4 \int_{\mathbb{R}} & \sum_{n_1\in\mathbb{Z}_\lambda} \sum_{\frac{1}{\lambda}\leq |n_2|\leq 1} \sum_{k\in\mathbb{Z}_\lambda}  \sum_{\theta\in\Theta(k;-c_1)} 
	\bigg[ \bigl| \widehat{F}_1(n_1, t) \bigr| \bigl| \widehat{F}_2(n_2, t) \bigr|\\
	& \bigl| \widehat{F}_5(k-n_1, t) \bigr| \bigl| \widehat{F}_3(-c_2 k - n_2-\theta, t) \bigr| \bigl| \widehat{F}_4(-c_1 k + \theta, t) \bigr|\bigg] \,dt.
\end{split}\]
Rearranging the order of integration yields 
\begin{equation}\label{1010-6}\begin{split}
	\text{LHS of (\ref{1010-5})} \sim \frac{1}{\lambda^4} & \int_{\mathbb{R}} 
	\sum_{\frac{1}{\lambda}\leq |n_2|\leq 1} |\widehat{F}_2(n_2, t)| \bigg(\sum_{k\in\mathbb{Z}_\lambda} \sum_{\theta\in\Theta(k;-c_1)} 
	\bigl| \widehat{F}_3(-c_2 k - n_2-\theta, t) \bigr| \bigl| \widehat{F}_4(-c_1 k + \theta, t) \bigr| \\
	& \Big[\sum_{n_1\in\mathbb{Z}_\lambda} \bigl| \widehat{F}_1(n_1, t) \bigr| \bigl| \widehat{F}_5(k-n_1, t) \bigr|\Big]\bigg) \,dt.
\end{split}\end{equation}

For each fixed $t$ and $k$, the Cauchy-Schwarz inequality gives
\[
	\frac{1}{2\pi\lambda}
	\sum_{n_1\in\mathbb Z_\lambda}
	\big|\widehat F_1(n_1, t)\big|
	\big|\widehat F_5(k-n_1, t)\big|\leq
	\|\widehat F_1(\cdot, t)\|_{\ell^2}
	\|\widehat F_5(\cdot, t)\|_{\ell^2}.
\]
Moreover, since the set $\Theta(n_{15};-c_1)$ contains at most one element, we know 
\[\begin{split}
	&\sum_{\theta\in\Theta(k;-c_1)} 
	\bigl| \widehat{F}_3(-c_2 k - n_2-\theta, t) \bigr| \bigl| \widehat{F}_4(-c_1 k + \theta, t) \bigr| \\
	=&\,\, \bigg( \sum_{\theta\in\Theta(k;-c_1)} 
	\bigl| \widehat{F}_3(-c_2 k - n_2-\theta, t) \bigr| \bigg) 
	\bigg(\sum_{\theta\in\Theta(k;-c_1)} \bigl| \widehat{F}_4(-c_1 k + \theta, t) \bigr| \bigg),
\end{split}\]
which implies that 
\[\begin{split}
	&\bigg| \frac{1}{2\pi\lambda}\sum_{k\in\mathbb{Z}_\lambda} \sum_{\theta\in\Theta(k;-c_1)} 
	\bigl| \widehat{F}_3(-c_2 k - n_2-\theta, t) \bigr| \bigl| \widehat{F}_4(-c_1 k + \theta, t) \bigr|\bigg| \\
	\le&\,\, \bigg\| \sum_{\theta\in\Theta(k;-c_1)} 
	\bigl| \widehat{F}_3(-c_2 k - n_2-\theta, t) \bigr| \bigg\|_{\ell^2_k} 
	\bigg\|\sum_{\theta\in\Theta(k;-c_1)} \bigl| \widehat{F}_4(-c_1 k + \theta, t) \bigr| \bigg\|_{\ell^2_k}.
\end{split}\]
Applying Lemma \ref{lem:admissible-frequency-sum} with
$(A,a,m)=(\kappa_\alpha,-c_1,0)$ yields 
\[
	\bigg\|\sum_{\theta\in\Theta(k;-c_1)} \bigl| \widehat{F}_4(-c_1 k + \theta, t) \bigr| \bigg\|_{\ell^2_k} \lesssim \|\widehat F_4(\cdot, t)\|_{\ell^2}.
\]
Meanwhile, since $c_1+c_2=1$, the change of variable $\theta\mapsto-\theta$ identifies
$\Theta(k;-c_1)$ with $\Theta(k;-c_2)$. Thus, 
\[\begin{split}
	\bigg\| \sum_{\theta\in\Theta(k;-c_1)} 
	\bigl| \widehat{F}_3(-c_2 k - n_2-\theta, t) \bigr| \bigg\|_{\ell^2_k} = 
	\bigg\| \sum_{\theta\in\Theta(k;-c_2)} 
	\bigl| \widehat{F}_3(-c_2 k - n_2+\theta, t) \bigr| \bigg\|_{\ell^2_k} \lesssim \|\widehat F_3(\cdot, t)\|_{\ell^2},
\end{split}\]
where the last inequality is due to  Lemma \ref{lem:admissible-frequency-sum} with
$(A,a,m) = (\kappa_\alpha,-c_2,-n_2)$.
Plugging all the above estimates into (\ref{1010-6}) gives 
\begin{equation}\label{1010-7}
	\text{LHS of (\ref{1010-5})} \lesssim \frac{1}{\lambda^2} \int_{\mathbb{R}} 
	\sum_{\frac{1}{\lambda}\leq |n_2|\leq 1} \bigl| \widehat{F}_2(n_2, t) \bigr| \|\widehat F_3(\cdot, t)\|_{\ell^2} \|\widehat F_4(\cdot, t)\|_{\ell^2} \|\widehat F_1(\cdot, t)\|_{\ell^2} \|\widehat F_5(\cdot, t)\|_{\ell^2} \,dt.
\end{equation}
Now we apply the Cauchy-Schwarz inequality again to find
\begin{equation}\label{eq:M51-low-frequency-sum}
	\frac{1}{2\pi\lambda}
	\sum_{\frac{1}{\lambda}\leq|n_2|\le1}
	\big|\widehat F_2(n_2,t)\big|
	\lesssim
	\|\widehat F_2(\cdot,t)\|_{\ell^2}.
\end{equation}
The combination of (\ref{1010-7}) and (\ref{eq:M51-low-frequency-sum}) yields 
\begin{equation}\label{1010-8}
	\text{LHS of (\ref{1010-5})}
	\lesssim
	\frac{1}{\lambda}
	\int_{\mathbb R}
	\prod_{i=1}^5
	\|\widehat F_i(\cdot,t)\|_{\ell^2(\mathbb{Z}_\lambda)}\,dt
	= \frac{1}{\lambda}\int_{\mathbb R}
	\prod_{i=1}^5 \| F_i(\cdot,t)\|_{L^2(\mathbb{T}_\lambda)}\,dt,
\end{equation}
where the last equality is due to the Plancherel identity.

Based on (\ref{1010-8}), we apply the H\"older inequality in time to obtain 
\begin{equation}\label{1011-2}
	\text{LHS of (\ref{1010-5})}
	\lesssim \frac{1}{\lambda} \|F_1\|_{L_t^2 L_x^2}
	\prod_{i=2}^5\|F_i\|_{L_t^8 L_x^2} \lesssim \frac{1}{\lambda} \prod_{i=1}^5\|f_i\|_{L^{2,2}_{x,t}},
\end{equation}
where the last inequality is due to (\ref{1011-1}) and Lemma \ref{lemma3}. Finally, thanks to the assumption $\lambda \lesssim N^{\frac13}$, we know $\frac{1}{\lambda} \lesssim \lambda^{-\frac32} N$. Therefore, (\ref{1011-2}) implies (\ref{1010-5}), and Case 1 is justified.

\textbf{Case 2.} $|n_1|\gg |n_5|$ or $|n_5|\gg |n_1|$.

Since $M_{51}^P$ is symmetric with respect to $n_1$ and $n_5$, we may
assume without loss of generality that $|n_1|\gg|n_5|$. If
$|n_5|\gtrsim N$, the argument is the same as in Case~1. We therefore
consider only the case $|n_5|\ll N$. By
\eqref{eq:inherited-frequency-relations}, the assumption
$|n_1|\gg|n_5|$ and the identity $n_{15}=n_1+n_5$ give
\[
	N_1\sim N_{15}\gtrsim N\gg N_5.
\]
So it follows from (\ref{eq:L1estTemp1}) that 
\begin{equation}\label{1011-3}
	\mathscr{M}_{51}(\vec{n}) \lesssim \frac{\lambda^{2+2s_{\alpha}+\epsilon} N_{15}^{-1+2s_{\alpha}+\epsilon}}{N_1^s \,N^{1-s} \langle n_5\rangle}
	\lesssim \frac{\lambda^{2+2s_{\alpha}+\epsilon}}{N^{2-2s_\alpha-\epsilon}\langle n_5\rangle},
\end{equation}
where we used the fact that $s>-1+2s_\alpha+\epsilon$. Since $\langle n_5\rangle$ is very small, the upper bound in (\ref{1011-3}) is not as small as (\ref{1011-4}) in Case 1. So we need to take advantage of the lower bound $\langle K_1(\vec{n})\rangle$ for $L_{M,1}$ in \eqref{eq:LM-Kj} to provide additional decay. 
By the definition of $K_1$ in (\ref{eq:Kj-LMj-def}), it admits
the decomposition
\begin{equation*}
	\begin{aligned}
		K_1(n_1,\cdots,n_5) 
		&=
		(n_{15}^3+n_2^3+\alpha n_3^3+\alpha n_4^3)
		+(\alpha n_1^3+\alpha n_5^3-n_{15}^3 )     \\
		&=
		H_\alpha(n_{15},n_2,n_3,n_4)
		+R(-n_{15},n_1,n_5),
	\end{aligned}
\end{equation*}
where $H_\alpha$ and $R$ are as defined in (\ref{eq:quartic-resonance-H}) and (\ref{DefOfRandRI}). According to (\ref{eq:H-bad-upper}), 
$|H_\alpha(n_{15},n_2,n_3,n_4)| \lesssim N_3^2$.
Meanwhile, 
\begin{equation*}
	\begin{aligned}
		|R(-n_{15},n_1,n_5)|
		&=
		|3\alpha n_{15}(n_1-c_1n_{15})(n_1-c_2n_{15})|     \\
		&=
		|3\alpha n_{15}(c_2n_1-c_1n_5)(c_1n_1-c_2n_5)|
		\gtrsim N_{15} N_1^2 \gtrsim N_3^3,
	\end{aligned}
\end{equation*}
where the last inequality is due to $N_1 \sim N_{15} \sim N_3$. Thus, 
\begin{equation}\label{K1est}
	L_{M,1} \gtrsim |K_1(n_1,\cdots,n_5)|\gtrsim N_{3}^3 \gtrsim N^3.
\end{equation}
Based on (\ref{1011-3}) and (\ref{K1est}), (\ref{1010-2}) is reduced to 
\[
	\frac{\lambda^{2+2s_{\alpha}+\epsilon}}{N^{2-2s_\alpha-\epsilon}} \frac{1}{N^{3/2}}
	\int_{\sum\limits_{i=1}^{5}\tau_i = 0}
	\int_{\Gamma_{5,1}^{\mathrm{Bad}}}
	\frac{L_{M,1}^{1/2}}{\langle n_5\rangle}
	\prod_{i=1}^{5} \frac{\widetilde f_i(n_i,\tau_i)}{\langle L_{i,1}\rangle^{1/2}}		
	\lesssim D_2(\lambda,N,\epsilon)\prod_{i=1}^{5}\|f_i\|_{L_{x,t}^{2}}.
\]
According to the formula (\ref{eq:D2-def}) for $D_2(\lambda,N,\epsilon)$, the above estimate is equivalent to 
\begin{equation}\label{1011-5}
	\int_{\sum\limits_{i=1}^{5}\tau_i = 0}
	\int_{\Gamma_{5,1}^{\mathrm{Bad}}}
	\frac{L_{M,1}^{1/2}}{\langle n_5\rangle}
	\prod_{i=1}^{5} \frac{\widetilde f_i(n_i,\tau_i)}{\langle L_{i,1}\rangle^{1/2}}		
	\lesssim \lambda^{-\frac32} N^{\frac32} \prod_{i=1}^{5}\|f_i\|_{L_{x,t}^{2}}.
\end{equation}

Since there exists $i\in\{1,2,3,4,5\}$ such that $L_{M,1} = \langle L_{i,1}\rangle$, we assume, without loss of generality, that $ L_{M,1} = \langle L_{1,1}\rangle$. Other cases can be treated similarly. Define functions $\{F_i\}_{i=1}^{5}$ whose space-time Fourier transforms $\{\widetilde{F}_i\}_{i=1}^{5}$ satisfy
\begin{equation*}
	\widetilde F_1(n_1, \tau_1) = \widetilde f_1(n_1,\tau_1),
	\qquad
	\widetilde F_i(n_i,\tau_i)
	=
	\frac{\widetilde f_i(n_i,\tau_i)}{\langle L_{i,1}\rangle^{1/2}},
	\qquad 2\le i\le 5,
\end{equation*}
(\ref{1011-5}) is converted to
\begin{equation}\label{1011-6}
	\int_{\sum\limits_{i=1}^{5}\tau_i = 0}
	\int_{\Gamma_{5,1}^{\mathrm{Bad}}}
	\frac{1}{\langle n_5\rangle}
	\prod_{i=1}^{5} \widetilde F_i(n_i,\tau_i) 
	\lesssim \lambda^{-\frac32} N^{\frac32} \prod_{i=1}^{5}\|f_i\|_{L_{x,t}^{2}}.
\end{equation}
Then similar to the derivation from (\ref{1013-11}) to (\ref{1011-7}), we have
\begin{equation}\label{1011-8}\begin{split}
	\text{LHS of (\ref{1011-6})} = \Big(\frac{1}{2\pi\lambda}\Big)^4 \int_{\mathbb{R}} &\sum_{n_1\in\mathbb{Z}_\lambda} \sum_{\frac{1}{\lambda}\leq |n_2|\leq 1} \sum_{n_5\in\mathbb{Z}_\lambda}  \sum_{\theta\in\Theta(n_{15};-c_1)} \Big[\bigl| \widehat{F}_1(n_1, t) \bigr| \bigl| \widehat{F}_2(n_2, t) \bigr| \\
	& \langle n_5\rangle^{-1} \bigl| \widehat{F}_5(n_5, t) \bigr| \bigl| \widehat{F}_3(-c_2 n_{15}-n_2-\theta, t) \bigr| \bigl| \widehat{F}_4(-c_1 n_{15}+\theta, t) \bigr|\Big] \,dt.
\end{split}\end{equation}
By rearranging the order of integration, we obtain 
\[\begin{split}
	\text{LHS of (\ref{1011-6})} = \Big(\frac{1}{2\pi\lambda}\Big)^4 \int_{\mathbb{R}} &
	\sum_{n_5\in\mathbb{Z}_\lambda} \langle n_5\rangle^{-1} \bigl| \widehat{F}_5(n_5, t) \bigr| \bigg[\sum_{n_1\in\mathbb{Z}_\lambda} \bigl| \widehat{F}_1(n_1, t) \bigr| \sum_{\theta\in\Theta(n_{15};-c_1)}\bigl| \widehat{F}_4(-c_1 n_{15}+\theta, t) \bigr| \\
	& \Big(\sum_{\frac{1}{\lambda}\leq |n_2|\leq 1} \bigl| \widehat{F}_2(n_2, t) \bigr| \bigl| \widehat{F}_3(-c_2 n_{15}-n_2-\theta, t) \bigr|\Big)\bigg] \,dt.
\end{split}\]
Now we perform a change of variable $n_1\to k:= n_1+n_5$ to find 
\begin{equation}\label{1011-9}\begin{split}
	\text{LHS of (\ref{1011-6})} = \Big(\frac{1}{2\pi\lambda}\Big)^4 \int_{\mathbb{R}} &
	\sum_{n_5\in\mathbb{Z}_\lambda} \langle n_5\rangle^{-1} \bigl| \widehat{F}_5(n_5, t) \bigr| \bigg[\sum_{k\in\mathbb{Z}_\lambda} \bigl| \widehat{F}_1(k-n_5, t) \bigr| \sum_{\theta\in\Theta(k;-c_1)}\bigl| \widehat{F}_4(-c_1 k+\theta, t) \bigr| \\
	& \Big(\sum_{\frac{1}{\lambda}\leq |n_2|\leq 1} \bigl| \widehat{F}_2(n_2, t) \bigr| \bigl| \widehat{F}_3(-c_2 k - n_2-\theta, t) \bigr|\Big)\bigg] \,dt.
\end{split}\end{equation}
Since 
\[\begin{split}
	\bigg| \frac{1}{\lambda} \sum_{\frac{1}{\lambda}\leq |n_2|\leq 1} \bigl| \widehat{F}_2(n_2, t) \bigr| \bigl| \widehat{F}_3(-c_2 k - n_2-\theta, t) \bigr| \bigg|
	&\lesssim \big\| \widehat{F}_2(n_2, t) \big\|_{\ell_{n_2}^2} \big\| \widehat{F}_3(-c_2 k - n_2-\theta, t) \big\|_{\ell_{n_2}^2} \\
	&= \big\| \widehat{F}_2(\cdot, t) \big\|_{\ell^2} \big\| \widehat{F}_3(\cdot, t) \big\|_{\ell^2},
\end{split}\]
and 
\[\begin{split}
	\bigg| \frac{1}{\lambda}\sum_{k\in\mathbb{Z}_\lambda} \bigl| \widehat{F}_1(k-n_5, t) \bigr| \sum_{\theta\in\Theta(k;-c_1)}\bigl| \widehat{F}_4(-c_1 k+\theta, t) \bigr| \bigg|
	&\lesssim \big\| \widehat{F}_1(k-n_5, t) \big\|_{\ell^2_k}
	\Big\| \sum_{\theta\in\Theta(k;-c_1)}\bigl| \widehat{F}_4(-c_1 k+\theta, t) \bigr| \Big\|_{\ell^2_k} \\
	&\lesssim \| \widehat{F}_1(\cdot, t) \|_{\ell^2} \| \widehat{F}_4(\cdot, t) \|_{\ell^2}, 
\end{split}\]
where the last inequality applies Lemma~\ref{lem:admissible-frequency-sum},
with $A=\kappa_\alpha$, $a=-c_1$ and $m=0$. Plugging the above estimates into (\ref{1011-9}) yields 
\begin{equation}\label{1011-10}\begin{split}
	\text{LHS of (\ref{1011-6})} \lesssim \frac{1}{\lambda^2} \int_{\mathbb{R}} &
	\sum_{n_5\in\mathbb{Z}_\lambda} \langle n_5\rangle^{-1} \bigl| \widehat{F}_5(n_5, t) \bigr| 
	\Big( \big\| \widehat{F}_2(\cdot, t) \big\|_{\ell^2} \big\| \widehat{F}_3(\cdot, t) \big\|_{\ell^2}\Big) \Big(\| \widehat{F}_1(\cdot, t) \|_{\ell^2} \| \widehat{F}_4(\cdot, t) \|_{\ell^2}\Big)	\,dt.
\end{split}\end{equation}

By the Cauchy-Schwarz inequality, 
\[
	\bigg|\frac{1}{\lambda} \sum_{n_5\in\mathbb{Z}_\lambda} \langle n_5\rangle^{-1} \bigl| \widehat{F}_5(n_5, t) \bigr| \bigg| \lesssim \| \langle n_5\rangle^{-1} \|_{\ell_{n_5}^2} \|\widehat{F}_5(n_5, t)\|_{\ell_{n_5}^2} \lesssim \|\widehat{F}_5(\cdot, t)\|_{\ell^2}.
\]
Putting this relation into (\ref{1011-10}) gives 
\[
	\text{LHS of (\ref{1011-6})} \lesssim \frac{1}{\lambda} \int_{\mathbb{R}} \prod_{i=1}^{5} \| \widehat{F}_i(\cdot, t) \|_{\ell^2(\mathbb{Z}_\lambda)} \,dt 
	= \frac{1}{\lambda} \int_{\mathbb{R}} \prod_{i=1}^{5} \| F_i(\cdot, t) \|_{L^2(\mathbb{T}_\lambda)} \,dt.
\]
Then similar to the derivation of (\ref{1011-2}), the above estimate implies that 
\[
	\text{LHS of (\ref{1011-6})} \lesssim \frac{1}{\lambda} \prod_{i=1}^{5}\|f_i\|_{L^{2,2}_{x,t}} \lesssim \text{RHS of (\ref{1011-6})},
\]
where we used the relation that 
\begin{equation}\label{para_req}
	\frac{1}{\lambda} \lesssim \lambda^{-\frac32} N^{\frac32}
\end{equation} 
due to the assumption $\lambda \lesssim N^{1/3}$. Hence, Case 2 is proved and therefore (\ref{eq:M51-est}) is established.

\subsection{Proof of \eqref{eq:M52-est}}
As the discussion in Section \ref{Subsec, setup}, the estimate (\ref{eq:M52-est}) is reduced to (\ref{eq:M5j-target}) with $j=2$, where $\mathcal{L}_{2}$ is given by (\ref{eq:Lj}) with $j=2$. 
Recalling 
$M_{52}^P(n_1,\ldots,n_5) = n_{25} \sigma_4^P(n_{1},n_{25},n_3,n_4)$, it then follows from \eqref{eq:inherited-frequency-relations} that 
\begin{equation}\label{1012-2}
	N_{1}\sim N_3\sim N_4\sim N_{235}\gtrsim N,
	\qquad \lambda^{-1}\leq|n_{25}|\le1.
\end{equation}
Let 
\begin{equation}\label{1012-1}
	\mathscr{M}_{52}(\vec{n}) = \frac{ |M_{52}^{P}(\vec{n})|}{\prod_{i=1}^{5} \langle n_i\rangle m(n_i)}, \qquad\forall\, \vec{n}=(n_1,n_2,n_3,n_4,n_5)\in \Gamma_{5,2}^{\text{Bad}}.
\end{equation}
Then (\ref{eq:M5j-target}) with $j=2$ boils down to 
\begin{equation}\label{1012-8}
	\int_{\sum\limits_{i=1}^{5}\tau_i = 0}
	\int_{\Gamma_{5,2}^{\mathrm{Bad}}}
	\mathscr{M}_{52}(\vec{n})
	\prod_{i=1}^{5} \frac{\widetilde f_i(n_i,\tau_i)}{\langle L_{i,2}\rangle^{1/2}}		
	\lesssim D_2(\lambda,N,\epsilon)\prod_{i=1}^{5}\|f_i\|_{L_{x,t}^{2}}.
\end{equation}
Thanks to \eqref{eq:sigma4j-est} and (\ref{1012-2}),
\[
	|M_{52}^{P}(\vec{n})| \lesssim \lambda^{2+2s_\alpha+\epsilon} N_1^{2s_\alpha+\epsilon} m^2(n_1).
\]
Meanwhile, it follows from (\ref{1012-2}) again that $\langle n_2\rangle\langle n_5\rangle m(n_2) m(n_5) \gtrsim 1$ and 
\[
	\langle n_1\rangle\langle n_3\rangle\langle n_4\rangle m(n_1) m(n_3) m(n_4) \sim N_1^3 m^3(n_1)\sim N_1^{2+s} N^{1-s} m^2(n_1).
\]
As a result, 
\begin{equation}\label{1012-3}
	\mathscr{M}_{52}(\vec{n}) \lesssim \frac{\lambda^{2+2s_\alpha+\epsilon}}{N_1^{2+s-2s_\alpha-\epsilon} N^{1-s}} \lesssim 
	\frac{\lambda^{2+2s_\alpha+\epsilon}}{N^{3-2s_\alpha-\epsilon}},
\end{equation}
Noticing that the upper bound in (\ref{1012-3}) matches that in (\ref{1011-4}) for $\mathscr{M}_{51}$ in Case 1 in Section \ref{Subsec, quin_est1}, so the rest of the proof can be carried out in a similar way as that in Case 1 in Section \ref{Subsec, quin_est1}.

\subsection{Proof of \eqref{eq:M53-est}}
\label{Subsec, M53proof}
As the discussion in Section \ref{Subsec, setup}, the estimate (\ref{eq:M53-est}) is reduced to (\ref{eq:M5j-target}) with $j=3$, where $\mathcal{L}_{3}$ is given by (\ref{eq:Lj}) with $j=3$. 
Recalling 
$M_{53}^P(n_1,\ldots,n_5) = n_{35} \sigma_4^P(n_{1},n_{2},n_{35},n_4)$, it then follows from \eqref{eq:inherited-frequency-relations} that 
\begin{equation}\label{1012-4}
	N_{1}\sim N_{35}\sim N_4\sim N_{235}\gtrsim N,
	\qquad \lambda^{-1}\leq|n_{2}|\le1.
\end{equation}
Let 
\begin{equation}\label{1012-5}
	\mathscr{M}_{53}(\vec{n}) = \frac{ |M_{53}^{P}(\vec{n})|}{\prod_{i=1}^{5} \langle n_i\rangle m(n_i)}, \qquad\forall\, \vec{n}=(n_1,n_2,n_3,n_4,n_5)\in \Gamma_{5,3}^{\text{Bad}}.
\end{equation}
Then (\ref{eq:M5j-target}) with $j=3$ boils down to 
\begin{equation}\label{1012-6}
	\int_{\sum\limits_{i=1}^{5}\tau_i = 0}
	\int_{\Gamma_{5,3}^{\mathrm{Bad}}}
	\mathscr{M}_{53}(\vec{n})
	\prod_{i=1}^{5} \frac{\widetilde f_i(n_i,\tau_i)}{\langle L_{i,3}\rangle^{1/2}}		
	\lesssim D_2(\lambda,N,\epsilon)\prod_{i=1}^{5}\|f_i\|_{L_{x,t}^{2}}.
\end{equation}
Based on \eqref{eq:sigma4j-est} and (\ref{1012-4}),
\[
	|M_{53}^{P}(\vec{n})| \lesssim \lambda^{2+2s_\alpha+\epsilon} N_1^{1+2s_\alpha+\epsilon} m^2(n_1).
\]
Meanwhile, it follows from (\ref{1012-4}) again that $\langle n_2\rangle m(n_2) \sim 1$ and 
\[
	\langle n_1\rangle \langle n_4\rangle m(n_1) m(n_4) \sim N_1^2 m^2(n_1).
\]
As a result, 
\begin{equation}\label{1012-7}
	\mathscr{M}_{53}(\vec{n}) \lesssim \frac{\lambda^{2+2s_\alpha+\epsilon} N_1^{-1+2s_\alpha+\epsilon}}{\langle n_3\rangle\langle n_5\rangle m(n_3) m(n_5)}.
\end{equation}
Since the relative sizes of $|n_3|$ and $|n_5|$ are not determined at this point, we will divide the proof into several subcases depending on the sizes of $|n_3|$ and $|n_5|$. 

Similar to the parametrization in (\ref{eq:M51-case1-parametrization}) in Section \ref{Subsec, quin_est1}, the frequencies $(n_1,n_2,n_3,n_4,n_5)$ in $\Gamma_{5,3}^{\text{Bad}}$ satisfy $(n_1, n_2, n_{35}, n_4)\in\Omega_*$, so $|n_4 + c_1 n_1| \leq \frac{\kappa_\alpha}{\lambda}$ due to the assumption (\ref{one_compo}). Setting $\theta = n_4 + c_1 n_1$, then 
\begin{equation}\label{para_M53}
	n_4 = -c_1 n_1 + \theta,
	\qquad
	n_{5} = -c_2 n_1 - n_{23} - \theta.
\end{equation}
Based on (\ref{Theta_def}), we know $\theta \in \Theta_{\kappa_\alpha}(n_1; -c_1)$, where 
\[
	\Theta_{\kappa_\alpha}(n_{1};-c_1) = 	
	\left\{\theta\in\mathbb{R}:
	|\theta|\leq\frac{\kappa_\alpha}{\lambda}
	\ \text{and}\
	-c_1 n_{1} + \theta\in\mathbb{Z}_\lambda\right\}.
\]
In addition, $\Theta_{\kappa_\alpha}(n_1; -c_1)$ contains at most one element and will be abbreviated as $\Theta(n_1; -c_1)$ in the rest of this proof.

\textbf{Case 1.} $|n_3|\sim |n_5|$.

By \eqref{1012-4} and $n_{35}=n_3+n_5$, we have
\[
	N_3\sim N_5\gtrsim N_{35} \sim N_1 \sim N_4\sim N_{235}\gtrsim N,
	\qquad \lambda^{-1}\leq|n_2| \le 1.
\]
It then follows from (\ref{1012-7}) that 
\begin{equation}\label{1012-9}
	\mathscr{M}_{53}(\vec{n}) \lesssim \frac{\lambda^{2+2s_\alpha+\epsilon} N_1^{-1+2s_\alpha+\epsilon}}{N_5^2 m^2(N_5)} \lesssim \frac{\lambda^{2+2s_\alpha+\epsilon}}{N^{3-2s_\alpha-\epsilon}}.
\end{equation}
Since the upper bound in (\ref{1012-9}) matches that in (\ref{1011-4}) for $\mathscr{M}_{51}$ in Case 1 in Section \ref{Subsec, quin_est1}, the rest of the proof can be shown in a similar way.

\textbf{Case 2.} $|n_3|\gg |n_5|$.

By \eqref{1012-4} and $|n_3|\gg|n_5|$, we have
$N_3\sim N_{35} \sim N_1 \sim N_4\sim N_{235}\gtrsim N$ and 
$\lambda^{-1}\leq|n_2|\le1$. Therefore, 
\[\begin{split}
	\langle n_3\rangle m(n_3) \langle n_5\rangle m(n_5) &\gtrsim \langle n_3\rangle m^2(n_3) \langle n_5\rangle \sim N^{2-2s} N_3^{2s-1} \langle n_5\rangle.
\end{split}\]
It then follows from (\ref{1012-7}) and the fact $N_3\sim N_1$ that 
\[
	\mathscr{M}_{53}(\vec{n}) \lesssim \frac{\lambda^{2+2s_\alpha+\epsilon}}{N^{2-2s} N_3^{2s-2s_\alpha-\epsilon}\langle n_5\rangle }. 
\]
Thanks to the assumption $\epsilon\in (0,2s-2s_\alpha)$ in Lemma \ref{lem:M5j-estimates}, the above estimate implies that 
\begin{equation}\label{1013-1}
	\mathscr{M}_{53}(\vec{n}) \lesssim \frac{\lambda^{2+2s_\alpha+\epsilon}}{N^{2-2s_\alpha-\epsilon} \langle n_5\rangle}.
\end{equation}
Since the decay for $\mathscr{M}_{53}$ in (\ref{1013-1}) is not sufficient to justify (\ref{1012-6}), we will take advantage of the lower bound $\langle K_3(\vec{n})\rangle$ for $L_{M,3}$ in (\ref{eq:LM-Kj}) to provide additional decay. 
We decompose $K_3$ as follows:
\begin{equation}\label{1013-3}
	\begin{aligned}
		K_3(n_1,\cdots,n_5)
		&=
		(n_1^3+n_2^3+\alpha n_{35}^3 +\alpha n_4^3)
		+(n_3^3+\alpha n_5^3-\alpha n_{35}^3 )       \\
		&=
		H_\alpha(n_1,n_2,n_{35},n_4)
		+R(n_3,n_5,-n_{35}) .
	\end{aligned}
\end{equation}
When $|n_3|\gg|n_5|$, $|R(n_3,n_5,-n_{35})|\gtrsim N_3^3$,
and by \eqref{eq:H-bad-upper} applied to
$(n_1,n_2,n_{35},n_4)$,
\begin{equation}\label{1013-5}
	|H_\alpha(n_1,n_2,n_{35},n_4)|\lesssim N_{35}^2.
\end{equation}
Hence, it follows from (\ref{1013-3}), (\ref{1013-5}) and the fact $N_{35} \sim N_3$ that
\begin{equation}\label{1013-2}
	L_{M,3} \gtrsim |K_3(n_1,\cdots,n_5)| \gtrsim N_3^3.
\end{equation}
Based on (\ref{1013-1}) and (\ref{1013-2}), which are parallel to the estimates (\ref{1011-3}) and (\ref{K1est}), we can justify (\ref{1012-6}) by following the similar argument as that in Case 2, Section \ref{Subsec, quin_est1}.

\textbf{Case 3.} $|n_5|\gg |n_3|$, and $|n_3|\geq \lambda^{1/2}$.	

By \eqref{1012-4} and $n_{35}=n_3+n_5$, we have
\[
	N_5\sim N_{35} \sim N_1 \sim N_4\sim N_{235}\gtrsim N,
	\qquad \lambda^{-1}\leq|n_2|\le1.
\]
Comparing this case with Case 2, the difference is the estimate for the term $R(n_3, n_5, -n_{35})$ in (\ref{1013-3}). In Case 2, the assumption $|n_3| \gg |n_5|$ and the fact $\alpha\neq 1$ together imply that 
\[
	|R(n_3, n_5, -n_{35})| = |n_3^3 +\alpha n_5^3 - \alpha n_{35}^3| \sim N_3^3.
\]
But in the current case where $|n_5| \gg |n_3|$, there is cancellation between $\alpha n_5^3$ and $\alpha n_{35}^3$. In fact, 
\begin{equation}\label{1013-4}
	|R(n_3, n_5, -n_{35})| = |n_3^3 +\alpha n_5^3 - \alpha (n_3+n_5)^3| \sim N_3 N_5^2. 
\end{equation}
So we impose a lower bound $\lambda^{\frac12}$ for $|n_3|$ in this case to guarantee a large lower bound for $|R(n_3, n_5, -n_{35})| $. The case when $|n_3| \leq \lambda^{\frac12}$ will be treated in Case 4 with a different decomposition of $K_3$. 

Recalling the decomposition (\ref{1013-3}) and the estimate (\ref{1013-5}) for $|H_\alpha(n_1, n_2, n_{35}, n_4)| $, we find 
\[
	|K_3| \geq |R(n_3, n_5, -n_{35})| - |H_\alpha(n_1, n_2, n_{35}, n_4)| \gtrsim N_3 N_5^2 - N_{35}^2 \gtrsim \lambda^{\frac12} N_5^2,
\]
which implies that 
\begin{equation}\label{1013-6}
	L_{M,3} \gtrsim |K_3| \gtrsim \lambda^{\frac12} N_5^2.
\end{equation}
Meanwhile, 
\[
	\langle n_3\rangle \langle n_5\rangle m(n_3) m(n_5) \gtrsim \langle n_3\rangle \langle n_5\rangle m^2(n_5) \sim N^{2-2s} N_5^{2s-1} \langle n_3\rangle.
\]
Plugging this estimate into (\ref{1012-7}) yields 
\begin{equation}\label{1013-7}
	\mathscr{M}_{53}(\vec{n}) 
	\lesssim \frac{\lambda^{2+2s_\alpha+\epsilon}}{N^{2-2s} N_{5}^{2s-2s_\alpha-\epsilon} \langle n_3\rangle}  
	\lesssim \frac{\lambda^{2+2s_\alpha+\epsilon}}{N^{2-2s_\alpha-\epsilon} \langle n_3\rangle}.
\end{equation}
 
Based on (\ref{1013-7}) and (\ref{1013-6}), which are parallel to the estimates (\ref{1011-3}) and (\ref{K1est}), we can justify (\ref{1012-6}) by following the similar argument as that in Case 2, Section \ref{Subsec, quin_est1}. Noting that although the lower bound in (\ref{1013-6}) is not as strong as that in (\ref{K1est}), it is sufficient to justify (\ref{1012-6}) since the relation (\ref{para_req}) has room to relax.

\textbf{Case 4.} $|n_5|\gg |n_3|$, and $|n_3|\le \lambda^{1/2}$.

Similar to Case 3, we have 
\[
	N_5\sim N_{35} \sim N_1 \sim N_4\sim N_{235}\gtrsim N,
	\qquad \lambda^{-1}\leq|n_2|\le1.
\]
As discussed previously, we will estimate $|K_3|$ in a different way by decomposing it as follows:
\begin{equation}\label{1014-3}
	\begin{aligned}
		K_3(n_1,\cdots,n_5)
		&=(n_1^3 +\alpha n_{235}^3 +\alpha n_4^3)
		+ (n_2^3+n_3^3+\alpha n_5^3-\alpha n_{235}^3) \\
		&=
		R(n_1,n_{235},n_4)
		+H_{\alpha}(n_2,n_3,n_5,-n_{235})  \\
		&=R(n_1,n_{235},n_4)
			-n_{23}\bigl(3\alpha n_5^2
			+3\alpha n_{23}n_5
			+3n_2n_3+(\alpha-1)n_{23}^2\bigr).
	\end{aligned}
\end{equation}
Since both the upper and lower bounds for $|R(n_1,n_{235},n_4)|$ have already been established (see \eqref{Rn1n4_UpperBound} and \eqref{Rn1n4_LowerBound}), it remains only to consider $|H_\alpha(n_2,n_3,n_5,-n_{235}) |$ to estimate $|K_3|$. Since $H_\alpha=0$ when $n_{23}=0$, we divide the argument into two cases depending whether $n_{23}$ vanishes.

\textbf{Subcase 4.1.} $n_{23}=0$.

In this case, we don't need additional decay from $L_{M,3}$. From $n_{23} = 0$, we can infer from the parametrization (\ref{para_M53}) that 
\begin{equation}\label{1013-10}
	n_4 = -c_1 n_1 + \theta,
	\qquad
	n_5=-c_2 n_1 - \theta,
	\qquad
	\theta \in \Theta(n_1; -c_1).
\end{equation}
In addition, since $\langle n_3\rangle = \langle n_2\rangle\sim 1$, it then follows from (\ref{1012-7}) that 
\[
		\mathscr{M}_{53}(\vec{n}) 
		\lesssim \frac{\lambda^{2+2s_\alpha+\epsilon} N_1^{-1+2s_\alpha+\epsilon}}{N^{1-s} N_5^s} 
		\lesssim \frac{\lambda^{2+2s_\alpha+\epsilon}}{N^{2-2s_\alpha-\epsilon}}.
\]
As a result, (\ref{1012-6}) is reduced to 
\[
	\frac{\lambda^{2+2s_{\alpha}+\epsilon}}{N^{2-2s_\alpha-\epsilon}} 
	\int_{\sum\limits_{i=1}^{5}\tau_i = 0}
	\int_{\Gamma_{5,3}^{\mathrm{Bad}}}
	\prod_{i=1}^{5} \frac{\widetilde f_i(n_i,\tau_i)}{\langle L_{i,3}\rangle^{1/2}}		
	\lesssim D_2(\lambda,N,\epsilon)\prod_{i=1}^{5}\|f_i\|_{L_{x,t}^{2}}.
\]

Define functions $\{F_i\}_{i=1}^{5}$ whose space-time Fourier transforms $\{\widetilde{F}_i\}_{i=1}^{5}$ satisfy
\begin{equation}\label{1013-9}
	\widetilde F_i(n_i,\tau_i)
	=
	\frac{\widetilde f_i(n_i,\tau_i)}{\langle L_{i,3}\rangle^{1/2}},
	\qquad 1\leq i\leq 5.
\end{equation}
Then based on the formula (\ref{eq:D2-def}) for $D_2(\lambda,N,\epsilon)$, the above estimate is equivalent to 
\begin{equation}\label{1013-8}
	\int_{\sum\limits_{i=1}^{5}\tau_i = 0}
	\int_{\Gamma_{5,3}^{\mathrm{Bad}}}
	\prod_{i=1}^{5} \widetilde F_i(n_i,\tau_i)
	\lesssim \lambda^{-\frac32} \prod_{i=1}^{5}\|f_i\|_{L_{x,t}^{2}}.
\end{equation}
Taking the inverse Fourier transform in the time variable, (\ref{1013-8}) is converted to
\begin{equation}\label{1013-12}
	\int_{\mathbb{R}}
	\int_{\Gamma_{5,3}^{\mathrm{Bad}}}
	\prod_{i=1}^{5} \bigl| \widehat F_i(n_i, t) \bigr| \,dt
	\lesssim \lambda^{-\frac32} \prod_{i=1}^{5}\|f_i\|_{L_{x,t}^{2}},
\end{equation}
where $\widehat{F}_{i}$ represents the Fourier transform of $F_i$ in the spatial variable. 

Using the relation $n_3 = -n_2$ and the variable $\theta$ in \eqref{1013-10}, we find
\[\begin{split}
	\text{LHS of (\ref{1013-12})} = \Big(\frac{1}{2\pi\lambda}\Big)^4 \int_{\mathbb{R}} &\sum_{n_1\in\mathbb{Z}_\lambda} \sum_{\frac{1}{\lambda}\leq |n_2|\leq 1} \sum_{\theta\in\Theta(n_{1};-c_1)} \Big[\bigl| \widehat{F}_1(n_1, t) \bigr| \bigl| \widehat{F}_2(n_2, t) \bigr| \\
	& \bigl| \widehat{F}_3(-n_2, t) \bigr| \bigl| \widehat{F}_4(-c_1 n_{1} + \theta, t) \bigr| \bigl| \widehat{F}_5(-c_2 n_{1} - \theta, t) \bigr|\Big] \,dt.
\end{split}\]
Rearranging the order of integration yields 
\begin{equation}\label{1013-13}\begin{split}
	\text{LHS of (\ref{1013-12})} \lesssim \frac{1}{\lambda^4} \int_{\mathbb{R}} & \|\widehat{F}_5(\cdot, t)\|_{\ell^\infty} \sum_{n_1\in\mathbb{Z}_\lambda} \big| \widehat{F}_1(n_1, t)\big| \sum_{\theta\in\Theta(n_{1};-c_1)} \big|\widehat{F}_4(-c_1 n_{1} + \theta, t)\big| \\
	& \sum_{n_2\in\mathbb{Z}_\lambda} \Big| \widehat{F}_2(n_2, t) \widehat{F}_3(-n_2, t) \Big| \,dt.
\end{split}\end{equation}
Then similar to the derivation from (\ref{1011-9}) to \eqref{1011-10}, we find 
\[
	\text{LHS of (\ref{1013-12})} \lesssim \frac{1}{\lambda^2}\int_{\mathbb{R}} \|\widehat{F}_5(\cdot, t)\|_{\ell^\infty} \prod_{i=1}^{4} \|\widehat F_i(\cdot, t)\|_{\ell^2} \,dt.
\]
Moreover,
\begin{equation*}
	\|\widehat F_5(\cdot, t)\|_{\ell^\infty}
	\le \|F_5(\cdot, t)\|_{L_x^1}\le 
	\lambda^{1/2}
	\| F_5(\cdot, t)\|_{L^2_{x}}.
\end{equation*}
Therefore, 
\[
	\text{LHS of (\ref{1013-12})} \lesssim \lambda^{-3/2} \int_{\mathbb{R}} \| F_5(\cdot, t)\|_{L^2_{x}} \prod_{i=1}^{4} \|\widehat F_i(\cdot, t)\|_{\ell^2} \,dt.
\]
Now applying the H\"older inequality in time and Plancherel identity yields 
\begin{equation}\label{1013-14}
	\text{LHS of (\ref{1013-12})} \lesssim \lambda^{-3/2} \| F_5\|_{L_{x,t}^{2,2}} \prod_{i=1}^{4} \|F_i\|_{L_{x,t}^{2,8}}.
\end{equation}
Finally, according to the definition (\ref{1013-9}) and Lemma \ref{lemma3}, it then follows from (\ref{1013-14}) that 
\[
	\text{LHS of (\ref{1013-12})} \lesssim \lambda^{-\frac32} \prod_{i=1}^{5}\|f_i\|_{L_{x,t}^{2}} = \text{RHS of (\ref{1013-12})}.
\]
Thus, (\ref{1013-12}) is justified and this case is settled.

\textbf{Subcase 4.2.} $n_{23}\neq 0$.

Since $n_2, n_3\in \mathbb{Z}_\lambda$ and $n_{23}\neq 0$, we have
\[
	N_{23}\ge \frac1\lambda.
\]
By condition \eqref{eq:condNandlam} and the bounds
$N_3\le \lambda^{1/2}$ and $N_2\le 1$, we obtain
$N_{23}\le N_2+N_3\le 2 \lambda^{1/2}\ll N^{1/2}$.
Thus with $N_5\gtrsim N$, where $N$ is sufficiently large, then
\begin{equation*}
	\begin{aligned}
		\left|
		3\alpha n_5^2
		+3\alpha n_{23}n_5
		+3n_2n_3+(\alpha-1)n_{23}^2
		\right|
		&\ge
		3\alpha n_5^2
		-3\alpha |n_{23} n_5|
		-3|n_2 n_3| -|\alpha-1| n_{23}^2 
		\geq \alpha |n_5|^2,
	\end{aligned}
\end{equation*}
which implies that 
\begin{equation}\label{1014-1}
	|H_\alpha(n_2, n_3, n_5, -n_{235})| = N_{23} | 3\alpha n_5^2
	+3\alpha n_{23}n_5
	+3n_2n_3+(\alpha-1)n_{23}^2 | \geq \frac{\alpha n_5^2}{\lambda}.
\end{equation}
On the other hand, it follows from \eqref{Rn1n4_UpperBound} that
\[
	|R(n_1,n_{235},n_4)| \le 2\alpha\rho_\alpha \kappa_\alpha \,\frac{|n_1|^2}{\lambda}\le \frac{\alpha |n_1|^2}{8 C_*^2\lambda },
\]
where $C_*=C_*(\alpha)\geq1$ is the constant defined in \eqref{eq:C-alpha}. Meanwhile, $C_*^{-1}|n_{1}|\le |n_{35}|$ 
according to Lemma \ref{lemma:sigma4}, so
\begin{equation}\label{1014-2}
	|R(n_1,n_{235},n_4)| \leq \frac{\alpha n_{35}^2}{8\lambda}.
\end{equation}
The combination of (\ref{1014-3}), (\ref{1014-1}) and \eqref{1014-2} yields 
\begin{equation*}\begin{split}
	|K_3(n_1,n_2,\cdots,n_5)|
	&\geq |H_\alpha(n_2,n_3,n_5,-n_{235})|-|R(n_1,n_{235},n_4)| \\
	&\geq \frac{\alpha n_5^2}{\lambda} - \frac{\alpha n_{35}^2}{8\lambda} \geq \frac{\alpha n_5^2}{2\lambda},
\end{split}\end{equation*}
where the last inequality is due to $|n_{35}| \leq 2|n_5|$. Thus, we obtain
\begin{equation}\label{1014-5}
	L_{M,3} \gtrsim |K_3(n_1,n_2,\cdots,n_5)| \gtrsim \frac{N_5^2}{\lambda} \gtrsim \frac{N^2}{\lambda}.
\end{equation}
On the other hand, since $\langle n_3\rangle m(n_3) \gtrsim 1$ and 
\[
	\langle n_5\rangle m(n_5) \sim N^{1-s} N_5^{s} \sim N^{1-s} N_1^{s}, 
\]
it follows from (\ref{1012-7}) that 
\begin{equation}\label{1014-4}
	\mathscr{M}_{53}(\vec{n}) \lesssim \frac{\lambda^{2+2s_\alpha+\epsilon} N_1^{-1+2s_\alpha+\epsilon}}{N^{1-s} N_1^{s}} \lesssim \frac{\lambda^{2+2s_\alpha+\epsilon}}{N^{2-2s_\alpha-\epsilon}}.
\end{equation}

Based on (\ref{1014-4}) and (\ref{1014-5}), (\ref{1012-6}) is reduced to
\[
	\frac{\lambda^{2+2s_\alpha+\epsilon}}{N^{2-2s_\alpha-\epsilon}}\Big(\frac{\lambda}{N^2}\Big)^{\frac12}
	\int_{\sum\limits_{i=1}^{5}\tau_i = 0}
	\int_{\Gamma_{5,3}^{\mathrm{Bad}}}
	L_{M,3}^{1/2}
	\prod_{i=1}^{5} \frac{\widetilde f_i(n_i,\tau_i)}{\langle L_{i,3}\rangle^{1/2}}		
	\lesssim D_2(\lambda,N,\epsilon)\prod_{i=1}^{5}\|f_i\|_{L_{x,t}^{2}}.
\]
According to the formula (\ref{eq:D2-def}) for $D_2(\lambda,N,\epsilon)$, the above estimate is equivalent to 
\begin{equation}\label{1014-6}
	\int_{\sum\limits_{i=1}^{5}\tau_i = 0}
	\int_{\Gamma_{5,3}^{\mathrm{Bad}}}
	L_{M,3}^{1/2}
	\prod_{i=1}^{5} \frac{\widetilde f_i(n_i,\tau_i)}{\langle L_{i,3}\rangle^{1/2}}		
	\lesssim \lambda^{-2} N \prod_{i=1}^{5}\|f_i\|_{L_{x,t}^{2}}.
\end{equation}
Since there exists $i\in\{1,2,3,4,5\}$ such that $L_{M,3} = \langle L_{i,3}\rangle$, we assume, without loss of generality, that $L_{M,3} = \langle L_{1,3} \rangle$. Other cases can be treated similarly. Define functions $\{F_i\}_{i=1}^{5}$ whose space-time Fourier transforms $\{\widetilde{F}_i\}_{i=1}^{5}$ satisfy
\begin{equation}\label{1014-9}
	\widetilde F_1(n_1,\tau_1) = \widetilde f_1(n_1,\tau_1),
	\qquad
	\widetilde F_i(n_i,\tau_i)
	=
	\frac{\widetilde f_i(n_i,\tau_i)}{\langle L_{i,3}\rangle^{1/2}},
	\qquad 2\le i\le 5,
\end{equation}
then (\ref{1014-6}) is converted to
\[
	\int_{\sum\limits_{i=1}^{5}\tau_i = 0}
	\int_{\Gamma_{5,3}^{\mathrm{Bad}}}
	\prod_{i=1}^{5} \widetilde F_i(n_i,\tau_i) 
	\lesssim \lambda^{-2} N \prod_{i=1}^{5}\|f_i\|_{L_{x,t}^{2}},
\]
which, via taking the inverse Fourier transform in the time variable, is further equivalent to
\begin{equation}\label{1014-7}
	\int_{\mathbb{R}} \int_{\Gamma_{5,3}^{\mathrm{Bad}}}
	\prod_{i=1}^{5} \bigl| \widehat F_i(n_i, t) \bigr| \,dt
	\lesssim \lambda^{-2} N \prod_{i=1}^{5}\|f_i\|_{L_{x,t}^{2}},
\end{equation}
where $\widehat{F}_{i}$ represents the Fourier transform of $F_i$ in the spatial variable. 
Using the variable $\theta$ in (\ref{para_M53}), it follows from (\ref{1014-7}) that 
\[\begin{split}
	\text{LHS of (\ref{1014-7})} = \Big(\frac{1}{2\pi\lambda}\Big)^4 \int_{\mathbb{R}} &\sum_{n_1\in\mathbb{Z}_\lambda} \sum_{\frac{1}{\lambda}\leq |n_2|\leq 1} \sum_{n_3\in\mathbb{Z}_\lambda} \sum_{\theta\in\Theta(n_{1};-c_1)} \Big[\bigl| \widehat{F}_1(n_1, t) \bigr| \bigl| \widehat{F}_2(n_2, t) \bigr| \\
	& \bigl| \widehat{F}_3(n_3, t) \bigr| \bigl| \widehat{F}_4(-c_1 n_{1} + \theta, t) \bigr| \bigl| \widehat{F}_5(-c_2 n_{1} - n_2 - n_3 - \theta, t) \bigr|\Big] \,dt.
\end{split}\]

Rearranging the order of the integration yields 
\[\begin{split}
	\text{LHS of (\ref{1014-7})} \lesssim \frac{1}{\lambda^4} \int_{\mathbb{R}} &
	\sum_{\frac{1}{\lambda}\leq |n_2|\leq 1} |\widehat{F}_2(n_2, t)| \bigg[\sum_{n_1\in\mathbb{Z}_\lambda} |\widehat{F}_1(n_1, t)| \sum_{\theta\in\Theta(n_1;-c_1)} |\widehat{F}_4(-c_1 n_1 + \theta, t)| \\
	& \Big(\sum_{n_3\in\mathbb{Z}_\lambda} |\widehat{F}_3(n_3, t) \widehat{F}_5(-c_2 n_{1} - n_2 - n_3 - \theta, t)| \Big)\bigg] \,dt.
\end{split}\]
Then similar to the derivation from (\ref{1011-9}) to \eqref{1011-10}, we find 
\[
	\text{LHS of (\ref{1014-7})} \lesssim \frac{1}{\lambda^2}\int_{\mathbb{R}} 
	\sum_{\frac{1}{\lambda}\leq |n_2|\leq 1} |\widehat{F}_2(n_2, t)|
	 \Big( \big\| \widehat{F}_1(\cdot, t) \big\|_{\ell^2} \big\| \widehat{F}_4(\cdot, t) \big\|_{\ell^2}\Big) \Big(\| \widehat{F}_3(\cdot, t) \|_{\ell^2} \| \widehat{F}_5(\cdot, t) \|_{\ell^2}\Big) \,dt.
\]
Meanwhile, plugging the relation (\ref{eq:M51-low-frequency-sum}) into the above estimate gives 
\[
	\text{LHS of (\ref{1014-7})} \lesssim \frac{1}{\lambda} \int_{\mathbb{R}} 
	\prod_{i=1}^{5} \|\widehat{F}_i(\cdot, t)\|_{\ell^2(\mathbb{Z}_\lambda)} \,dt.
\]
Now applying the H\"older inequality in time yields 
\begin{equation}\label{1014-8}
	\text{LHS of (\ref{1014-7})} \lesssim \frac{1}{\lambda} \| F_1\|_{L_t^2 L_x^2} \prod_{i=2}^{5} \|F_i\|_{L_t^8 L_x^2} \lesssim \frac{1}{\lambda} \prod_{i=1}^{5} \|f_i\|_{L_{x,t}^{2,2}},
\end{equation}
where the last inequality is due to (\ref{1014-9}) and Lemma \ref{lemma3}.

Finally, we conclude that (\ref{1014-8}) implies (\ref{1014-7}) since $\lambda \lesssim N$. Hence, this case is finished and (\ref{eq:M53-est}) is proved.

\subsection{Proof of \eqref{eq:M54-est}}
\label{Subsec, M54proof}

As the discussion in Section \ref{Subsec, setup}, the estimate (\ref{eq:M54-est}) is reduced to (\ref{eq:M5j-target}) with $j=4$, where $\mathcal{L}_{4}$ is given by (\ref{eq:Lj}) with $j=4$. 
Recalling 
$M_{54}^P(n_1,\ldots,n_5) = n_{45} \sigma_4^P(n_{1},n_{2},n_{3},n_{45})$, it then follows from \eqref{eq:inherited-frequency-relations} that 
\begin{equation}\label{1015-1}
	N_{1}\sim N_{3}\sim N_{45}\sim N_{23}\gtrsim N,
	\qquad \lambda^{-1}\leq|n_{2}|\le1.
\end{equation}
Let 
\begin{equation}\label{1015-2}
	\mathscr{M}_{54}(\vec{n}) = \frac{ |M_{54}^{P}(\vec{n})|}{\prod_{i=1}^{5} \langle n_i\rangle m(n_i)}, \qquad\forall\, \vec{n}=(n_1,n_2,n_3,n_4,n_5)\in \Gamma_{5,4}^{\text{Bad}}.
\end{equation}
Then (\ref{eq:M5j-target}) with $j=4$ boils down to 
\begin{equation}\label{1015-3}
	\int_{\sum\limits_{i=1}^{5}\tau_i = 0}
	\int_{\Gamma_{5,4}^{\mathrm{Bad}}}
	\mathscr{M}_{54}(\vec{n})
	\prod_{i=1}^{5} \frac{\widetilde f_i(n_i,\tau_i)}{\langle L_{i,4}\rangle^{1/2}}		
	\lesssim D_2(\lambda,N,\epsilon)\prod_{i=1}^{5}\|f_i\|_{L_{x,t}^{2}}.
\end{equation}
Based on \eqref{eq:sigma4j-est} and (\ref{1015-1}),
\[
	|M_{54}^{P}(\vec{n})| \lesssim \lambda^{2+2s_\alpha+\epsilon} N_1^{1+2s_\alpha+\epsilon} m^2(n_1).
\]
Meanwhile, it follows from (\ref{1015-1}) again that $\langle n_2\rangle m(n_2) \sim 1$ and 
\[
	\langle n_1\rangle \langle n_3\rangle m(n_1) m(n_3) \sim N_1^2 m^2(n_1).
\]
As a result, 
\begin{equation}\label{1015-4}
	\mathscr{M}_{54}(\vec{n}) \lesssim \frac{\lambda^{2+2s_\alpha+\epsilon} N_1^{-1+2s_\alpha+\epsilon}}{\langle n_4\rangle\langle n_5\rangle m(n_4) m(n_5)}.
\end{equation}
Since the relative sizes of $|n_4|$ and $|n_5|$ are not determined at this point, we will divide the proof into several subcases depending on the sizes of $|n_4|$ and $|n_5|$. 

Similar to the parametrization in (\ref{eq:M51-case1-parametrization}) in Section \ref{Subsec, quin_est1}, the frequencies $(n_1,n_2,n_3,n_4,n_5)$ in $\Gamma_{5,4}^{\text{Bad}}$ satisfy $(n_1, n_2, n_{3}, n_{45})\in\Omega_*$, so $|n_{45} + c_1 n_1| \leq \frac{\kappa_\alpha}{\lambda}$ due to the assumption (\ref{one_compo}). Setting $\theta_1 = n_{45} + c_1 n_1$, then 
\begin{equation}\label{para_M54}
	n_{45} = -c_1 n_1 + \theta_1,
	\qquad
	n_{23} = -c_2 n_1 - \theta_1.
\end{equation}
Based on (\ref{Theta_def}), we know $\theta_1 \in \Theta_{\kappa_\alpha}(n_1; -c_1)$ which contains at most one element and will be abbreviated as $\Theta(n_1; -c_1)$ in the rest of this proof. We point out that the notation $\theta_1$ here is used to distinguish from $\theta_2$ which will be introduced later in (\ref{para2_sparsity}).

\textbf{Case 1.} $|n_5|\sim |n_4|$.

This case can be justified in a similar way as that for Case 1 in Section \ref{Subsec, M53proof}, so we omit the details.

\textbf{Case 2.} $|n_4|\gg |n_5|$.

This case can be proved analogously as that for Case 2 in Section \ref{Subsec, M53proof}, so we skip the argument.

\textbf{Case 3.} $|n_5|\gg |n_4|$ and $|n_4| \geq \lambda^{1/2}$.

This case is parallel to Case 3 in Section \ref{Subsec, M53proof} by switching the role of $n_3$ and $n_4$, so we omit the proof.

\textbf{Case 4.} $|n_5|\gg |n_4|$ and $|n_4| \leq \lambda^{1/2}$.

In this case, the combination of \eqref{1015-1} and $n_{45}=n_4+n_5$ implies
\[
	N_5\sim N_{45} \sim N_1\sim N_3 \sim N_{23}\gtrsim N,
	\qquad \lambda^{-1}\leq|n_2|\le1.
\]
So $\langle n_4\rangle m(n_4) \gtrsim 1$ and 
\[
	\langle n_5\rangle m(n_5) \sim N^{1-s} N_{5}^{s} \sim N^{1-s} N_{1}^{s}.
\]
Plugging this estimate into (\ref{1015-4}) yields 
\begin{equation}\label{1015-5}
	\mathscr{M}_{54}(\vec{n}) 
	\lesssim \frac{\lambda^{2+2s_\alpha+\epsilon}}{N^{1-s} N_{1}^{1+s-2s_\alpha-\epsilon} }  
	\lesssim \frac{\lambda^{2+2s_\alpha+\epsilon}}{N^{2-2s_\alpha-\epsilon}}.
\end{equation}

Since the upper bound in (\ref{1015-5}) is not sufficiently small, we next derive a lower bound for $|K_4|$ to provide additional decay as we did previously. Since both $n_2$ and $n_4$ are small in this case, we rewrite the parametrization (\ref{para_M54}) as 
\begin{equation}\label{para2_M54}
	n_{5} = -c_1 n_1 - n_4 + \theta_1,
	\qquad
	n_{3} = -c_2 n_1 - n_2 - \theta_1.
\end{equation}
Based on (\ref{para2_M54}) and (\ref{eq:Kj-LMj-def}), we can express $K_4$ in terms of $\{n_1, n_2, n_4, \theta_1\}$:
\[\begin{split}
	K_4(n_1,\ldots,n_5) &= n_1^3 + n_2^3 + \alpha n_3^3 + n_4^3 + \alpha n_5^3 \\
	&= n_1^3 + n_2^3 + \alpha (-c_2 n_1 - n_2 - \theta_1)^3 + n_4^3 + \alpha (-c_1 n_1 - n_4 + \theta_1)^3.
\end{split}\]
Noticing that $N_1$ is the dominating term compared to $N_2$, $N_4$ and $|\theta_1|$, so we arrange the terms in $K_4$ according to the powers of $n_1$:
\begin{equation}\label{K_4decop}
	K_4(n_1,\ldots,n_5) = - 3\alpha n_1^2 (c_2^2n_2+c_1^2n_4) + Q,
\end{equation}
where
\begin{equation*}
	\begin{aligned}
		Q:=&
		\underbrace{-3\alpha(c_2^2-c_1^2)\theta_1n_1^2}_{Q_1}
		\underbrace{-3\alpha n_1
		\bigl[c_2(n_2+\theta_1)^2+c_1(n_4-\theta_1)^2\bigr]}_{Q_2} \\
		&+ \underbrace{n_2^3+n_4^3
		-\alpha\bigl[(n_2+\theta_1)^3+(n_4-\theta_1)^3\bigr]}_{Q_3}.
	\end{aligned}
\end{equation*}
Since $|c_2^2-c_1^2|=|c_2-c_1| = \rho_\alpha/3$ and
$|\theta_1|\leq\kappa_\alpha/\lambda$, we have
\[
	|Q_1| \leq\frac{\alpha \rho_\alpha \kappa_\alpha}{\lambda}|n_1|^2
	\leq
	\frac{\alpha}{16C_*^2}\frac{N_1^2}{\lambda},
\]
where the last inequality takes advantage of the choice of $\kappa_\alpha$ in \eqref{eq:kappa-alpha-choice}.
Moreover, using $|n_2|\leq1$, $|n_4|\leq\lambda^{1/2}$, and
$|\theta_1|\leq\kappa_\alpha/\lambda\leq \frac18$, we obtain
\[\begin{split}
	|Q_2| \leq 3\alpha N_1 (|c_1|+|c_2|)(4\lambda) = 12 \alpha (|c_1|+|c_2|) N_1\lambda, \qquad
	|Q_3| \leq 2\lambda^{\frac32} + 16\alpha \lambda ^{\frac32}.
\end{split}\]
The combination of the estimates of $Q_1$, $Q_2$ and $Q_3$ together yields 
\begin{equation*}
	|Q| \leq
	\,\frac{\alpha}{16C_*^2}\frac{N_1^2}{\lambda} + 12 \alpha (|c_1|+|c_2|) N_1\lambda + 2\lambda^{\frac32} + 16\alpha \lambda ^{\frac32} .
\end{equation*}
Thanks to the assumption $\lambda\lesssim N^{\frac13}$ in (\ref{eq:condNandlam}) in Lemma \ref{lem:M5j-estimates}, we know 
\[
	\frac{N_1^2}{\lambda} \gtrsim \frac{N^2}{\lambda} \gg N\lambda \gg \lambda^{\frac32}.
\]
So by taking $N$ sufficiently large depending only on $\alpha$, the upper bound for $|Q_1|$ dominates $|Q_2|+|Q_3|$, which implies that 
\begin{equation}\label{Qest}
	|Q| \leq 2\cdot  \frac{\alpha}{16C_*^2}\frac{N_1^2}{\lambda}=
	\frac{\alpha}{8C_*^2}\frac{N_1^2}{\lambda}.
\end{equation}
In view of \eqref{K_4decop} and \eqref{Qest}, $|K_4|$ possesses a lower bound as follows:
\begin{equation}\label{K4_lbd}
	|K_4(n_1,\dots, n_5)| \gtrsim 3\alpha N_1^2 |c_2^2 n_2 + c_1^2 n_4| - \frac{\alpha}{8C_*^2}\frac{N_1^2}{\lambda} = \alpha N_1^2\Big(3|c_2^2 n_2 + c_1^2 n_4| - \frac{1}{8C_*^2 \lambda}\Big).
\end{equation}

In order to find a lower bound for $|K_4|$, the order of the term $|c_2^2 n_2 + c_1^2 n_4|$ needs to be at least $\frac{1}{\lambda}$, so we divide the rest of the argument into two subcases according to the size of
$|c_2^2 n_2+c_1^2 n_4|$.

\textbf{Subcase 4.1.}
$|c_2^2n_2+c_1^2n_4|>\frac{1}{\lambda}$.

Under this assumption, it then follows from (\ref{K4_lbd}) that 
\begin{equation}\label{eq:K4-lower-nonres}
	|K_4(n_1,\ldots,n_5)| \geq \frac{\alpha N_1^2}{\lambda} \Big( 3 - \frac{1}{8 C_*^2} \Big) \gtrsim \frac{N^2}{\lambda}.
\end{equation}
Noticing that the bounds in \eqref{1015-5} and \eqref{eq:K4-lower-nonres} are the same as the bounds in \eqref{1014-4} and \eqref{1014-5} in the proof for the Subcase 4.2 in Section \ref{Subsec, M53proof}, so we can follow the similar argument after (\ref{1014-4}) in that case to finish the proof.

\textbf{Subcase 4.2.}
$|c_2^2 n_2 + c_1^2 n_4| \leq \lambda^{-1}$.

In this case, the estimate (\ref{K4_lbd}) in the preceding argument does not yield a useful lower bound for
$|K_4|$, but the assumption $|c_2^2n_2+c_1^2n_4|\leq\lambda^{-1}$ provides additional sparsity of the frequencies that we can make use of. More precisely, the condition $|c_2^2n_2+c_1^2n_4|\leq\lambda^{-1}$, i.e.,
\begin{equation}\label{sparsity_cond2}
	\Big| n_2 + \frac{c_1^2}{c_2^2} n_4 \Big| \leq \frac{1}{c_2^2 \lambda},
\end{equation}
implies that, for each fixed $n_4$, $n_2$ is restricted to an interval of length $O(\lambda^{-1})$. Hence, there are only $O(1)$ admissible values
of $n_2$. Let 
\begin{equation}\label{para2_sparsity}
	\theta_2 = n_2 + \frac{c_1^2}{c_2^2} n_4.
\end{equation}
Then $n_2$ can be written as
$n_2 = - \frac{c_1^2}{c_2^2}\,n_4 + \theta_2$.
According to definition (\ref{Theta_def}) in Lemma \ref{lem:admissible-frequency-sum}, we know $\theta_2\in\Theta_{1/{c_2^2}}(n_4;-c_1^2/c_2^2)$,
where
\[
	\Theta_{1/{c_2^2}}(n_4;-c_1^2/c_2^2)
	:=
	\left\{\theta_2\in\mathbb R:
	|\theta_2|\leq\frac{1}{c_2^2\lambda}
	\ \text{and}\
	-\frac{c_1^2}{c_2^2}n_4+\theta_2\in\mathbb Z_\lambda\right\}.
\]
To simplify the notation, and to distinguish it from $\Theta$, we write
\[
	\widetilde{\Theta}(n_4;-c_\#)
	:=\Theta_{1/c_2^2}(n_4;-c_\#), \quad\text{where}\quad c_{\#}:=c_1^2/c_2^2.
\]

As a summary, for each fixed $n_1$, the parameter
$\theta_1$ ranges over $\Theta(n_1;-c_1)$, while for each fixed
$n_4$, the parameter $\theta_2$ ranges over
$\widetilde{\Theta}(n_4;-c_\#)$. Both admissible sets $\Theta(n_1;-c_1)$ and $\widetilde{\Theta}(n_4;-c_\#)$ have finite cardinality. Based on (\ref{para2_M54}) and (\ref{para2_sparsity}), for any frequency vector $(n_1, \dots, n_5)$ in Subcase 4.2, we parameterize it as 
\begin{equation}\label{para_M54freq}
	(n_1, \dots, n_5) = \big(n_1, \, -c_\# n_4+\theta_2, \, -c_2 n_1 + c_\# n_4 - \theta_1 - \theta_2, \, n_4, \, -c_1 n_1 - n_4 + \theta_1\big).
\end{equation}
According to the estimate (\ref{1015-5}) for $\mathscr{M}_{54}(\vec{n})$ and the formula (\ref{eq:D2-def}) for $D_2(\lambda, N, \epsilon)$, the target estimate (\ref{1015-3}) is reduced to 
\begin{equation}\label{1015-6}
	\int_{\sum\limits_{i=1}^{5}\tau_i = 0}
	\int_{\Gamma_{5,4}^{\mathrm{Bad}}}
	\prod_{i=1}^{5} \frac{\widetilde f_i(n_i,\tau_i)}{\langle L_{i,4}\rangle^{1/2}}		
	\lesssim \lambda^{-\frac32} \prod_{i=1}^{5}\|f_i\|_{L_{x,t}^{2}}.
\end{equation}
Define functions $\{F_i\}_{i=1}^{5}$ whose space-time Fourier transforms $\{\widetilde{F}_i\}_{i=1}^{5}$ satisfy
\begin{equation}\label{1015-7}
	\widetilde F_i(n_i, \tau_i)
	=
	\frac{\widetilde f_i(n_i,\tau_i)}{\langle L_{i,4}\rangle^{1/2}},
	\qquad 1\le i\le 5.
\end{equation}
Then (\ref{1015-6}) is converted to
\[
	\int_{\sum\limits_{i=1}^{5}\tau_i = 0}
	\int_{\Gamma_{5,4}^{\mathrm{Bad}}}
	\prod_{i=1}^{5} \widetilde{F}_i(n_i,\tau_i)
	\lesssim \lambda^{-\frac32} \prod_{i=1}^{5}\|f_i\|_{L_{x,t}^{2}},
\]
which, via the inverse Fourier transform in time, is further reduced to 
\begin{equation}\label{1015-8}
	\int_{\mathbb{R}}
	\int_{\Gamma_{5,4}^{\mathrm{Bad}}}
	\prod_{i=1}^{5} \bigl| \widehat{F}_i(n_i, t) \bigr| \,dt
	\lesssim \lambda^{-\frac32} \prod_{i=1}^{5}\|f_i\|_{L_{x,t}^{2}},
\end{equation}
where $\widehat{F}_{i}$ represents the Fourier transform of $F_i$ in the spatial variable. 

Then applying the parametrization (\ref{para_M54freq}) to (\ref{1015-8}) yields 
\[\begin{split}
	\text{LHS of (\ref{1015-8})} \lesssim \frac{1}{\lambda^4} \int_{\mathbb{R}} &\sum_{n_1\in\mathbb{Z}_\lambda} \sum_{n_4\in\mathbb{Z}_\lambda} \sum_{\theta_1 \in \Theta(n_1;-c_1)} \sum_{\theta_2 \in \widetilde{\Theta}(n_4;-c_\#)} |\widehat{F}_1(n_1, t)| |\widehat{F}_4(n_4, t)| |\widehat{F}_2(-c_\# n_4 + \theta_2,t)| \\
	& |\widehat{F}_3(-c_2 n_1 + c_\# n_4 - \theta_1 - \theta_2,t)| |\widehat{F}_5(-c_1 n_1 - n_4 + \theta_1,t)| \,dt.
\end{split}\]
For each fixed $t$, by bounding $|\widehat{F}_3(-c_2 n_1 + c_\# n_4 - \theta_1 - \theta_2, t)| $ via its $\|\widehat{F}_3(\cdot, t)\|_{\ell^\infty}$ norm, and then rearranging the order of integration, we obtain
\begin{equation}\label{1016-1}\begin{split}
	\text{LHS of (\ref{1015-8})} \lesssim \frac{1}{\lambda^4} \int_{\mathbb{R}} & \|\widehat{F}_3(\cdot,t)\|_{\ell^\infty} \bigg[ \sum_{n_4\in\mathbb{Z}_\lambda} |\widehat{F}_4(n_4, t)| \sum_{\theta_2 \in \widetilde{\Theta}(n_4;-c_\#)} |\widehat{F}_2(-c_\# n_4 + \theta_2,t)| \\
	& \Big(\sum_{n_1\in\mathbb{Z}_\lambda} |\widehat{F}_1(n_1, t)| \sum_{\theta_1 \in \Theta(n_1;-c_1)}  |\widehat{F}_5(-c_1 n_1 - n_4 + \theta_1,t)| \Big) \bigg] \,dt.
\end{split}\end{equation}

For fixed $t$ and $n_4$, we apply the
Cauchy-Schwarz inequality in $n_1$ and Lemma~\ref{lem:admissible-frequency-sum} (with $A=\kappa_\alpha$, $a=-c_1$ and $m=-n_4$) to find
\begin{equation}\label{1016-2}
	\begin{aligned}
		&\frac{1}{\lambda}\sum_{n_1\in\mathbb{Z}_\lambda}
		\big|\widehat F_1(n_1, t)\big|
		\sum_{\theta_1\in\Theta(n_1;-c_1)}
		\big|\widehat F_5(-c_1n_1 - n_4 +\theta_1, t)\big|\\
		\lesssim \,\,&
		\|\widehat F_1(\cdot, t)\|_{\ell^2_{n_1}}
		\Big\|
		\sum_{\theta_1\in\Theta(n_1;-c_1)}
		\big|\widehat F_5(-c_1n_1 - n_4+\theta_1,t)\big|
		\Big\|_{\ell^2_{n_1}}\\
		\lesssim & 
		\|\widehat F_1(\cdot, t)\|_{\ell^2}
		\|\widehat F_5(\cdot, t)\|_{\ell^2}.
	\end{aligned}
\end{equation}
Plugging (\ref{1016-2}) into (\ref{1016-1}) yields 
\begin{equation}\label{1016-3}\begin{split}
	\text{LHS of (\ref{1015-8})} \lesssim \frac{1}{\lambda^3} \int_{\mathbb{R}} & \|\widehat{F}_3(\cdot,t)\|_{\ell^\infty} \|\widehat F_1(\cdot, t)\|_{\ell^2}
	\|\widehat F_5(\cdot, t)\|_{\ell^2}\\
	& \Big(  \sum_{n_4\in\mathbb{Z}_\lambda} |\widehat{F}_4(n_4, t)| \sum_{\theta_2 \in \widetilde{\Theta}(n_4;-c_\#)} |\widehat{F}_2(-c_\# n_4 + \theta_2,t)| \Big) \,dt.
\end{split}\end{equation}
Again, we apply the Cauchy-Schwarz inequality in $n_4$ and
Lemma~\ref{lem:admissible-frequency-sum} (with
$A=1/c_2^2$, $a=-c_\#$ and $m=0$) to get
\begin{equation}\label{1016-4}
	\begin{aligned}
		&\frac{1}{\lambda}\sum_{n_4\in\mathbb{Z}_\lambda}
		\big|\widehat F_4(n_4,t)\big|
		\sum_{\theta_2\in\widetilde{\Theta}(n_4;-c_\#)}
		\big|\widehat F_2\left(-c_\# n_4+\theta_2, t\right)\big|\\
		\lesssim\,\, &
		\|\widehat F_4(\cdot, t)\|_{\ell^2_{n_4}}
		\Big\|
		\sum_{\theta_2\in\widetilde{\Theta}(n_4;-c_\#)}
		\big|\widehat F_2\left(-c_\# n_4+\theta_2, t\right)\big|
		\Big\|_{\ell^2_{n_4}}\\
		\lesssim\,\, &
		\|\widehat F_4(\cdot,t)\|_{\ell^2}
		\|\widehat F_2(\cdot,t)\|_{\ell^2}.
	\end{aligned}
\end{equation}
Putting (\ref{1016-4}) into (\ref{1016-3}) gives
\begin{equation}\label{1016-5}\begin{split}
	\text{LHS of (\ref{1015-8})} \lesssim \frac{1}{\lambda^2} \int_{\mathbb{R}} & \|\widehat{F}_3(\cdot,t)\|_{\ell^\infty} 
	\Big(\|\widehat F_1(\cdot, t)\|_{\ell^2}
	\|\widehat F_5(\cdot, t)\|_{\ell^2}\Big) 
	\Big( \|\widehat F_4(\cdot, t)\|_{\ell^2}
	\|\widehat F_2(\cdot, t)\|_{\ell^2} \Big)
	\,dt.
\end{split}\end{equation}

By Plancherel identity, (\ref{1016-5}) implies that 
\begin{equation}\label{1016-6}\begin{split}
	\text{LHS of (\ref{1015-8})} \lesssim \frac{1}{\lambda^2} \int_{\mathbb{R}} & \|\widehat{F}_3(\cdot,t)\|_{\ell^\infty} 
	\Big(\| F_1(\cdot, t)\|_{L^2(\mathbb{T}_\lambda)}
	\| F_5(\cdot, t)\|_{L^2(\mathbb{T}_\lambda)}\Big) \\
	&\Big( \| F_4(\cdot, t)\|_{L^2(\mathbb{T}_\lambda)}
	\| F_2(\cdot, t)\|_{L^2(\mathbb{T}_\lambda)} \Big)
	\,dt.
\end{split}\end{equation}
Noticing that
\begin{equation*}
	\|\widehat F_3(\cdot,t)\|_{\infty}
	\le\|F_3(\cdot,t)\|_{L_x^1(\mathbb{T}_\lambda)}\le 
	\lambda^{1/2}
	\| F_3(\cdot,t)\|_{L^2_{x}(\mathbb{T}_\lambda)},
\end{equation*}
it then follows from (\ref{1016-6}) that 
\[
	\text{LHS of (\ref{1015-8})} \lesssim \lambda^{-\frac32} \int_{\mathbb{R}}  \prod_{i=1}^{5} \|F_i(\cdot,t)\|_{L^2(\mathbb{T}_\lambda)} \,dt.
\]
Applying the H\"older inequality in the time variable yields 
\begin{equation}\label{1016-7}
	\text{LHS of (\ref{1015-8})} \lesssim \lambda^{-\frac32} 
	\|F_1\|_{L^{2,2}_{x,t}} \prod_{i=2}^{5}	\|F_i\|_{L^{2,8}_{x,t}}.
\end{equation}
Finally, according to the definition (\ref{1015-7}) for $F_i$, we can apply Lemma \ref{lemma3} to (\ref{1016-7}) to find that 
\[
	\text{LHS of (\ref{1015-8})} \lesssim \lambda^{-\frac32} \prod_{i=1}^{5}\|f_i\|_{L^{2,2}_{x,t}} = \text{RHS of (\ref{1015-8})},
\]
which justifies (\ref{1015-8}). Hence, the proof for \eqref{eq:M54-est} is finished.

Combining all the results through Section \ref{Subsec, quin_est1} -- Section \ref{Subsec, M54proof} establishes Lemma \ref{lem:M5j-estimates}.

\section*{Acknowledgments}
X.Yang is supported by National Natural Science Foundation of China (No. 12401299), Natural Science Foundation of Jiangsu Province (No. BK20241260), Scientific Research Center of Applied Mathematics of Jiangsu Province (No. BK20233002). 

{\small
	\bibliographystyle{plain}
	\bibliography{reference.bib}
}

\bigskip

\thanks{(K. Wang) School of Mathematics, Southeast University, Nanjing, Jiangsu 211189, China} 

\thanks{Email: ke.wang.math@seu.edu.cn}

\medskip

\thanks{(X. Yang) School of Mathematics, Southeast University, Nanjing, Jiangsu 211189, China} 

\thanks{Email: xinyang@seu.edu.cn}

\end{document}